\pdfoutput=1

\documentclass{article}
\usepackage{excludeonly}

\usepackage[disable]{todonotes}
\usepackage{savesym}
\usepackage{scrextend}  % \begin{labeling}{longest} \item   ..  \end{labeling}
\usepackage{amsthm,amsmath}
\usepackage{amssymb,latexsym,graphicx}
\usepackage{amscd}
 \usepackage[all,cmtip]{xy}
\usepackage{tikz-cd}
  \usetikzlibrary{arrows}
\tikzset{
  commutative diagrams/.cd,
  arrow style=tikz
  }
\usepackage{tikz}

\usepackage{accents}
\usepackage{cite}
\usepackage{mathtools}
\usepackage{stmaryrd} % provides \llbracket (formal power series)

\usepackage{calligra}
\DeclareMathAlphabet{\mathcalligra}{T1}{calligra}{m}{n}
\DeclareMathAlphabet{\mathpzc}{OT1}{pzc}{m}{it}

\usepackage{hycolor}
\usepackage{xcolor}
\usepackage[
                       colorlinks=true,
                       linkcolor=black, %red,
                       citecolor=black, %green,
                       urlcolor=blue,
                     ]{hyperref}
\usepackage{soul} %provides strikethrough-command st{.. text ..}

\usepackage{stmaryrd}  % \interleave \biginterleave |||
\usepackage{stackengine} %%% provides Nequation - named and numbered equation
\usepackage{booktabs} % tables

\newtheorem{theoremABC}{Theorem}

\newtheorem{theorem}{Theorem}[section]

\newtheorem{corollary}[theorem]{Corollary}

\newtheorem{lemma}[theorem]{Lemma}
\newtheorem{proposition}[theorem]{Proposition}

\theoremstyle{definition}
\newtheorem{definition}[theorem]{Definition}

\newtheorem{remark}[theorem]{Remark}

\theoremstyle{remark}
\newtheorem*{notation}{Notation}
\newcommand{\B}{{\mathbb{B}}}

\newcommand{\N}{{\mathbb{N}}}

\newcommand{\R}{{\mathbb{R}}}
\renewcommand{\SS}{{\mathbb{S}}}

\newcommand{\Z}{{\mathbb{Z}}}
\newcommand{\Aa}{{\mathcal{A}}}   % connections
\newcommand{\Bb}{{\mathcal{B}}}
\newcommand{\Cc}{{\mathcal{C}}}   % configuration space

\newcommand{\Ee}{{\mathcal{E}}}
\newcommand{\Ff}{{\mathcal{F}}}
\newcommand{\Gg}{{\mathcal{G}}}   % gauge transformations

\newcommand{\Ii}{{\mathcal{I}}}

\newcommand{\Ll}{{\mathcal{L}}}   % Lagrangian planes
\newcommand{\Pp}{{\mathcal{P}}}

\newcommand{\Ss}{{\mathcal{S}}}
\newcommand{\Tt}{{\mathcal{T}}}
\newcommand{\Uu}{{\mathcal{U}}}
\newcommand{\Vv}{{\mathcal{V}}}
\newcommand{\Ww}{{\mathcal{W}}}

\newcommand{\coker}{{\rm coker\, }}  % cokernel
\newcommand{\id}{{\rm id}}                % identity
\newcommand{\Id}{{\rm Id}}
\newcommand{\supp}{{\rm supp\, }}         % support
\newcommand{\INDEX}{\mathop{\mathrm{index}}}     % (Fredholm) index
\newcommand{\cgraph}[1]{\Gamma_{\kern-.5ex{}#1}}     % contact graph map
\renewcommand{\Im}{{\rm Im}}       % imaginary part
\newcommand{\I}{{\rm I}}               % unit interval
\newcommand{\CAP}{\mathop{\cap}}           % cap in formulas: adjusted spaces
 
\newcommand{\norm}{{\rm norm}}

\newcommand{\eps}{{\varepsilon}}

\renewcommand{\O}{{\rm O}}     % \O killt \Oersted !!

\newcommand{\inner}[2]{\langle #1, #2\rangle}   % inner product < , >
\newcommand{\INNER}[2]{\left\langle #1, #2\right\rangle}

\def\Nablatop#1{\nabla^{#1}\kern-.5ex{}}

\def\NABLA#1{{\mathop{\nabla\kern-.5ex\lower1ex\hbox{$#1$}}}}
\def\Nabla#1{\nabla\kern-.5ex{}_{#1}}
\def\Tabla#1{\Tilde\nabla\kern-.5ex{}_{#1}}
\def\Babla#1{\widebar\nabla\kern-.5ex{}_{#1}}
\def\abs#1{\mathopen|#1\mathclose|}   
\def\Abs#1{\left|#1\right|}            
\def\norm#1{\mathopen\|#1\mathclose\|}
\def\Norm#1{\left\|#1\right\|}

\renewcommand{\Tilde}{\widetilde}

\newcommand{\p}{{\partial}}

\newcommand{\INTO}{\hookrightarrow}              % embedding
\newlength\eqshift
\renewcommand\theequation{\thesection.\arabic{equation}}
\let\savetheequation\theequation

\makeatletter
\renewcommand*\env@matrix[1][\arraystretch]{%
  \edef\arraystretch{#1}%
  \hskip -\arraycolsep
  \let\@ifnextchar\new@ifnextchar
  \array{*\c@MaxMatrixCols c}}
\makeatother

\makeatletter
\let\save@mathaccent\mathaccent
\newcommand*\if@single[3]{%
  \setbox0\hbox{${\mathaccent"0362{#1}}^H$}%
  \setbox2\hbox{${\mathaccent"0362{\kern0pt#1}}^H$}%
  \ifdim\ht0=\ht2 #3\else #2\fi
  }
\newcommand*\rel@kern[1]{\kern#1\dimexpr\macc@kerna}
\newcommand*\widebar[1]{\@ifnextchar^{{\wide@bar{#1}{0}}}{\wide@bar{#1}{1}}}
\newcommand*\wide@bar[2]{\if@single{#1}{\wide@bar@{#1}{#2}{1}}{\wide@bar@{#1}{#2}{2}}}
\newcommand*\wide@bar@[3]{%
  \begingroup
  \def\mathaccent##1##2{%
    \let\mathaccent\save@mathaccent
    \if#32 \let\macc@nucleus\first@char \fi
    \setbox\z@\hbox{$\macc@style{\macc@nucleus}_{}$}%
    \setbox\tw@\hbox{$\macc@style{\macc@nucleus}{}_{}$}%
    \dimen@\wd\tw@
    \advance\dimen@-\wd\z@
    \divide\dimen@ 3
    \@tempdima\wd\tw@
    \advance\@tempdima-\scriptspace
    \divide\@tempdima 10
    \advance\dimen@-\@tempdima
    \ifdim\dimen@>\z@ \dimen@0pt\fi
    \rel@kern{0.6}\kern-\dimen@
    \if#31
      \overline{\rel@kern{-0.6}\kern\dimen@\macc@nucleus\rel@kern{0.4}\kern\dimen@}%
      \advance\dimen@0.4\dimexpr\macc@kerna
      \let\final@kern#2%
      \ifdim\dimen@<\z@ \let\final@kern1\fi
      \if\final@kern1 \kern-\dimen@\fi
    \else
      \overline{\rel@kern{-0.6}\kern\dimen@#1}%
    \fi
  }%
  \macc@depth\@ne
  \let\math@bgroup\@empty \let\math@egroup\macc@set@skewchar
  \mathsurround\z@ \frozen@everymath{\mathgroup\macc@group\relax}%
  \macc@set@skewchar\relax
  \let\mathaccentV\macc@nested@a
  \if#31
    \macc@nested@a\relax111{#1}%
  \else
    \def\gobble@till@marker##1\endmarker{}%
    \futurelet\first@char\gobble@till@marker#1\endmarker
    \ifcat\noexpand\first@char A\else
      \def\first@char{}%
    \fi
    \macc@nested@a\relax111{\first@char}%
  \fi
  \endgroup
}
\makeatother

\def\Xint#1{\mathchoice
   {\XXint\displaystyle\textstyle{#1}}%
   {\XXint\textstyle\scriptstyle{#1}}%
   {\XXint\scriptstyle\scriptscriptstyle{#1}}%
   {\XXint\scriptscriptstyle\scriptscriptstyle{#1}}%
   \!\int}
\def\XXint#1#2#3{{\setbox0=\hbox{$#1{#2#3}{\int}$}
     \vcenter{\hbox{$#2#3$}}\kern-.5\wd0}}

\def\dashint{\Xint-}

\long\def\symbolfootnote[#1]#2{\begingroup%
\def\thefootnote{\fnsymbol{footnote}}\footnote[#1]{#2}\endgroup}
\tikzset{
  symbol/.style={
    draw=none,
    every to/.append style={
      edge node={node [sloped, allow upside down, auto=false]{$#1$}}}
  }
}

\newcommand{\VD}{D}          % vertical differential

\begin{document}
\sloppy

\author{\quad Urs Frauenfelder \quad \qquad\qquad
             Joa Weber\footnote{
  Email: urs.frauenfelder@math.uni-augsburg.de
  \hfill
  joa@unicamp.br
  }
    \\
    Universit\"at Augsburg \qquad\qquad%\quad
    UNICAMP
}

\title{Hilbert manifold structures on path spaces}
\date{\today}

%\begin{titlepage}
\maketitle %(to set the title page and copyright page; see note)
                 %\include files (e.g., preface, introduction)
%\end{titlepage}

%%%%%%%%%%%%%%%%%%%%%%%%%%%%%%%%%%%
%%%%%%% Abstract %%%%%%%%%%%%%%%%%%%%%
%%%%%%%%%%%%%%%%%%%%%%%%%%%%%%%%%%%
\begin{abstract}
In Floer theory one has to deal with two-level manifolds
like for instance the space of $W^{2,2}$ loops and the space of
$W^{1,2}$ loops. Gradient flow lines in Floer theory are then trajectories
in a two-level manifold. Inspired by our endeavor to find a general
setup to construct Floer homology we therefore address in this paper
the question if the space of paths on a two-level manifold has itself
the structure of a Hilbert manifold.
In view of the two topologies on a two-level manifold
it is unclear how to define the exponential map on a general two-level
manifold.
We therefore study a different approach how to define charts on path
spaces of two-level manifolds.
To make this approach work we need an additional structure on a
two-level manifold which we refer to as \emph{tameness}.
We introduce the notion of tame maps and show that the composition of
tame is tame again.
Therefore it makes sense to introduce the notion of a tame two-level
manifold.
The main result of this paper shows that the path spaces on tame
two-level manifolds have the structure of a Hilbert manifold.
\end{abstract}

%\newpage
\tableofcontents

\newpage %.\newpage
%%%%%%%%%%%%%%%%%%%%%%%%%%%%%%%%%%%
%%%%%%%%%%%%%%%%%%%%%%%%%%%%%%%%%%%
%%%%%%% Section  %%%%%%%%%%%%%%%%%%%%%%
%%%%%%%%%%%%%%%%%%%%%%%%%%%%%%%%%%%
%%%%%%%%%%%%%%%%%%%%%%%%%%%%%%%%%%%
\section{Introduction}

Suppose that $H_2\subset H_1$ are Hilbert spaces
with a dense and compact inclusion of $H_2$ into $H_1$.
The fact that a dense and compact inclusion exists immediately implies
that both Hilbert spaces are separable, see~\cite[Cor.\,A.5]{Frauenfelder:2024c}.
It follows that there are two cases, either $H_1=H_2$ is finite dimensional,
or $H_1\not=H_2$ and both Hilbert spaces are infinite dimensional and separable.

By considering $C^2$ diffeomorphisms between open subsets
of $H_1$ which restrict to $C^2$ diffeomorphisms
between the intersection of these subsets with $H_2$
we obtain charts of $C^2$ \textbf{two-level manifolds} 
$X=X_1\supset X_2$.

For two points $x_-$ and $x_+$ we consider
$W^{1,2}(\R,H_1)\cap L^2(\R,H_2)$-paths from $x_-$ to $x_+$.
By $\Pp_{x_-x_+}$ we denote the set of such paths.
Over this space of paths there is a bundle, called the
\textbf{weak tangent bundle}, namely
$$
   \Ee\to\Pp
   ,\qquad
   \Ee=\Ee_{x_-x_+}
   ,\quad
   \Pp=\Pp_{x_-x_+} ,
$$
where the fiber over a path $x\in\Pp$
consists of $L^2(\R,H_1)$-vector fields along $x$.

In this article we are interested in the question
whether $\Pp$ and $\Ee$ have the structure of a $C^1$ Hilbert manifold.
We do not require that the tangent bundle $T\Pp$ is a
$C^1$ Hilbert manifold, but only its $L^2$ completion $\Ee$.
We don't know if this is true for general two-level manifolds.
However, we introduce a condition for $C^2$ diffeomorphisms
called \emph{tameness}. We then show that the composition of two tame
$C^2$ diffeomorphisms is tame as well (Theorem~\ref{thm:tame}).
Therefore by considering
atlases all whose transition functions are tame
we obtain the notion of a \emph{tame two-level manifold}.

Under these assumption we are able to show that the
space of paths $\Pp$ and its $L^2$ completion $\Ee$ of its tangent bundle
are $C^1$ Hilbert manifolds. Namely, the main result of this article
is the following theorem.

\begin{theoremABC}\label{thm:A}
Let $X$ be a tame two-level manifold.
Then for every pair of points $x_\pm\in X$ the space of paths $\Pp_{x_-x_+}$ 
and the $L^2$ completion $\Ee_{x_-x_+}$ of its tangent bundle
have the structure of $C^1$ Hilbert manifolds.
\end{theoremABC}

The main technical result to prove Theorem~\ref{thm:A}
is the following theorem on Hilbert space valued Sobolev spaces.
We abbreviate
\begin{equation}\label{eq:notation-spaces}
   W^{1,2}_{H_1}:=W^{1,2}(\R,H_1),\qquad
   L^2_{H_i}:=L^2(\R,H_i).
\end{equation}

\begin{theoremABC}\label{thm:B}
Suppose that $\phi\colon H_1\to H_1$ is tame.
Then the map 
\begin{equation*}
\begin{split}
   T\Phi
   \colon
   (W^{1,2}_{H_1}\cap L^2_{H_2})\times L^2_{H_1}
   &\to
   (W^{1,2}_{H_1}\cap L^2_{H_2})\times L^2_{H_1}
   \\
   (\xi,\eta)
   &\mapsto
   \left(\phi\circ\xi,d\phi|_\xi \eta \right)
\end{split}
\end{equation*}
is well defined and continuously differentiable.
\end{theoremABC}

In order to prove Theorem~\ref{thm:A}
we also will need a parametrized version of Theorem~\ref{thm:B}
in which $\phi$ is additionally allowed to depend on time.

Different from other approaches to manifold structures on mapping
spaces between finite dimensional manifolds,
see e.g.~\cite{Eliasson:1967a,schwarz:1993a},
we do not use an exponential map in the infinite dimensional case
since such seems difficult to incorporate
in the case where one has several levels of manifolds.

\begin{remark}[Dense inclusion, not necessarily compact]\mbox{}
\begin{itemize}\setlength\itemsep{0ex} 
\item[a)]
If we only assume that $H_1$ and $H_2$ are Hilbert spaces
with a dense inclusion of $H_2$ into $H_1$, not necessarily compact,
then we do not know if Theorem~\ref{thm:A} holds true.
\item[b)]
However, for the notion of tameness we do not need that the inclusion
is compact and Theorem~\ref{thm:B} continues to hold in this more
general case.
\item[c)]
For Hilbert manifolds a complementary approach to our proof
could probably be carried out with the help of the exponential map.
The exponential map for $C^3$ Hilbert manifolds
is studied in Lang's book~\cite[IV \S3]{lang:2001a}.
\end{itemize}
\end{remark}

\medskip\noindent
{\bf Acknowledgements.}
UF~acknowledges support by DFG grant
FR~2637/4-1.

%%%%%%%%%%%%%%%%%%%%%%%%%%%%%%%%%%%
%%%%%%% Subsection  %%%%%%%%%%%%%%%%%%%
%%%%%%%%%%%%%%%%%%%%%%%%%%%%%%%%%%%
\subsection{Outlook}

The Hilbert manifold structure of the space of maps $\Pp$
and the bundle $\Ee$ play an important role in Floer theory.
There one considers a functional $\Aa$ on special types of two-level
manifolds $X$ and is interested in the gradient flow equation
$$
   \Ff(u):=\p_s u+\Nabla{}\Aa(u)=0
$$
between critical points $x_-,x_+\in X_2$.
The gradient is unregularized in the sense of Floer~\cite{floer:1988c}.
Namely, if $x$ is in $X_2$, then $\Nabla{} \Aa(x)$ lies only in $T_x X_1$,
not in $T_x X_2$. We can interpret Floer gradient flow lines as the
zeroes of the section in the bundle
\begin{equation*}
\begin{tikzcd}[row sep=normal]
\Ee\arrow[d]
\\
\Pp
\arrow[u, bend right=50,"\Ff"']
\end{tikzcd}
\end{equation*}
If $\Ee$ and $\Pp$ have the structure of $C^1$ manifolds
we can look at the differential of the section at the point $u\in\Pp$,
in symbols
$$
   d\Ff_u\colon T_u \Pp \to T_{\Ff(u)}\Ee .
$$
If $u$ is a zero of the section, i.e. a gradient flow line,
we have a canonical splitting
$$
   T_u\Ee=T_u\Pp\oplus \Ee_u .
$$
We denote by $\pi_u\colon T_u\Ee\to\Ee_u$ the projection along $T_u\Pp$.
Hence we obtain the vertical differential
$$
   \VD\Ff_u:=\pi_u\circ d\Ff_u\colon T_u\Pp\to \Ee_u .
$$
A crucial property in a Floer theory is that this vertical
differential is a Fredholm operator.
In view of our endeavor to construct abstract Floer homologies
which can also be applied to quite general Hamiltonian delay
equations, as outlined in our 
articles~\cite{Frauenfelder:2024c,Frauenfelder:2025a},
the property that the space of paths and the $L^2$ bundle are both
Hilbert manifolds \underline{of class $C^1$} seems to be crucial
and this is the reason for the current study.

%%%%%%%%%%%%%%%%%%%%%%%%%%%%%%%%%%%
%%%%%%% Subsection  %%%%%%%%%%%%%%%%%%%
%%%%%%%%%%%%%%%%%%%%%%%%%%%%%%%%%%%
\subsection{Outline}

This text is organized as follows.

\smallskip
In Section~\ref{sec:tameness}  we introduce the notion of
\emph{tame maps}.
The main result about tame maps is that the composition is again tame.
This the content of Theorem~\ref{thm:tame}.
Thanks to this result one can introduce the notion of
\emph{tame manifolds}. Namely, these are manifolds all of whose
transitions maps are tame.
An example of a tame manifold is the loop space of a smooth
finite dimensional manifold.
We also introduce the notion of \emph{parametrized tameness}
which will be used later to show that the path spaces
of tame manifolds are Hilbert manifolds.

\smallskip
In Section~\ref{sec:differentiability} we prove a purely analytical
result. Namely, we show in Theorem~\ref{thm:B-Phi}
that composing a class of Hilbert space valued
Sobolev maps with a tame map gives rise to a $C^1$ map.
This result, in particular, implies
that for constant paths in a tame manifold
the change of chart gives rise to a $C^1$ diffeomorphism.
To extend this result to non-constant paths
we describe as well a parametrized version of Theorem~\ref{thm:B-Phi},
namely Theorem~\ref{thm:B-Phi-parametrized}.
\\
We further show in this section that the \underline{weak} differentials
of the above maps are themselves $C^1$ maps
-- which would not hold for the differentials.
This then leads to the proof of Theorem~\ref{thm:B}
and its parametrized version Theorem~\ref{thm:B-TPhi-parametrized}.

\smallskip
In Section~\ref{sec:coordinate-charts}
we prove that the space of paths of a tame manifold,
as well as its weak tangent bundle, are Hilbert manifolds.
For these purposes we construct charts
for the space of paths of a tame manifold
and show that their transition maps are $C^1$ diffeomorphisms.
In case that the path can be covered by one single chart in the tame
manifold the proof of this is a rather straightforward application of
the parametrized version of Theorem~\ref{thm:B}.
In the case that the path cannot be covered by a single chart
the argument gets quite subtle and in particular
uses the compact inclusion of $H_2$ into $H_1$.
Since it is not clear how to implement an exponential map
on a two-level manifold,
our construction is based on an interpolation procedure.
The intricate part is to show that this interpolation procedure
preserves tameness.

\smallskip
In Appendix~\ref{sec:Hmf-of-paths} we provide the reader with
a self-contained construction of Hilbert space valued Sobolev spaces
used in this work.
The (automatic) fact that our Hilbert spaces are separable simplifies 
various parts of this construction.
In particular, we give a detailed proof of
Pettis' Theorem in the separable case.
We also show how Hilbert space valued Sobolev maps
can be approximated by differentiable Hilbert space valued maps.

\smallskip
In Appendix~\ref{sec:IFT} we recall a quantitative version of the
implicit function theorem which plays an important role in the
interpolation procedure to produce our charts.

\smallskip
In Appendix~\ref{sec:Hmf-of-paths} we recall that for
finite dimensional manifolds the spaces of paths can be endowed
with the structure of Hilbert manifolds using the exponential map.
This is to illustrate the difference between exponential map and
our construction based on interpolation.

\begin{notation}
We often write $\xi_s$ instead of $\xi(s)$ and $df|_x$ instead of
$df(x)$ for better readability, avoiding excessive use of parentheses.
For Hilbert space norms we often write $\abs{\cdot}$
for distinction of the norm $\norm{\cdot}$ in mapping spaces.
\end{notation}

%\newpage
%%%%%%%%%%%%%%%%%%%%%%%%%%%%%%%%%%%
%%%%%%%%%%%%%%%%%%%%%%%%%%%%%%%%%%%
%%%%%%% Section  %%%%%%%%%%%%%%%%%%%%%%
%%%%%%%%%%%%%%%%%%%%%%%%%%%%%%%%%%%
%%%%%%%%%%%%%%%%%%%%%%%%%%%%%%%%%%%
\section{Tameness}
\label{sec:tameness}

Assume that $H_2\subset H_1$ are two separable Hilbert spaces
with a dense inclusion of $H_2$ into $H_1$. 
To avoid annoying constants we assume that
$\abs{\cdot}_1\le \abs{\cdot}_2$.
We abbreviate the norms in $H_1$ and $H_2$ by
\begin{equation}\label{eq:12}
   \abs{\cdot}_1:=\norm{\cdot}_{H_1}
   ,\qquad
   \abs{\cdot}_2:=\norm{\cdot}_{H_2}
   ,\qquad
   \abs{\cdot}_1\le \abs{\cdot}_2 .
\end{equation}

%%%%%%%%%%%%%%%%%%%%%%%%%%%%%%%%%%%
%%%%%%% Subsection  %%%%%%%%%%%%%%%%%%%
%%%%%%%%%%%%%%%%%%%%%%%%%%%%%%%%%%%
\subsection{Tameness (unparametrized)}
\label{sec:tameness-unparametrized}

Let $U_1$ be an open subset
of $H_1$ and let $U_2:=U_1\cap H_2$ be its part in $H_2$.

\begin{definition}\label{def:tame}
A $C^2$ function $\phi\colon U_1\to H_1$ is called
\textbf{tame}, or more precisely \textbf{\boldmath $(H_1,H_2)$-tame},
if if it satisfies two conditions.
(i)~Its restriction to $U_2$ takes values in $H_2$
and is $C^2$ as a map $\phi|_{U_2} \colon U_2 \to H_2$.
(ii)~For every $x\in U_1$ there exists an $H_1$-open neighborhood
$W_x\subset U_1$ of $x$ and a constant $\kappa=\kappa(x)>0$
such that for every $y\in W_x\cap H_2$ and all $\xi,\eta\in H_2$
we have the estimate
\begin{equation}\label{eq:tame}
   \Abs{d^2 \phi|_y(\xi,\eta)}_2 
   \le \kappa\bigl(\abs{\xi}_1 \abs{\eta}_2+\abs{\xi}_2\abs{\eta}_1
   +\abs{y}_2 \abs{\xi}_1 \abs{\eta}_1\bigr) .
\end{equation}
\end{definition}

\begin{remark}[Vector space]\label{rem:tame-v.sp.}
Assume that $\phi,\psi\colon U_1\to H_1$ are both tame.
Then their sum is tame as well.
Indeed for $x\in U_1$ take the intersection of the two neighborhoods
and the sum of their constants.
\end{remark}

\begin{remark}[Constant scales $H_2=H_1=:H$ are tame]
\label{rem:tame}
We claim that in this special case
every $C^2$-function $\phi\colon U\to H$ is automatically tame.
To see this pick $x\in U$. Since $\phi$ is two times continuously
differentiable there exists an open neighborhood $W_x$ of $x$ in $U$
such that for every $y\in W_x$ the following holds
$$
   \Norm{d^2\phi|_y}_{\Ll(H,H;H)}
   \le \Norm{d^2\phi|_x}_{\Ll(H,H;H)}+1 .
$$
For $y\in W_x$ and $\xi,\eta\in H$ we estimate
\begin{equation*}
\begin{split}
   \Abs{d^2\phi|_y(\xi,\eta)}
   &\le \Norm{d^2\phi|_y}_{\Ll(H,H;H)}\abs{\xi}\abs{\eta}\\
   &\le\underbrace{\left(\Norm{d^2\phi|_x}_{\Ll(H,H;H)}+1\right)}_{=:\kappa}
   \abs{\xi}\abs{\eta}\\
   &\le \kappa
   \left(\abs{\xi}\abs{\eta} +\abs{\xi}\abs{\eta}
     +\abs{y}\abs{\xi}\abs{\eta}\right)
\end{split}
\end{equation*}
and this confirms the taming condition~(\ref{eq:tame}).
\end{remark}

\begin{remark}
In the definition of tame the condition that the restriction
$\phi|_{U_2}\colon U_2\to H_2$ is of class $C^2$ is not used,
neither in Theorem~\ref{thm:tame}, nor in Theorem~\ref{thm:B}.
In both theorems one gets away with $C^1$.
\end{remark}

The next theorem tells that tameness is invariant under
coordinate change.

\begin{theorem}\label{thm:tame}
Given open subsets $U_1,V_1\subset H_1$,
consider $(H_1,H_2)$-tame maps
$$
   \phi\colon U_1\to V_1
   ,\qquad
   \psi\colon V_1\to H_1 .
$$
Then the composition $\psi\circ\phi\colon U_1\to H_1$ is
$(H_1,H_2)$-tame as well.
\end{theorem}

The proof needs the following two technical lemmas.

\begin{lemma}\label{le:tame-1}
Assume that $\phi\colon U_1\to H_1$ is tame.
Then for every $x\in U_1$ there exists an $H_1$-open neighborhood
$W_x^\prime\subset U_1$ of $x$ and a constant $\kappa_1>0$
such that for every $y\in W_x^\prime\cap H_2$ and every $\eta\in H_2$
it holds the estimate
\begin{equation*}
\begin{split}
   \Abs{d\phi|_y\eta}_2
   &\le
   \kappa_1\left(\abs{\eta}_2+\abs{y}_2\abs{\eta}_1\right) .
\end{split}
\end{equation*}
\end{lemma}

\begin{proof}
For $x\in U_1$ let $W_x$ and $\kappa>0$ be such that
for every $y\in W_x\cap H_2$ the estimate~(\ref{eq:tame}) holds.
Since $H_2$ is dense in $H_1$
we can pick $x_0$ in $W_x\cap H_2$ and $\delta>0$ such that
the radius-$\delta$ ball in $H_1$ about $x_0$ is a subset of $W_x$
containing~$x$, in symbols $x\in B^1_\delta(x_0)\subset W_x$.
Pick $y\in B^1_\delta(x_0)\cap H_2$ and $\eta\in H_2$, then
by the fundamental theorem of calculus and monotonicity of the
integral we obtain
\begin{equation*}
\begin{split}
   &\Abs{d\phi|_y\eta}_2
\\
   &\le 
   \int_0^1\Abs{\tfrac{d}{dt} d\phi|_{x_0+t(y-x_0)}\eta}_2\, dt
   +\Abs{d\phi|_{x_0}\eta}_2
\\
   &=
   \int_0^1\Abs{d^2\phi|_{x_0+t(y-x_0)}(y-x_0,\eta)}_2\, dt
   +\Abs{d\phi|_{x_0}\eta}_2
\\
   &\stackrel{3}{\le}
   \kappa\int_0^1\Bigl(\abs{y-x_0}_1\abs{\eta}_2
   +\abs{y-x_0}_2\abs{\eta}_1
   +\underbrace{\abs{x_0+t(y-x_0)}_2}_{\le\abs{x_0}_2+t\abs{y-x_0}_2}
      \abs{y-x_0}_1\abs{\eta}_1
   \Bigr) dt
   \\
   &\quad+\Norm{d\phi|_{x_0}}_{\Ll(H_2)}\abs{\eta}_2
\\
   &\stackrel{4}{\le}
   \kappa\Bigl(\delta\abs{\eta}_2+\abs{y-x_0}_2\abs{\eta}_1
   +\delta\abs{x_0}_2\abs{\eta}_1+\tfrac12 \delta \abs{y-x_0}_2\abs{\eta}_1
   \Bigr)
   \\
   &\quad+\Norm{d\phi|_{x_0}}_{\Ll(H_2)}\abs{\eta}_2
\\
   &\stackrel{5}{\le}
   \kappa\Bigl(\delta\abs{\eta}_2
   +\abs{y}_2\abs{\eta}_1 +\abs{x_0}_2\abs{\eta}_1
   +\tfrac{3}{2}\delta \abs{x_0}_2\abs{\eta}_1
   +\tfrac12\delta \abs{y}_2\abs{\eta}_1
   \Bigr)
   \\
   &\quad+\Norm{d\phi|_{x_0}}_{\Ll(H_2)}\abs{\eta}_2
\\
   &\stackrel{6}{\le}
   \kappa\Bigl(\delta\abs{\eta}_2
   +\abs{y}_2\abs{\eta}_1
   +\abs{x_0}_2\abs{\eta}_2
   +\tfrac{3}{2}\delta \abs{x_0}_2\abs{\eta}_2
   +\tfrac12\delta \abs{y}_2\abs{\eta}_1
   \Bigr)
   \\
   &\quad+\Norm{d\phi|_{x_0}}_{\Ll(H_2)}\abs{\eta}_2
\\
   &\stackrel{7}{\le}
   \kappa_1\left(\abs{\eta}_2+\abs{y}_2\abs{\eta}_1\right) .
\end{split}
\end{equation*}
Step~3 is by tameness~(\ref{eq:tame}),
step~4 by integration and since $y\in B^1_\delta(x_0)$.
Step~5 is the triangle inequality and
step~6 the constant-avoiding assumption
$\abs{\cdot}_1\le \abs{\cdot}_2$.
In step 7 we used the constant defined by
$$
   \kappa_1:=\max\Bigl\{\kappa \left(\delta+(1+\tfrac{3}{2}\delta)\abs{x_0}_2\right)
   +\Norm{d\phi|_{x_0}}_{\Ll(H_2)}, \kappa\left(1+\tfrac12\delta\right)\Bigr\}.
$$
Hence Lemma~\ref{le:tame-1} follows with $W_x^\prime=B_\delta^1(x_0)$.
\end{proof}

\begin{lemma}\label{le:tame-0}
Assume that $\phi\colon U_1\to H_1$ is tame.
Then for every $x\in U_1$ there exists an $H_1$-open neighborhood
$W_x^{\prime\prime}\subset U_1$ of $x$ and a constant $\kappa_0>0$
such that for every $y\in W_x^{\prime\prime}\cap H_2$
it holds the estimate
\begin{equation*}
\begin{split}
   \abs{\phi(y)}_2
   &\le
   \kappa_0\left(1+\abs{y}_2\right) .
\end{split}
\end{equation*}
\end{lemma}

\begin{proof}
For $x\in U_1$ let $W_x^\prime$ and $\kappa_1>0$ be as in
Lemma~\ref{le:tame-1}.
Since $H_2$ is dense in $H_1$ we can pick $x_0$ in $W_x^\prime\cap H_2$
and $\delta>0$ such that
the radius-$\delta$ ball in $H_1$ about $x_0$ is a subset of $W_x^\prime$
containing~$x$, in symbols $x\in B^1_\delta(x_0)\subset W_x^\prime$.
Pick $y\in B^1_\delta(x_0)\cap H_2$, then
by the fundamental theorem of calculus we calculate
\begin{equation}\label{eq:kjhghgfc6hh}
\begin{split}
   &\abs{\phi(y)}_2
\\
   &\le 
   \abs{\phi(x_0)}_2
   +\int_0^1\Abs{\tfrac{d}{dt} \phi(x_0+t(y-x_0))}_2\, dt
\\
   &=
   \abs{\phi(x_0)}_2
   +\int_0^1\Abs{d\phi|_{x_0+t(y-x_0)}(y-x_0)}_2\, dt
\\
   &\stackrel{3}{\le}
   \abs{\phi(x_0)}_2
   +\kappa_1\int_0^1\Bigl(\abs{y-x_0}_2
   +\underbrace{\abs{x_0+t(y-x_0)}_2}_{\le\abs{x_0}_2+t\abs{y-x_0}_2}
      \abs{y-x_0}_1
   \Bigr) dt
\\
   &\stackrel{4}{\le}
   \abs{\phi(x_0)}_2
   +\kappa_1\Bigl(1+\tfrac12\abs{y-x_0}_1\Bigr)
   \underbrace{\abs{y-x_0}_2}_{\le \abs{y}_2+\abs{x_0}_2}
   +\kappa_1\abs{x_0}_2\abs{y-x_0}_1
\\
   &\stackrel{5}{\le}
   \abs{\phi(x_0)}_2
   +\kappa_1\abs{x_0}_2\left(1+\tfrac{3}{2}\delta\right)
   +\kappa_1\left(1+\tfrac{1}{2}\delta\right)\abs{y}_2
   .
\end{split}
\end{equation}
Step~3 is by Lemma~\ref{le:tame-1}, step 4 by integration,
step 5 by $\abs{y-x_0}_1 \le\delta$. Define
$$
   \kappa_0
   :=
   \max\Bigl\{\abs{\phi(x_0)}_2
   +\kappa_1\abs{x_0}_2\left(1+\tfrac{3}{2}\delta\right)
   ,
   \kappa_1\left(1+\tfrac12\delta\right)\Bigr\} .
$$
Hence Lemma~\ref{le:tame-0} follows with
$W_x^{\prime\prime}=B_\delta^1(x_0)$.
\end{proof}

We summarize the two lemmas and~(\ref{eq:tame}) in the following corollary.

\begin{corollary}\label{cor:tame}
Assume that $\phi\colon U_1\to H_1$ is tame.
Then for every $x\in U_1$ there exists an $H_1$-open neighborhood
$W_x\subset U_1$ of $x$ and a constant $\kappa>0$
such that for every $y\in W_x\cap H_2$
and all $\xi,\eta\in H_2$
there are the three estimates
\begin{equation}\label{eq:cor-tame}
\begin{split}
   \Abs{\phi(y)}_2
   &\le
   \kappa\left(1+\abs{y}_2\right)
\\
   \Abs{d\phi|_y\eta}_2
   &\le
   \kappa\left(\abs{\eta}_2+\abs{y}_2\abs{\eta}_1\right)
\\
   \Abs{d^2 \phi|_y(\xi,\eta)}_2 
   &\le \kappa\bigl(\abs{\xi}_1 \abs{\eta}_2+\abs{\xi}_2\abs{\eta}_1
   +\abs{y}_2 \abs{\xi}_1 \abs{\eta}_1\bigr)
   .
\end{split}
\end{equation}
\end{corollary}

\begin{proof}[Proof of Theorem~\ref{thm:tame}]
Assume that $x\in U_1$.
Concerning $\phi$ pick an open neighborhood $W_x$ of $x$ in $U_1$
such that all three estimates in Corollary~\ref{cor:tame} hold true.
Since $\phi\colon U_1\to V_1$ is $C^2$ we can additionally assume
that for every $y\in W_x$ and all $\xi,\eta\in H_1$
the following three estimates hold true as well
\begin{equation}\label{eq:add-tame}
\begin{split}
   \Abs{\phi(y)}_1
   &\le
   \kappa
\\
   \Abs{d\phi|_y\eta}_1
   &\le
   \kappa\abs{\eta}_1
\\
   \Abs{d^2 \phi|_y(\xi,\eta)}_1
   &\le \kappa \abs{\xi}_1 \abs{\eta}_1 .
\end{split}
\end{equation}
Concerning $\psi\colon V_1\to H_1$
pick an open neighborhood $W_{\phi(x)}$ of $\phi(x)$
in $V_1$ such that, maybe after enlarging the constant $\kappa$,
the six estimates hold as well for $\phi$ replaced by $\psi$.
Maybe after shrinking the open neighborhood $W_x$ we can additionally
assume that $W_x\subset \phi^{-1}(W_{\phi(x)})$.

Using that $d(\psi\circ\phi)|_y\xi=d\psi|_{\phi(y)} d\phi|_y \xi$ we calculate
\begin{equation*}
\begin{split}
   &\Abs{d^2(\psi\circ\phi)|_y(\xi,\eta)}_2
\\
   &\le\Abs{
   d^2\psi|_{\phi(y)}\left(d\phi|_y\xi,d\phi|_y\eta\right)}_2
   +\Abs{d\psi|_{\phi(y)} d^2\phi|_y (\xi,\eta)}_2
\\
   &\stackrel{2}{\le}\kappa\left(
      \Abs{d\phi|_y\xi}_1 \Abs{d\phi|_y\eta}_2
      +\Abs{d\phi|_y\xi}_2\Abs{d\phi|_y\eta}_1
      +\abs{\phi(y)}_2 \Abs{d\phi|_y\xi}_1 \Abs{d\phi|_y\eta}_1
   \right)
   \\
   &\quad
   +\kappa\left(
   \Abs{d^2\phi|_y (\xi,\eta)}_2+\abs{\phi(y)}_2\Abs{d^2\phi|_y (\xi,\eta)}_1
   \right)
\\
   &\stackrel{3}{\le}
   \kappa
   \Bigl(
   \kappa\abs{\xi}_1\kappa\left(\abs{\eta}_2+\abs{y}_2\abs{\eta}_1\right)
   +\kappa\abs{\eta}_1\kappa\left(\abs{\xi}_2+\abs{y}_2\abs{\xi}_1\right)
   \Bigr)
   \\
   &\quad+
   \kappa (1+\abs{y}_2)\kappa^2\abs{\xi}_1\abs{\eta}_1
   +\kappa^2
   \bigl(\abs{\xi}_1 \abs{\eta}_2+\abs{\xi}_2\abs{\eta}_1
   +\abs{y}_2 \abs{\xi}_1 \abs{\eta}_1\bigr)
   \\
   &\quad+\kappa (1+\abs{y}_2) \kappa \abs{\xi}_1\abs{\eta}_1
\\
   &=
   (\kappa^3+\kappa^2)\abs{\xi}_1\abs{\eta}_2
   +(\kappa^3+\kappa^2)\abs{\xi}_2\abs{\eta}_1
   \\
   &\quad
   +\underbrace{(3\kappa^3+2\kappa^2)}_{=:\kappa^\prime}
   \abs{y}_2 \abs{\xi}_1\abs{\eta}_1
   +(\kappa^3+\kappa^2)\abs{\xi}_1\underbrace{\abs{\eta}_1}_{\le \abs{\eta}_2}
\\
   &\le 
   \kappa^\prime \bigl(\abs{\xi}_1 \abs{\eta}_2+\abs{\xi}_2\abs{\eta}_1
   +\abs{y}_2 \abs{\xi}_1 \abs{\eta}_1\bigr) .
\end{split}
\end{equation*}
In step 2 we used~(\ref{eq:cor-tame}) for $\psi$, namely for summand one
the third estimate and for summand two the second estimate.
In step 3 we used~(\ref{eq:cor-tame}) and~(\ref{eq:add-tame}) for $\phi$.
This finishes the proof of Theorem~\ref{thm:tame}.
\end{proof}

For later reference we state the following lemma.

\begin{lemma}[Differences]\label{le:tame-3}
Assume that $\phi\colon U_1\to H_1$ is tame.
Then for every $x\in U_1$ there exists an $H_1$-open neighborhood
$W_x\subset U_1$ of $x$ and a constant $\kappa>0$
such that for all $y_0,y_1\in W_x\cap H_2$ and $\eta\in H_2$
there is the estimate
\begin{equation}\label{eq:tame-3}
\begin{split}
   &\Abs{\left( d\phi|_{y_1}-d\phi|_{y_0}\right)\eta}_2\\
   &\le\kappa\left(\abs{y_1-y_0}_1\abs{\eta}_2+\abs{y_1-y_0}_2\abs{\eta}_1\right)
   +\tfrac{\kappa}{2}\left(\abs{y_1}_2+\abs{y_0}_2\right)
      \abs{y_1-y_0}_1\abs{\eta}_1.
\end{split}
\end{equation}
\end{lemma}

\begin{proof}
For $x\in U_1$ let $W_x$ and $\kappa$ be as in Definition~\ref{def:tame}.
Without loss of generality we can additionally assume that $W_x$ is convex.
Hence we estimate using the fundamental theorem of calculus
and~(\ref{eq:tame}) in~step~3 to get
\begin{equation*}
\begin{split}
   &\Abs{\left( d\phi|_{y_1}-d\phi|_{y_0}\right)\eta}_2
\\
   &=\biggl|\int_0^1 \tfrac{d}{dt} d\phi|_{ty_1+(1-t)y_0} \eta \, dt\biggr|_2
\\
   &\le\int_0^1\Bigl|d^2\phi|_{ty_1+(1-t)y_0}\left(y_1-y_0,\eta\right)\Bigr|_2 dt
\\
   &\stackrel{3}{\le}\kappa\int_0^1
   \left(\abs{y_1-y_0}_1\abs{\eta}_2+\abs{y_1-y_0}_2\abs{\eta}_1\right)
   +\left(t\abs{y_1}_2+(1-t)\abs{y_0}_2\right)
   \abs{y_1-y_0}_1\abs{\eta}_1\, dt .
\end{split}
\end{equation*}
Integration then concludes the proof.
\end{proof}

%\newpage %.\newpage
%%%%%%%%%%%%%%%%%%%%%%%%%%%%%%%%%%%
%%%%%%% Section  %%%%%%%%%%%%%%%%%%%%%%
%%%%%%%%%%%%%%%%%%%%%%%%%%%%%%%%%%%
\subsection{Tame two-level manifolds}
\label{sec:two-level-mfs}

\begin{definition}\label{def:two-level-manifold}
Let $H_2\subset H_1$ be separable Hilbert spaces
with a dense inclusion of $H_2$ into $H_1$.
Let $X=X_1$ be a $C^2$ Hilbert manifold modeled on~$H_1$.

\begin{itemize}\setlength\itemsep{0ex} 
\item[a)]
  A \textbf{tame two-level structure on $X_1$}, more precisely a \textbf{tame
  \boldmath$(H_1,H_2)$-two-level structure on $X_1$},
  is a $C^2$ atlas $\Aa=\{\psi\colon H_1\supset U_1\to V_1\subset X_1\}$
  of $X_1$ with the property that all transition maps
  $$
     \phi:=\tilde\psi^{-1}\circ\psi\colon
     H_1\supset \psi^{-1}(V_1\cap \tilde V_1)
     \to\tilde\psi^{-1}(V_1\cap \tilde V_1) \subset H_1
  $$
  are $(H_1,H_2)$-tame. 
  We refer to such an atlas as a \textbf{tame two-level atlas}.
\item[b)]
  For an open subset $U_1\subset H_1$ let $U_2=U_1\cap H_2$ be its part
  in $H_2$.
  If $\Aa$ is a tame two-level atlas of $X_1$
we define a subset $X_2\subset X_1$ by
$$
   X_2
   :=\bigcup_{\psi\in\Aa} \psi(U_2) .
$$
Then $X_2$ is itself a $C^2$ Hilbert manifold
with $C^2$ atlas $\Aa_2=\{\psi|_{U_2}\mid\psi \in\Aa\}$
and $X_2$ is dense in $X_1$.
\item[c)]
The pair $(X_1,X_2)$ is referred to as a
\textbf{tame two-level manifold}, more precisely as a 
\textbf{tame \boldmath$(H_1,H_2)$-two-level manifold}.
\end{itemize}
\end{definition}

\begin{remark}[Maximal tame two-level atlas]\label{rem:max-atlas}
Two tame two-level atlases are \textbf{compatible} if their union is itself a
tame two-level atlas.
Compatibility is an equivalence relation.
That compatibility is reflexive and symmetric is obvious.
That compatibility is transitive follows from Theorem~\ref{thm:tame}
which tells us that a composition of tame maps is tame,
together with the fact that tameness is a local condition
(the restriction of a tame map to an open subset of its domain is itself tame).

Since compatibility is an equivalence relation,
for each tame two-level manifold there exists a unique
\textbf{maximal} tame two-level atlas obtained by taking the union of all tame
two-level atlases compatible to a given one.
\end{remark}

%\newpage
%%%%%%%%%%%%%%%%%%%%%%%%%%%%%%%%%%%
%%%%%%% Subsection  %%%%%%%%%%%%%%%%%%%
%%%%%%%%%%%%%%%%%%%%%%%%%%%%%%%%%%%
\subsection{Parametrized tameness}\label{sec:tameness-parametrized}

Let $\O_1$ be an open subset of $\R\times H_1$
and let $O_2$ be its part in $\R\times H_2$.

\begin{definition}\label{def:tame-parametrized}
A $C^2$ map $\phi\colon O_1\to H_1$ is called
\textbf{parametrized tame}, if it satisfies two conditions.
Firstly, its restriction to $O_2=O_1\cap(\R\times H_2)$ takes values
in $H_2$ and is $C^2$ as a map $\phi|_{O_2} \colon O_2 \to H_2$.
Secondly, for every point $(s,x)\in O_1$ there exists an open neighborhood
$W_{s,x}$ of $(s,x)$ in $O_1$ and a constant $\kappa_{s,x}>0$ such that for
every $(t,y)\in W_{s,x}\cap (\R\times H_2)$ it holds the estimate
\begin{equation}\label{eq:tame-parametrized}
   \Abs{d^2 \phi_t|_y(\xi,\eta)}_2 
   \le \kappa_{s,x}\bigl(\abs{\xi}_1 \abs{\eta}_2+\abs{\xi}_2\abs{\eta}_1
   +\abs{y}_2 \abs{\xi}_1 \abs{\eta}_1\bigr) .
\end{equation}
\end{definition}

\begin{definition}\label{def:tame-asy-const}
The \textbf{\boldmath$s$-slices} of $O_1$ and $O_2$ are defined by
\begin{equation*}
\begin{split}
   U_1^s
   &:=\{x\in H_1\mid (s,x)\in O_1\}
   =O_1\cap (\{s\}\times H_1) ,
\\
   U_2^s
   &:=\{x\in H_2\mid (s,x)\in O_2\}
   =O_2\cap (\{s\}\times H_2) 
   =U_1^s\cap H_2 .
\end{split}
\end{equation*}
A parametrized tame map is called \textbf{asymptotically constant}
if there exists $T>0$ such that, firstly,
$\phi_s:=\phi(s,\cdot)\colon U_1^s\to H_1$
is independent of $s$ and, secondly, $\phi_s(0)=0$,
both whenever $\abs{s}\ge T$.
\end{definition}

%\newpage %.\newpage
%%%%%%%%%%%%%%%%%%%%%%%%%%%%%%%%%%%
%%%%%%% Subsection  %%%%%%%%%%%%%%%%%%%
%%%%%%%%%%%%%%%%%%%%%%%%%%%%%%%%%%%
\subsection{Loop space is a tame two-level manifold}
\label{sec:tameness-loop-space}

The loop space $X=\Lambda \R$ of $M=\R$ has the form of a Hilbert space pair
$$
   X_1:=H_1:=W^{1,2}(\SS^1,\R)
   ,\qquad
   X_2:=H_2:=W^{2,2}(\SS^1,\R) .
$$
A $C^\infty$ diffeomorphism $\varphi\colon\R\to\R$
induces the loop space \textbf{transition map}
$$
   \phi\colon U_1:=H_1\to H_1
   ,\quad
   x\mapsto [t\mapsto \varphi(x(t))].
$$

\begin{lemma}\label{le:loop-space-tame}
The induced loop space transition map $\phi$ is tame.
\end{lemma}

\begin{proof}
Condition~(i).
By assumption $\varphi\colon \R\to\R$ is $C^\infty$.
For $x\in H_1$ and $y\in H_2$ and abbreviating $x_t:=x(t)$ we get
$$
   \left(\tfrac{d}{dt} \phi(x)\right)_t
   =\varphi^\prime(x_t) \dot x_t
   ,\qquad
   \left(\tfrac{d^2}{dt^2} \phi(y)\right)_t
   =\varphi^{\prime\prime}(y_t) \dot y_t \dot y_t
   +\varphi^\prime(y_t) \ddot y_t
$$
for every $t\in\SS^1$.
This shows that $\phi$ on $H_1$ takes values in $H_1$
and on $H_2$ it takes values in $H_2$.
For $x,\xi,\eta\in H_1$ we calculate
\begin{equation*}
\begin{split}
   \left(d\phi|_x\xi\right)_t
   :=\left(\left.\tfrac{d}{d\eps}\right|_{\eps=0}
   \phi(x+\eps\xi)\right)_t
   =\left.\tfrac{d}{d\eps}\right|_{\eps=0} \phi(x_t+\eps\xi_t)
   =\varphi^\prime|_{x_t}\xi_t
\end{split}
\end{equation*}
whenever $t\in\SS^1$ and similarly
$$
   \left(d^2\phi|_x(\xi,\eta)\right)_t
   =\varphi^{\prime\prime}(x_t)\xi_t\eta_t .
$$
At $x\in H_1$ and $y\in H_2$
the first two derivatives with respect to time $t\in\SS^1$ are
\begin{equation*}
\begin{split}
   \tfrac{d}{dt}\left(d^2\phi|_x(\xi,\eta)\right)
   &=\varphi^{\prime\prime\prime}(x)\dot x\xi\eta
   +\varphi^{\prime\prime}(x)\dot\xi\eta
   +\varphi^{\prime\prime}(x)\xi\dot \eta
\\
   \tfrac{d^2}{dt^2}\left(d^2\phi|_y(\xi,\eta)\right)
   &=\varphi^{\prime\prime\prime\prime}(y)\dot y^2\xi\eta
   +\varphi^{\prime\prime\prime}(y)\ddot y\xi\eta
   +2\varphi^{\prime\prime\prime}(y)\dot y\dot\xi\eta
   +2\varphi^{\prime\prime\prime}(y)\dot y\xi\dot\eta
   \\
   &\quad
   +2\varphi^{\prime\prime}(y)\dot\xi\dot\eta
   +\varphi^{\prime\prime}(y)\ddot\xi\eta
   +\varphi^{\prime\prime}(y)\xi\ddot\eta .
\end{split}
\end{equation*}
The first identity holds for all $\xi,\eta\in H_1$
and the second for all $\xi,\eta\in H_2$.
Using the Sobolev embedding $W^{1,2}(\SS^1,\R)\INTO C^0(\SS^1,\R)$
with constant $1$ shows that both right hand sides are in $L^2(\SS^1,\R)$.
Thus $d^2\phi|_x(\xi,\eta)\in H_1$ and $d^2\phi|_y(\xi,\eta)\in H_2$.
Moreover, these formulae show that both maps
$$
   d^2\phi\colon H_1\to\Ll(H_1,H_1;H_1)
   ,\qquad
   d^2\phi\colon H_2\to\Ll(H_2,H_2;H_2) ,
$$
are continuous.
\\
Condition~(ii).
By the Sobolev embedding $W^{1,2}(\SS^1,\R)\INTO C^0(\SS^1,\R)$
with constant $1$ we have for each $y\in H_2$ the estimate
\begin{equation*}
\begin{split}
   \norm{d^2\phi|_y(\xi,\eta)}_{2,2}^2
   &=\norm{\varphi^{\prime\prime}(y)\xi\eta}_{2}^2
   +\norm{\tfrac{d}{dt}\varphi^{\prime\prime}(y)\xi\eta}_{2}^2
   +\norm{\tfrac{d^2}{dt^2}\varphi^{\prime\prime}(y)\xi\eta}_{2}^2
\\
   &\le c_y\Bigl(
   \norm{\xi}_{2,2} \norm{\eta}_{1,2}
   +\norm{\xi}_{1,2} \norm{\eta}_{2,2}
   +\norm{y}_{2,2}\norm{\xi}_{1,2} \norm{\eta}_{1,2}
   \Bigr)
\end{split}
\end{equation*}
for some constant $c_y$ which depends continuously on
$\abs{y}_1=\norm{y}_{1,2}$.
\\
This proves Lemma~\ref{le:loop-space-tame}.
\end{proof}

Lemma~\ref{le:loop-space-tame} shows that the loop space
is a tame two-level manifold.
Strictly speaking we only showed this for $M=\R$ to simply notation.
The same argument should, however, work for general smooth finite
dimensional manifolds $M$.

\newpage%.\newpage
%%%%%%%%%%%%%%%%%%%%%%%%%%%%%%%%%%%
%%%%%%%%%%%%%%%%%%%%%%%%%%%%%%%%%%%
%%%%%%% Section  %%%%%%%%%%%%%%%%%%%%%%
%%%%%%%%%%%%%%%%%%%%%%%%%%%%%%%%%%%
%%%%%%%%%%%%%%%%%%%%%%%%%%%%%%%%%%%
\section[Differentiability in Hilbert space valued Sobolev spaces]{Differentiability in Hilbert space valued\newline Sobolev spaces}
\label{sec:differentiability}

In this section we assume that $H_1$ and $H_2$ are separable Hilbert
spaces together with a dense incusion of $H_2$ in $H_1$.
In Section~\ref{sec:differentiability} we do not need that the
inclusion is compact.
For $i=1,2$ we abbreviate
$$
   W^{1,2}_{H_1}:=W^{1,2}(\R,H_1),\qquad
   L^2_{H_i}:=L^2(\R,H_i) .
$$
The intersection
$$
   W^{1,2}_{H_1}\cap L^2_{H_2}
$$
is itself a Hilbert space with inner product
$$
   \langle\langle\langle
   \xi,\eta
   \rangle \rangle \rangle
%   \INNER{\xi}{\eta}^{1,2,change}
   :=\int_{-\infty}^\infty \INNER{\dot\xi(s)}{\dot\eta(s)}_{H_1} ds
   +\int_{-\infty}^\infty \INNER{\xi(s)}{\eta(s)}_{H_2} ds .
$$
For $\xi,\eta\in \Ll(W^{1,2}_{H_1}\cap L^2_{H_2})$ we denote
\begin{equation}\label{eq:1,2-triple}
   \interleave \eta\interleave _{1,2}^2
   :=\norm{\eta}_{L^2_{H_2}}^2 +\norm{\dot\eta}_{L^2_{H_1}}^2
   ,\qquad
   \norm{\eta}_{1,2}^2= \norm{\eta}_{L^2_{H_1}}^2 +\norm{\dot\eta}_{L^2_{H_1}}^2.
\end{equation}

\boldmath
%%%%%%%%%%%%%%%%%%%%%%%%%%%%%%%%%%%
%%%%%%% Subsection  %%%%%%%%%%%%%%%%%%%
%%%%%%%%%%%%%%%%%%%%%%%%%%%%%%%%%%%
\subsection{Theorem~\ref{thm:B}}
\label{sec:Thm.B}
\unboldmath

\boldmath
%%%%%%%%%%%%%%%%%%%%%%%%%%%%%%%%%%%
%%%%%%% Subsubsection  %%%%%%%%%%%%%%%%
%%%%%%%%%%%%%%%%%%%%%%%%%%%%%%%%%%%
\subsubsection[Base component]{Base component $\Phi$}
\label{sec:Thm.B-Phi}
\unboldmath

\begin{theorem}\label{thm:B-Phi}
Let $U_1\subset H_1$ be open
and $\phi\colon U_1\to H_1$ be tame.
Assume in addition that $0\in U_1$ and $\phi(0)=0$.
With $W^{1,2}_{H_1}$ and $L^2_{H_i}$ as 
in~(\ref{eq:notation-spaces}) the map
\begin{equation*}
\begin{split}
   \Phi
   \colon
   W^{1,2}_{H_1}\cap L^2_{H_2}\supset\Uu
   &\to
   W^{1,2}_{H_1}\cap L^2_{H_2}
   \\
   \xi
   &\mapsto
   [s\mapsto \phi\circ\xi_s] 
\end{split}
\end{equation*}
is well defined and continuously differentiable where
\begin{equation}\label{eq:Uu}
   \Uu:=\{\xi\in W^{1,2}_{H_1}\cap L^2_{H_2}
   \mid \text{$\xi_s\in U_1$, $\forall s\in\R$} \} .
\end{equation}
\end{theorem}

\begin{proof}
The proof is in four steps.
We abbreviate
\begin{equation}\label{eq:abbreviate-1}
   \abs{\cdot}_1:=\norm{\cdot}_{H_1}
   ,\quad
   \abs{\cdot}_2:=\norm{\cdot}_{H_2}
   ,\quad
   \xi_s:=\xi(s)
   ,\quad
   \dot\xi:=\tfrac{d}{ds}\xi.
\end{equation}

\medskip
\noindent
\textbf{Step~1 (Well defined).}

\begin{proof}
By~(\ref{thm:Sob-estimate}) there is an embedding
\begin{equation}\label{eq:Kreuter}
   W^{1,2}_{H_1}:=W^{1,2}(\R,H_1)\INTO C^0(\R,H_1)
   ,\qquad
   \norm{v}_\infty\le \norm{v}_{1,2}
\end{equation}
with constant $1$
and this implies that $\Uu$ is open in $W^{1,2}_{H_1}\cap L^2_{H_2}$.
There exists an $H_1$-open neighborhood $W_0$ of $0$ in $U_1$
and a constant $\kappa_\infty>0$ such that
\begin{equation}\label{eq:phi(y)}
   \abs{\phi(y)}_2\le\kappa_\infty \abs{y}_2
\end{equation}
for every $y\in W_0$.
To see this observe that, since $\phi(0)=0$,
we can in the proof of Lemma~\ref{le:tame-0} choose
$x_0=0$ and then the estimate follows from~(\ref{eq:kjhghgfc6hh}).

\smallskip
\noindent
\textsc{Step~A.}
Let $\xi\in \Uu$, then $\Phi(\xi)\in L^2_{H_2}:=L^2(\R,H_2)\subset L^2(\R,H_1)$.

\smallskip
\noindent
$\bullet$ Asymptotic part:
To see this we claim that there exists $T=T(\xi)>0$
such that $\xi_s\in W_0$ whenever $\abs{s}\ge T$.
Since $W_0$ is an $H_1$-open neighborhood of $0$
there exists a $\delta>0$ such that the closed $\delta$-ball about
$0$ is contained in $W_0$.
Since $\xi\in W^{1,2}(\R,H_1)$ there exists $T>0$ such that
the restrictions satisfy
$\norm{\xi|_{(-\infty,T)}}_{1,2}\le \delta$ and
$\norm{\xi|_{(T,\infty)}}_{1,2}\le \delta$.
Hence, by Remark~\ref{rem:Sob-estimate}, we have
$\norm{\xi|_{(-\infty,T)}}_{\infty}\le \delta$ and
$\norm{\xi|_{(T,\infty)}}_{\infty}\le \delta$
and therefore $\xi_s\in W_0$ whenever $\abs{s}\ge T$.

\noindent
$\bullet$ Compact part $[-T,T]$: Since $\phi$ is tame,
for every $s\in[-T,T]$ there exists, by Lemma~\ref{le:tame-0},
an $H_1$-open neighborhood $W_s$
of $\xi_s:=\xi_s$ in $U_1$ and a constant $\kappa_s>0$ such that
for every $y\in W_s\cap H_2$ it holds
$\abs{\phi(y)}_2\le \kappa_s(1+\abs{y}_2)$.
The family $\{W_s\}_{s\in[-T,T]}$ is an open cover of the image $\xi([-T,T])$.
Since $\xi$ is in $W^{1,2}(\R,H_1)$ it is in particular a continuous
map $\xi\colon[-T,T]\to H_1$ and therefore this image is compact
in $H_1$.
Therefore there exist finitely many times $s_1,\dots,s_N\in [-T,T]$
such that the image $\xi([-T,T])$ already lies in the finite union
$\cup_{k=1}^N W_{s_k}$.

\noindent
$\bullet$ Let $\kappa:=\max\{\kappa_{s_1},\dots,\kappa_{s_N},\kappa_\infty\}$.
Then we have the estimate
$$
   \abs{\phi(\xi_s)}_2\le \kappa (1+\abs{\xi_s}_2)
   ,\quad
   \forall s\in[-T,T].
$$
In the asymptotic case $\abs{s}>T$, by~(\ref{eq:phi(y)}), we have the estimate
$$
   \abs{\phi(\xi_s)}_2\le \kappa \abs{\xi_s}_2
   ,\quad
   \forall \abs{s}>T .
$$
Using these two displayed estimates we obtain
\begin{equation*}
\begin{split}
   \norm{\Phi(\xi)}_{L^2_{H_2}}^2
   &=\int_{-\infty}^\infty \abs{\phi(\xi_s)}_2^2 \, ds
\\
   &=\int_{-\infty}^{-T} \abs{\phi(\xi_s)}_2^2 \, ds
   +
   \int_{-T}^T \abs{\phi(\xi_s)}_2^2 \, ds
   +
   \int_T^\infty \abs{\phi(\xi_s)}_2^2 \, ds
\\
   &\le\kappa^2 \int_{-\infty}^{-T} \abs{\xi_s}_2^2 \, ds
   +
   \kappa^2 \int_{-T}^T
   \underbrace{(1+\abs{\xi_s}_2)^2}_{\le 2(1+\abs{\xi_s}_2^2)} \, ds
   +
   \kappa^2\int_T^\infty \abs{\xi_s}_2^2 \, ds
\\
   &\le 2\kappa^2 \int_{-\infty}^\infty \abs{\xi_s}_2^2 \, ds
   +4\kappa^2 T   
\\
   &=2\kappa^2\norm{\xi}_{L^2_{H_2}}^2 +4\kappa^2 T <\infty.
\end{split}
\end{equation*}
This proves Step~A.

\smallskip
\noindent
\textsc{Step~B.}
Let $\xi\in \Uu$, then $\frac{d}{ds}\Phi(\xi)
=\frac{d}{ds} [s\mapsto \phi\circ\xi_s]\in L^2_{H_1}:=L^2(\R,H_1)$.

\smallskip
\noindent
$\bullet$
By definition of $\Phi$ we have
$\frac{d}{ds}(\Phi(\xi))_s=d\phi|_{\xi_s} \frac{d}{ds}\xi_s$.
As explained in Step~A (asymptotic part),
in view of Remark~\ref{rem:Sob-estimate}
it holds
$
   \lim_{s\to\mp\infty} \xi_s=0\in H_1
$.
Moreover, since $\xi\in W^{1,2}_{H_1}$, it is continuous as a map $\R\to H_1$
and therefore the image of $\xi$ in $H_1$ is compact.
Hence, since $\phi$ is continuously differentiable, there exists a constant
$\kappa$ such that $\norm{d\phi|_{\xi_s}}_{\Ll(H_1)}\le \kappa$
for every $s\in\R$.
Therefore $\norm{\frac{d}{ds}\Phi(\xi)}_{L^2_{H_1}}
\le\kappa\norm{\frac{d}{ds}\xi}_{L^2_{H_1}}$.
This finishes the proof of Step~B and Step~1.
\end{proof}

\smallskip
In the following three steps 2--4
we show that $\Phi$ is differentiable.

\medskip
\noindent
\textbf{Step~2 (Candidate).}
Given $\xi\in\Uu$, a natural candidate for the derivative is
$$
   \left(d\Phi|_\xi\hat\xi\right)(s)
   =d\phi|_{\xi_s}\hat\xi_s
$$
whenever $s\in\R$ and $\hat\xi\in W^{1,2}_{H_1}\cap L^2_{H_2}$.
The map $d\Phi|_\xi$ lies in $\Ll(W^{1,2}_{H_1}\cap L^2_{H_2})$.

\begin{proof} The proof has two steps A and B.

\smallskip
\noindent
\textsc{Step~A.}
We show that $d\Phi|_\xi\hat\xi$ lies in
$ W^{1,2}_{H_1}\cap L^2_{H_2}$.
To see this we first show that $d\Phi|_\xi\hat\xi\in L^2_{H_2}$.
Since the image of $\xi$ in $H_1$ is compact
and $\phi$ is tame, by the second inequality in~(\ref{eq:cor-tame}),
there exists $\kappa>0$ such that
$$
   \abs{d\phi|_{\xi_s}\eta}_2
   \le
   \kappa\left(\abs{\eta}_2+\abs{\xi_s}_2\abs{\eta}_1\right)
$$
for all $s\in\R$ and $\eta\in H_2$.
Use this estimate to obtain step 2 in what follows
\begin{equation}\label{eq:Phi-1-NEW}
\begin{split}
   \norm{d\Phi|_\xi\hat\xi}_{L^2_{H_2}}^2
   &=\int_\R \abs{d\phi|_{\xi_s}\hat\xi_s}_2^2 ds\\
   &\stackrel{2}{\le} \kappa^2 \int_\R \left(\abs{\hat\xi_s}_2
   +\abs{\xi_s}_2\abs{\hat\xi_s}_1\right)^2 ds\\
   &\le 2\kappa^2 \int_\R \left(\abs{\hat\xi_s}_2^2
   +\abs{\xi_s}_2^2\abs{\hat\xi_s}_1^2\right) ds\\
   &\le2\kappa^2\norm{\hat \xi}_{L^2_{H_2}}^2
   +2\kappa^2\norm{\hat\xi}_{L^\infty_{H_1}}^2
   \int_\R\abs{\xi_s}_2^2\, ds\\
   &\stackrel{5}{\le} \kappa^2\underline{\norm{\hat \xi}_{L^2_{H_2}}^2}
   +2\kappa^2 \underline{\norm{\hat\xi}_{W^{1,2}_{H_1}}^2}
   \norm{\xi}_{L^2_{H_2}}^2\\
   &\le\kappa^2\Bigl(1+2\norm{\xi}_{W^{1,2}_{H_1}\cap L^2_{H_2}}^2\Bigr)
   \underline{\norm{\hat \xi}_{W^{1,2}_{H_1}\cap L^2_{H_2}}^2} .
\end{split}
\end{equation}
In step 5 we used the embedding Theorem~\ref{thm:Sob-estimate}.

\smallskip
\noindent
\textsc{Step~B.}
Next we show that $\tfrac{d}{ds} d\Phi_\xi\hat\xi\in L^2_{H_1}$. To see
this we calculate
\begin{equation}\label{eq:S2-B}
\begin{split}
   \int_\R \abs{\tfrac{d}{ds}(d\phi|_{\xi_s}\hat\xi_s)}_1^2 ds
   =\int_\R \abs{d^2\phi|_{\xi_s}(\hat\xi_s,\dot\xi_s)
   +d\phi|_{\xi_s} \dot{\hat\xi}_s}_1^2ds .
\end{split}
\end{equation}
Term 1. As $d^2\phi$ is continuous and $\xi_s$
asymptotically converges to zero, there is a constant $c=c(\xi)$ such
that $\max\{\norm{d^2\phi|_{\xi_s}}_{\Ll(H_1,H_1;H_1)},
\norm{d\phi|_{\xi_s}}_{\Ll(H_1)}\}\le c$ for every $s\in\R$.
By~(\ref{eq:Sob-estimate}) we have
$\norm{\hat \xi}_{L^\infty_{H_1}}\le \norm{\hat\xi}_{W^{1,2}_{H_1}}$.
So we get
\begin{equation*}
\begin{split}
   \int_\R
   \abs{d^2\phi|_{\xi_s}(\hat\xi_s,\dot\xi_s)}_1^2ds
   &\le c^2
   \norm{\dot\xi}_{L^2_{H_1}}^2 \norm{\hat\xi}_{W^{1,2}_{H_1}}^2 .
\end{split}
\end{equation*}
Term 2. We have
$
   \int_\R\abs{d\phi|_{\xi_s}\dot{\hat\xi}_s}_1^2ds
   \le c^2 \norm{\dot{\hat{\xi}}}_{L^2_{H_1}}^2
   \le c^2 \norm{\hat{\xi}}_{W^{1,2}_{H_1}}^2
$.
Together we get
\begin{equation}\label{eq:Phi-2-NEW}
\begin{split}
   \norm{\tfrac{d}{ds}  d\Phi|_\xi\hat\xi}_{L^2_{H_1}}
   \le c \sqrt{2}\sqrt{1+\norm{\dot\xi}_{L^2_{H_1}}^2}
   \norm{\hat{\xi}}_{W^{1,2}_{H_1}} .
\end{split}
\end{equation}
By~(\ref{eq:Phi-1-NEW}) and~(\ref{eq:Phi-2-NEW}) there is a constant
$C_\xi$ such that
$$
   \norm{d\Phi|_\xi\hat\xi}_{W^{1,2}_{H_1}\cap L^2_{H_2}}
   \le C_\xi \norm{\hat\xi}_{W^{1,2}_{H_1}\cap L^2_{H_2}} .
$$
Hence the operator norm is bounded by
$$
   \norm{d\Phi|_\xi}_{\Ll(W^{1,2}_{H_1}\cap L^2_{H_2})}
   \le C_\xi .
$$
Thus the candidate is well defined and this proves Step~2.
\end{proof}

\medskip
\noindent
\textbf{Step~3 (Candidate from Step~2 is derivative of \boldmath$\Phi$).}

\begin{proof}
To $\xi,\eta\in \Ll(W^{1,2}_{H_1}\cap L^2_{H_2})$
applies the triple norm~\ref{eq:1,2-triple}.
Note that, in view of~(\ref{eq:12}), we have
\begin{equation}\label{eq:1,2-triple-est}
\boxed{
   \norm{\eta}_{1,2}
   \le \interleave \eta\interleave _{1,2} .
}
\end{equation}
To see that our candidate actually is the derivative we show that the
limit
\begin{equation}\label{eq:S3-Phi}
\begin{split}
   &\lim_{h\to 0}\sup_{\interleave \eta\interleave _{1,2}\le 1}
   \frac{\interleave \Phi(\xi+h\eta)-\Phi(\xi)-h\,d\phi|_\xi\eta\interleave _{1,2}^2}{h^2}
\\
   &=\lim_{h\to 0}\sup_{\interleave \eta\interleave _{1,2}\le 1}
   \frac{\int_\R\Abs{
   \phi (\xi_s+h\eta_s)-\phi (\xi_s)
   -h\,d\phi|_{\xi_s} \eta_s}_2^2ds}{h^2}
   \\
   &\quad+\lim_{h\to 0}\sup_{\interleave \eta\interleave _{1,2}\le 1}
   \frac{\int_\R\Abs{\tfrac{d}{ds}\left(\phi (\xi_s+h\eta_s)
   -\phi (\xi_s)-h\,d\phi|_{\xi_s} \eta_s\right)}_1^2ds}{h^2}
\end{split}
\end{equation}
exists and vanishes. We treat each of the two summands separately.

\smallskip\noindent
\emph{Summand 1.}
To see that summand 1 vanishes we use the fundamental
theorem of calculus to write it in the form
\begin{equation}\label{eq:hbjhbhj576567-NEW}
\begin{split}
   &%\lim_{h\to 0}
   \sup_{\interleave \eta\interleave _{1,2}\le 1}\tfrac{1}{h^2}\int_\R
   \biggl|\int_0^1\biggl(\underbrace{\tfrac{d}{dt}\phi(\xi_s+th\eta_s)}
      _{=d\phi|_{\xi_s+th\eta_s} h\eta_s}
   -h\, d\phi|_{\xi_s}\eta_s\biggr) dt\biggr|_2^2 ds
\\
   &=%\lim_{h\to 0}
   \sup_{\interleave \eta\interleave _{1,2}\le 1}
   \int_\R
   \biggl|\int_0^1\biggl(d\phi|_{\xi_s+th\eta_s}
   -d\phi|_{\xi_s}\biggr)\eta_s dt\biggr|_2^2 ds
\\
   &\le%\lim_{h\to 0}
   \sup_{\interleave \eta\interleave _{1,2}\le 1}
   \int_\R\left(
   \int_0^1\Abs{\left(d\phi|_{\xi_s+th\eta_s}
   -d\phi|_{\xi_s}\right)\eta_s}_2 dt\right)^2 ds .
\end{split}
\end{equation}
%
%\begin{remark}\label{rem:uniform-kappa}
By Lemma~\ref{le:tame-3} for every $s\in\R$ there exists
an $H_1$-open neighborhood $W_s$ of $\xi_s$ in $H_1$
and a constant $\kappa_s>0$ such that for all
$y_0,y_1\in W_s$ estimate~(\ref{eq:tame-3}) holds with $\kappa$
replaced by $\kappa_s$.
Since the image of $\xi$ in $H_1$ is compact
we can find a uniform constant $\kappa$ independent of $s$ such that
estimate~(\ref{eq:tame-3}) holds.
%\end{remark}
%
Since $W^{1,2}_{H_1}$ embeds in $L^\infty_{H_1}$
with constant $1$, see~(\ref{eq:Sob-estimate}),
there exists $h_0\in(0,1]$ such that $\xi_s+h\eta_s$ is element of $W_s$
for every $\eta$ with $\norm{\eta}_{1,2}\le h_0$. We estimate
\begin{equation*}
\begin{split}
   &%\lim_{h\to 0}
   \sup_{\interleave \eta\interleave _{1,2}\le 1}
   \int_\R\left(
   \int_0^1\Abs{\left(d\phi|_{\xi_s+th\eta_s}
   -d\phi|_{\xi_s}\right)\eta_s}_2 dt\right)^2 ds
\\
   &\stackrel{1}{\le} %\lim_{h\to 0}
   \sup_{\interleave \eta\interleave _{1,2}\le 1}
   \int_\R\left(
   \int_0^1
   2\kappa th\abs{\eta_s}_2\abs{\eta_s}_1
   +\tfrac{\kappa}{2}\left(\abs{\xi_s+th\eta_s}_2+\abs{\xi_s}_2\right)
   th\abs{\eta_s}_1^2
   \, dt\right)^2 ds
\\
   &\stackrel{2}{\le} %\lim_{h\to 0}
   \sup_{\interleave \eta\interleave _{1,2}\le 1}
   \int_\R\left(
   \int_0^1
   2\kappa th\abs{\eta_s}_2\abs{\eta_s}_1
   +\kappa th\abs{\xi_s}_2\abs{\eta_s}_1^2
   +\tfrac{\kappa}{2} t^2h^2 \abs{\eta_s}_2\abs{\eta_s}_1^2
   \, dt\right)^2 ds
\\
   &\stackrel{3}{\le} %\lim_{h\to 0}
   \sup_{\interleave \eta\interleave _{1,2}\le 1}
   \int_\R\left(
   \kappa h\abs{\eta_s}_2\abs{\eta_s}_1
   +\kappa \tfrac12 h\abs{\xi_s}_2\abs{\eta_s}_1^2
   +\tfrac{\kappa}{6} h^2 \abs{\eta_s}_2\abs{\eta_s}_1^2
   \right)^2 ds
\\
   &\stackrel{4}{\le} %\lim_{h\to 0}
   \sup_{\interleave \eta\interleave _{1,2}\le 1}
   h^2\kappa^2\max\{\norm{\eta}_{L^\infty_{H_1}}^2, \norm{\eta}_{L^\infty_{H_1}}^4\}
   \int_\R\underbrace{\left(
   \abs{\eta_s}_2
   +\abs{\xi_s}_2
   +\abs{\eta_s}_2
   \right)^2}_{\le 8 \abs{\eta_s}_2^2+2 \abs{\xi_s}_2} ds
\\
   &\stackrel{5}{\le} %\lim_{h\to 0}
   \sup_{\interleave \eta\interleave _{1,2}\le 1}
   h^2\kappa^2\max\{\norm{\eta}_{1,2}^2, \norm{\eta}_{1,2}^4\}
   \left(
      8\norm{\eta}_{L^2_{H_2}}^2+2\norm{\xi}_{L^2_{H_2}}
   \right)
\\
   &\stackrel{6}{\le} %\lim_{h\to 0}
   h^2\kappa^2
   \left(
      8+2\norm{\xi}_{L^2_{H_2}}
   \right)
   \longrightarrow 0\text{, as $h\to 0$.}
\end{split}
\end{equation*}
Step 1 is by~(\ref{eq:tame-3}.
Step 2 is by the triangle inequality.
Step 3 is by integrating $t$.
Step 4 pulls out the supremum norm of $\abs{\eta_s}_1$
and uses that $h\ge 1$.
Step 5 is by the Sobolev estimate
$\norm{\eta}_{L^\infty_{H_1}}\le \norm{\eta}_{1,2}$ from~(\ref{eq:Sob-estimate}).
Step 6 is by~(\ref{eq:1,2-triple-est}).

The previous estimate, which starts at~(\ref{eq:hbjhbhj576567-NEW}),
tells that summand 1 is zero.

\smallskip\noindent
\emph{Summand 2.}
To see that summand 2 vanishes, we abbreviate and compute
\begin{equation*}
\begin{split}
   G(s):
   &=\tfrac{d}{ds}\left(\phi (\xi_s+h\eta_s)
   -\phi (\xi_s)-h\, d\phi|_{\xi_s} \eta_s\right)
\\
%1
   &=
   d\phi|_{\xi_s+h\eta_s}(\dot\xi_s+h\dot\eta_s)
   \\
   &\quad-d\phi|_{\xi_s} \dot\xi_s
   \\
   &\quad
   -h\, d^2\phi|_{\xi_s}(\dot\xi_s,\eta_s)
   -h\, d\phi|_{\xi_s} \dot\eta_s
\\
%2
   &=
   d\phi|_{\xi_s+h\eta_s}\dot\xi_s
   -d\phi|_{\xi_s} \dot\xi_s
   -h\, d^2\phi|_{\xi_s}(\dot\xi_s,\eta_s)
   \\
   &\quad +d\phi|_{\xi_s+h\eta_s} h\dot\eta_s -h\, d\phi|_{\xi_s} \dot\eta_s
\\
%3
   &=
   h\int_0^1\left(d^2\phi|_{\xi_s+th\eta_s}-d^2\phi|_{\xi_s}\right)
   (\dot\xi_s,\eta_s)\, dt
   \\
   &\quad+h \left(d\phi|_{\xi_s+h\eta_s} -d\phi|_{\xi_s}
   \right) \dot \eta_s
\end{split}
\end{equation*}
for every $s\in\R$.
Square this identity and integrate to obtain
\begin{equation*}
\begin{split}
   &\sup_{\interleave \eta\interleave _{1,2}\le 1}\tfrac{1}{h^2}\int_\R\abs{G(s)}_{1}^2 ds
   \\
   &\le
   \sup_{\interleave \eta\interleave _{1,2}\le 1}
   \int_\R
     4\biggl|\int_0^1\left(d^2\phi|_{\xi_s+th\eta_s}
     -d^2\phi|_{\xi_s}\right)
   (\dot\xi_s,\eta_s)\, dt \biggr|^2_{1} ds
   \\
   &\quad+\sup_{\interleave \eta\interleave _{1,2}\le 1}
   \int_\R 4\biggl|\left(d\phi|_{\xi_s+h\eta_s} -d\phi|_{\xi_s}
   \right) \dot \eta_s \biggr|^2_{1} ds .
\end{split}
\end{equation*}
There are two terms in the sum.
\\
Term 1.
We first claim that for every $\eps>0$ there exists $h_0=h_0(\eps)>0$
such that whenever $h\in[0,h_0]$, then
\begin{equation}\label{eq:hvgvh49}
\begin{split}
   \norm{d\phi|_{\xi_s+th\eta_s}
   -d\phi|_{\xi_s}}_{\Ll(H_1)}
   &<\eps ,
\\
   \norm{d^2\phi|_{\xi_s+th\eta_s}
   -d^2\phi|_{\xi_s}}_{\Ll(H_1,H_1;H_1)}
   &<\eps .
\end{split}
\end{equation}
To see this note that by~(\ref{eq:Sob-estimate})
and~(\ref{eq:1,2-triple-est}) we have that
$$
\boxed{
   \abs{\eta(s)}_1
   \le\norm{\eta}_{L^\infty_{H_1}}
   \le\norm{\eta}_{1,2}
   \le\interleave \eta\interleave _{1,2}
   \le 1
}
$$
whenever $s\in\R$.
Hence the claim follows since $\phi$ is $C^2$
on $H_1$, by Definition~\ref{def:tame} of tameness.
Now suppose that $h\le h_0$, then we estimate term 1 by
\begin{equation*}
\begin{split}
   &\sup_{\interleave \eta\interleave _{1,2}\le 1}
   \int_\R
     4\biggl|\int_0^1\left(d^2\phi|_{\xi_s+th\eta_s}
     -d^2\phi|_{\xi_s}\right)
   (\dot\xi_s,\eta_s)\, dt \biggr|^2_{1} ds
\\
%1
   &\le \sup_{\interleave \eta\interleave _{1,2}\le 1}
   \int_\R
     4\left(\int_0^1\Abs{\left(d^2\phi|_{\xi_s+th\eta_s}
     -d^2\phi|_{\xi_s}\right)
   (\dot\xi_s,\eta_s)}_{1} dt \right)^2 ds
\\
%2
   &\le \sup_{\interleave \eta\interleave _{1,2}\le 1}
   \int_\R
     4\left(\int_0^1\norm{d^2\phi|_{\xi_s+th\eta_s}
     -d^2\phi|_{\xi_s}}_{\Ll(H_1,H_1;H_1)}\cdot
   \abs{\dot\xi_s}_{1}\abs{\eta_s}_{1}
   dt \right)^2 ds
\\
%3
   &\stackrel{3}{\le} 4\eps^2
   \int_\R \left(\int_0^1 \abs{\dot\xi_s}_{1} dt\right)^2  ds
\\
%4
   &\le 4\eps^2\norm{\xi}_{1,2}^2 .
\end{split}
\end{equation*}
In step 3 we used that $\abs{\eta_s}_1\le 1$, as shown after~(\ref{eq:hvgvh49}),
and we also used estimate two in~(\ref{eq:hvgvh49}).
Since $\eps>0$ was arbitrary, the limit of term 1, as $h\to 0$, is zero.
\\
Term 2. We estimate
\begin{equation*}
\begin{split}
   &\sup_{\interleave \eta\interleave _{1,2}\le 1}
   \int_\R 4\Abs{\left(d\phi|_{\xi_s+h\eta_s} -d\phi|_{\xi_s}
   \right) \dot \eta_s}^2_{1} ds
\\
%1
   &\le4\sup_{\interleave \eta\interleave _{1,2}\le 1}
   \int_\R \norm{d\phi|_{\xi_s+h\eta_s}
     -d\phi|_{\xi_s}}_{\Ll(H_1)}^2
   \abs{\dot\eta_s}^2_{1} ds
\\
%2
   &\le 4\eps^2 \sup_{\interleave \eta\interleave _{1,2}\le 1}
   \underbrace{\int_\R \abs{\dot\eta_s}^2_{1}
     ds}_{=\norm{\dot\eta}_{L^2_{H_1}}^2\le\interleave \eta\interleave _{1,2}^2\le 1}
\\
%5
   &\le 4\eps^2 .
\end{split}
\end{equation*}
Since $\eps>0$ was arbitrary, the limit of term 2, as $h\to 0$, is
zero as well.

We proved that summand 2 vanishes as well.
Therefore our candidate is the derivative
and this concludes the proof of Step~3.
\end{proof}

\medskip
\noindent
\textbf{Step~4 (Differential is continuous).}
The differential
$$
   d\Phi\colon \Uu\to \Ll(W^{1,2}_{H_1}\cap L^2_{H_2})
   ,\quad
   \xi\mapsto d\Phi|_{\xi}
$$
is continuous.

\begin{proof}
Recall $\interleave \cdot\interleave _{1,2}$ in~(\ref{eq:1,2-triple}).
For $\xi, \tilde\xi \in \Uu$ and 
$\eta\in W^{1,2}_{H_1}\cap L^2_{H_2}$ we estimate
\begin{equation}\label{eq:S4-Phi}
\begin{split}
   &\interleave d\Phi|_\xi \eta-d\Phi|_{\tilde \xi} \eta\interleave _{1,2}^2\\
%1
   &=\int_\R \Abs{\bigl(d\phi|_{\xi_s}-d\phi|_{\tilde\xi_s}\bigr)\eta_s}^2_{2} ds
   +\int_\R
   \Abs{\tfrac{d}{ds}\bigl[\bigl(d\phi|_{\xi_s}-d\phi|_{\tilde\xi_s}\bigr)\eta_s\bigr]}^2_{1} ds
\\
%2
   &\le\int_\R \Abs{\bigl(d\phi|_{\xi_s}-d\phi|_{\tilde\xi_s}\bigr)\eta_s}^2_{2} ds
   \\
   &\quad+\int_\R\Abs{\bigl(
   d^2\phi|_{\xi_s}\dot\xi_s
{\color{gray}\;
   -\;d^2\phi|_{\tilde\xi_s}\dot\xi_s
   +d^2\phi|_{\tilde\xi_s}\dot\xi_s
}
   -d^2\phi|_{\tilde\xi_s}\dot{\tilde\xi}_s\bigr)
   \eta_s}^2_{1} ds
   \\
   &\quad+\int_\R\abs{\bigl(d\phi|_{\xi_s}-d\phi|_{\tilde\xi_s}\bigr)
   \dot\eta_s}^2_{1} ds
\\
%3
   &\le\int_\R \Abs{\bigl(d\phi|_{\xi_s}-d\phi|_{\tilde\xi_s}\bigr)\eta_s}^2_{2} ds
   \\
   &\quad+\int_\R 2\Abs{\bigl(
   d^2\phi|_{\xi_s}\dot\xi_s
   -d^2\phi|_{\tilde\xi_s}\dot\xi_s
   \bigr)
   \eta_s}^2_{1} ds
   \\
   &\quad+\int_\R 2\Abs{
   d^2\phi|_{\tilde\xi_s}
   \bigl(\dot\xi_s-\dot{\tilde\xi}_s,\eta_s\bigr)}^2_{1} ds
   \\
   &\quad+\int_\R \Abs{\bigl(d\phi|_{\xi_s}-d\phi|_{\tilde\xi_s}\bigr)
   \dot\eta_s}^2_{1} ds .
\end{split}
\end{equation}
\underline{Term 1.}
By Lemma~\ref{le:tame-3} for every $s\in\R$ there exists
an $H_1$-open neighborhood $W_s$ of $\xi_s$ in $H_1$
and a constant $\kappa_s>0$ such that for all
$y_0,y_1\in W_s$ estimate~(\ref{eq:tame-3}) holds with $\kappa$
replaced by $\kappa_s$.
Since the image of $\xi$ in $H_1$ is compact
we can find a uniform constant $\kappa$ independent of $s$ such that
estimate~(\ref{eq:tame-3}) holds.
\\
If $\tilde\xi$ lies in a sufficiently small
$\interleave \cdot\interleave _{1,2}$-neighborhood around $\xi$, we can assume that
$\tilde \xi_s\in W_s$ for every $s\in\R$.
We use this in inequality 1 in the estimate
\begin{equation}\label{eq:Term-1-xx-NEW}
\begin{split}
   &\int_\R \Abs{\bigl(d\phi|_{\xi_s}
   -d\phi|_{\tilde\xi_s}\bigr)\eta_s}^2_{2} ds
\\
%1
   &\stackrel{1}{\le} \kappa^2\int_\R
   \biggl(
   \abs{\xi_s-\tilde\xi_s}_1\abs{\eta_s}_2
   +\abs{\xi_s-\tilde\xi_s}_2\abs{\eta_s}_1
   +\tfrac12\underbrace{\bigl(\abs{\xi_s}_2+\abs{\tilde\xi_s}_2\bigr)}
      _{\le2\abs{\xi_s}_2+\abs{\xi_s-\tilde\xi_s}_2}
   \abs{\xi_s-\tilde\xi_s}_1\abs{\eta_s}_1
   \biggr)^2 ds
\\
%2
   &\stackrel{2}{\le} 4\kappa^2\int_\R
   \biggl(
   \abs{\xi_s-\tilde\xi_s}_1^2\abs{\eta_s}_2^2
   +\abs{\xi_s-\tilde\xi_s}_2^2\abs{\eta_s}_1^2
   +\left(\abs{\xi_s}_2^2+\abs{\xi_s-\tilde\xi_s}_2^2\right)
   \abs{\xi_s-\tilde\xi_s}_1^2\abs{\eta_s}_1^2
   \biggr) ds
\\
%3
   &\stackrel{3}{\le} 4\kappa^2
   \left(
   \norm{\xi-\tilde\xi}_\infty^2\norm{\eta}_{L^2_{H_2}}^2
   +\norm{\eta}_\infty^2\norm{\xi-\tilde\xi}_{L^2_{H_2}}^2\right)
   \\
   &\quad+4\kappa^2
   \left(\norm{\xi}_{L^2_{H_2}}^2+\norm{\xi-\tilde\xi}_{L^2_{H_2}}^2\right)
      \norm{\xi-\tilde\xi}_\infty^2 \norm{\eta}_\infty^2
\\
%4
   &\le4\kappa^2 \interleave \xi-\tilde\xi\interleave _{1,2}^2
   \left(2+\interleave \xi\interleave _{1,2}^2+\interleave \xi-\tilde\xi\interleave _{1,2}^2\right)
   \interleave \eta\interleave _{1,2}^2 .
\end{split}
\end{equation}
In step 2 we used the inequality $(a+b+c+d)^2\le 4(a^2+b^2+c^2+d^2)$.
In step 3 we pulled out $L^\infty$ norms
which in step 4 we estimated by the $\interleave \cdot\interleave _{1,2}$
norms~(\ref{eq:1,2-triple}) with constant $1$.
\\
\underline{Term 2.}
Since tame maps are $C^2$ on level one, the second differential
$d^2\phi$ is continuous.
Hence, given $\eps>0$, there exists $\delta>0$ such that
$$
   \norm{\tilde\xi-\xi}_{1,2}\le\delta
   \quad\Rightarrow\quad
   \norm{d^2\phi|_{\xi_s}-d^2\phi|_{\tilde\xi_s}}_{\Ll(H_1,H_1;H_1)}
  \le\eps .
$$
In particular, for every $\tilde\xi$ in the $W^{1,2}$ $\delta$-ball
around $\xi$ it holds that
\begin{equation*}
\begin{split}
   \int_\R 2\Abs{\bigl(
   d^2\phi|_{\xi_s}\dot\xi_s
   -d^2\phi|_{\tilde\xi_s}\dot\xi_s
   \bigr)
   \eta_s}_{1}^2ds
   &\le2\eps^2
   \Norm{\abs{\dot\xi}_{1}\cdot\abs{\eta}_{1}}_{L^2_{H_1}}^2
\\
   &\le
   2\eps^2\interleave \xi\interleave _{1,2}^2\interleave \eta\interleave _{1,2}^2 .
\end{split}
\end{equation*}
\underline{Term 3.}
By continuity of $d^2\phi$ there exists a constant $c>0$, only
depending on $\xi$ but not $\tilde\xi$, such that
$$
   \norm{\tilde\xi-\xi}_{1,2}\le\delta
   \quad\Rightarrow\quad
   \norm{d^2\phi|_{\tilde\xi_s}}_{\Ll(H_1,H_1;H_1)}
  \le c .
$$
Hence we estimate term 3, similarly as term 2, by
\begin{equation*}
\begin{split}
   \int_\R 2\abs{
   d^2\phi|_{\tilde\xi_s}
   \bigl(\dot\xi_s-\dot{\tilde\xi}_s,\eta_s\bigr)}_{1}^2ds
   &\le 2c^2
   \Norm{\abs{\dot\xi-\dot{\tilde\xi}}_{1}\cdot\abs{\eta}_{1}}_{L^2_{H_1}}^2\\
   &\le 2c^2
   \interleave \dot\xi-\dot{\tilde\xi}\interleave _{1,2}^2\cdot
   \interleave \eta\interleave _{1,2}^2 .
\end{split}
\end{equation*}
\underline{Term 4.}
Given $\eps>0$, by continuity of the map $d\phi$,
there exists $\delta>0$ with
$$
   \norm{\tilde\xi-\xi}_{1,2}\le\delta
   \quad\Rightarrow\quad
   \norm{d\phi|_{\xi_s}-d\phi|_{\tilde\xi_s}}_{\Ll(H_1)}
  \le\eps .
$$
In particular, for every $\tilde\xi$ in the $W^{1,2}$ $\delta$-ball
around $\xi$ it holds that
$$
   \int_\R \abs{\bigl(d\phi|_{\xi_s}
   -d\phi|_{\tilde\xi_s}\bigr)\dot\eta_s}_{1}^2 ds
   \le\eps^2\int_\R \abs{\dot{\eta}_s}_{1}^2 \, ds
   \le\eps^2\interleave \eta\interleave _{1,2}^2 .
$$

\medskip\noindent
\underline{Conclusion.}
The term by term analysis above shows that
for every $\eps>0$ there exists a $\delta>0$ such that
\begin{equation}\label{eq:S4-concl}
\begin{split}
   &\lim_{\interleave \tilde\xi-\xi\interleave _{1,2}\to 0}
   \norm{d\Phi|_\xi-d\Phi|_{\tilde \xi} }_{\Ll(W^{1,2}_{H_1}\cap L^2_{H_2})}^2
\\
   &=\lim_{\interleave \tilde\xi-\xi\interleave _{1,2}\to 0}\sup_{\interleave \eta\interleave _{1,2}\le 1}
   \interleave d\Phi|_\xi \eta-d\Phi|_{\tilde \xi} \eta\interleave _{1,2}^2
\\
   &=0 .
\end{split}
\end{equation}
This proves Step~4.
\end{proof}

The proof of
Theorem~\ref{thm:B-Phi}
is complete.
\end{proof}

\boldmath
%%%%%%%%%%%%%%%%%%%%%%%%%%%%%%%%%%%
%%%%%%% Subsubsection  %%%%%%%%%%%%%%%%
%%%%%%%%%%%%%%%%%%%%%%%%%%%%%%%%%%%
\subsubsection[Weak tangent map]{Weak tangent map $T\Phi$}
\label{sec:Thm.B-TPhi}
\unboldmath

\begin{theorem}\label{thm:B-TPhi}
Let $U_1\subset H_1$ be open and $\phi\colon U_1\to H_1$ be tame.
Assume in addition that $0\in U_1$ and $\phi(0)=0$. With
$W^{1,2}_{H_1}$ and $L^2_{H_i}$ as in~(\ref{eq:notation-spaces}) and
$\Uu$ as in~(\ref{eq:Uu}) the weak tangent map
\begin{equation*}
\begin{split}
   T\Phi
   \colon
   \Uu\times L^2_{H_1}
   &\to
   (W^{1,2}_{H_1}\cap L^2_{H_2})\times L^2_{H_1}
   \\
   (\xi,\eta)
   &\mapsto
   \left(\phi\circ\xi, d\phi|_\xi \eta\right) =\left(\Phi(\xi), d\Phi|_\xi\eta \right)
\end{split}
\end{equation*}
is well defined and continuously differentiable.
\end{theorem}

\begin{proof}
The first component map $\xi\mapsto\phi\circ\xi$ is of class $C^1$ as
we proved in Theorem~\ref{thm:B-Phi}.
We use the abbreviations from~(\ref{thm:B-Phi}).
The proof is in four steps.

\medskip
\noindent
\textbf{Step~1 (Well defined).}

\begin{proof}
In view of Theorem~\ref{thm:B-Phi}
it suffices to show that the map $s\mapsto d\phi|_{\xi_s}\eta_s$
is in $L^2_{H_1}$.
Since $\xi$ is in $W^{1,2}_{H_1}$, it is continuous (with embedding
constant $1$) and converges asymptotically to zero;
see Section~\ref{sec:Sobolev}.
Therefore $\norm{d\phi|_{\xi_s}}_{\Ll(H_1)}$ is uniformly bounded
by a constant $c_1(\xi)$ depending on $\xi$. Hence the $L^2$ norm of
$d\phi|_{\xi_s}\eta_s$ can be estimated above by this constant
times the $L^2_{H_1}$ norm of $\eta_s$.
\end{proof}

\medskip
\noindent
\textbf{Step~2 (Candidate).}
A natural candidate for the derivative of $T\Phi$ at $(\xi,\eta)$ is
\begin{equation}\label{eq:cand-dTPhi-NEW}
\begin{split}
   d(T\Phi)|_{(\xi,\eta)}\colon
   \left(W^{1,2}_{H_1}\cap L^2_{H_2}\right)\times L^2_{H_1} 
   &\to \left(W^{1,2}_{H_1}\cap L^2_{H_2}\right)\times L^2_{H_1} 
\\
   \left(d(T\Phi)|_{(\xi,\eta)}(\hat\xi,\hat\eta)\right)(s)
   &=\left(d\Phi|_\xi\hat\xi,
   d\left(d\Phi|_\xi\eta\right)(\hat\xi,\hat\eta)\right) (s)
\\
  :&=\left(
   d\phi|_{\xi_s}\hat{\xi}_s,
   d^2\phi|_{\xi_s}(\hat\xi_s,\eta_s)+d\phi|_{\xi_s}\hat{\eta}_s
   \right)
\end{split}
\end{equation}
for all $s\in\R$ and $(\xi,\eta)\in\Uu\times L^2_{H_1}$
and where $d(d\Phi|_\xi\eta)$ means
$d\left[(\xi,\eta)\mapsto d\Phi|_\xi\eta\right]$. The map
$d(T\Phi)|_{(\xi,\eta)}$ is in $\Ll((W^{1,2}_{H_1}\cap L^2_{H_2})\times L^2_{H_1})$.

\begin{proof}
Component one is the subject of Theorem~\ref{thm:B-Phi}.
Concerning component two,
first we show that $s\mapsto d^2\phi|_{\xi_s}(\hat\xi_s,\eta_s)$
is in $L^2_{H_1}$.
As $\xi$ is in $W^{1,2}_{H_1}$, it is continuous and converges
asymptotically to zero.
Therefore, since $\phi$ is $C^2$, there is a constant $c_2(\xi)$ such that
$\norm{d^2\phi|_{\xi_s}}_{\Ll(H_1,H_1;H_1)}\le c_2(\xi)$ for any $s\in\R$.
We get
\begin{equation*}
\begin{split}
   \norm{d^2\phi|_{\xi}(\hat\xi,\eta)}_{L^2_{H_1}}
   &\le c_2(\xi)\norm{\eta}_{L^2(\R,\R^n)}\norm{\hat\xi}_{C^0(\R,\R^n)}\\
   &\le c_2(\xi) \norm{\eta}_{L^2_{H_1}}
   \norm{\hat\xi}_{W^{1,2}(\R,\R^n)} .
\end{split}
\end{equation*}
As explained in Step~1, we further have that
$\norm{d\phi|_{\xi}\hat{\eta}}_{L^2_{H_1}}
\le c_1(\xi) \norm{\hat{\eta}}_{L^2_{H_1}}$.
These estimates together with
Step~2 from Theorem~\ref{thm:B-Phi}
show that $d(T\Phi)|_{(\xi,\eta)}$ is a bounded linear map on
$W^{1,2}_{H_1}\cap L^2_{H_2}$.
This proves Step~2.
\end{proof}

\begin{remark}
To show that $T\Phi$, as a map
$$
   T\Phi\colon
   \underbrace{\Uu\times\Ww}_{T_\Uu\Ww}
   \to \underbrace{\Ww\times\Ww}_{T\Ww}
   ,\qquad
   \Uu\subset\Ww:=(W^{1,2}_{H_1}\cap L^2_{H_2}) ,
%   \Uu \times (W^{1,2}_{H_1}\cap L^2_{H_2})
%   \to  (W^{1,2}_{H_1}\cap L^2_{H_2}) \times (W^{1,2}_{H_1}\cap L^2_{H_2})
$$
is $C^1$ would fail right away in Step~2:
It would require that both maps
$s\mapsto d^2\phi|_{\xi_s}(\hat\xi_s,\eta_s)$
and $s\mapsto \tfrac{d}{ds} d^2\phi|_{\xi_s}(\hat\xi_s,\eta_s)$
are in $L^2_{H_1}$. But the second one would require three
continuous derivatives of $\phi$.
\end{remark}

\medskip
\noindent
\textbf{Step~3 (Candidate from Step~2 is derivative of \boldmath$T\Phi$).}

\begin{proof}
Let $\xi,\hat\xi\in W^{1,2}_{H_1}$
and $\eta,\hat\eta\in L^2_{H_1}$.
Recall that the norm $\interleave \cdot\interleave _{1,2}$ on $W^{1,2}_{H_1}\cap L^2_{H_2}$
is given by~(\ref{eq:1,2-triple}).
To see that our candidate actually is the derivative we show that the limit
\begin{equation*}
\begin{split}
   &\lim_{h\to 0}\sup_{\interleave \hat\xi\interleave _{1,2}^2+\norm{\hat\eta}_{L^2_{H_1}}^2\le 1}
   \frac{\interleave T\Phi(\xi+h\hat\xi,\eta+h\hat\eta)-T\Phi(\xi,\eta)
   -h\,dT\Phi(\xi,\eta)(\hat\xi,\hat\eta)\interleave _{1,2}^2}{h^2}
\end{split}
\end{equation*}
exists and vanishes.
The map $T\Phi=(\Phi,d\Phi)$ has two components.
Step~2 in Theorem~\ref{thm:B-Phi}
shows for the first component that the limit exists and vanishes.
Hence it suffices to show this for the second component, namely
\begin{equation*}
\begin{split}
   &\lim_{h\to 0}\sup_{\interleave \hat\xi\interleave _{1,2}^2+\norm{\hat\eta}_{L^2_{H_1}}^2\le 1}
   \frac{\norm{
   d\phi|_{\xi+h\hat\xi}(\eta+h\hat\eta)
   -d\phi|_\xi\eta
   -h\, d^2\phi|_\xi (\hat\xi,\eta)-h\, d\phi|_\xi\hat\eta}_{L^2_{H_1}}^2
   }{h^2}
\\
   &\le
   \lim_{h\to 0}\sup_{\interleave \hat\xi\interleave _{1,2}^2+\norm{\hat\eta}_{L^2_{H_1}}^2\le 1}
   \frac{\norm{
   d\phi|_{\xi+h\hat\xi}\eta
   -d\phi|_\xi\eta
   -h\, d^2\phi|_\xi (\hat\xi,\eta)}_{L^2_{H_1}}^2
   }{h^2}
   \\
   &\quad +
   \lim_{h\to 0}\sup_{\interleave \hat\xi\interleave _{1,2}^2+\norm{\hat\eta}_{L^2_{H_1}}^2\le 1}
   \frac{h^2 \norm{
   d\phi|_{\xi+h\hat\xi}\hat\eta
   -d\phi|_\xi\hat\eta}_{L^2_{H_1}}^2
   }{h^2}
\end{split}
\end{equation*}
We treat each of the two summands separately.
\\
\emph{Summand 1.}
To see that summand 1 vanishes note the following.
Since $\phi$ is tame it is $C^2$. Hence for any given $\eps>0$ there
exists $h_0>0$ such that
$$
   h\in(0,h_0)
   \quad\Rightarrow\quad
   \norm{d^2\phi|_{\xi_s+th\hat\xi_s}
   -d^2\phi|_{\xi_s}}_{\Ll(H_1,H_1;H_1)}\le\eps .
$$
whenever $\interleave \hat\xi\interleave _{1,2}\le 1$, $s\in\R$, and $t\in[0,1]$.
Pick $h\in(0,h_0)$ and use the fundamental theorem of calculus to estimate
\begin{equation*}
\begin{split}
   &\sup_{\interleave \hat\xi\interleave _{1,2}^2+\norm{\hat\eta}_{L^2_{H_1}}^2\le 1}
   \tfrac{1}{h^2}
   \int_\R \biggl|\int_0^1 \biggl(
   \underbrace{\tfrac{d}{dt} d\phi|_{\xi_s+th\hat\xi_s} \eta_s}
      _{h d^2\phi|_{\xi_s+th\hat\xi_s}(\hat\xi_s,\eta_s)}
   -h\, d^2\phi|_{\xi_s} (\hat\xi_s,\eta_s)
   \biggr) dt \biggl|_{H_1}^2
   ds
\\
   &\le \sup_{\interleave \hat\xi\interleave _{1,2}^2+\norm{\hat\eta}_{L^2_{H_1}}^2\le 1}
   \underbrace{\int_\R\int_0^1 \Abs{
   d^2\phi|_{\xi_s+th\hat\xi_s}(\hat\xi_s,\eta_s)- d^2\phi|_{\xi_s} (\hat\xi_s,\eta_s)
   }_{H_1}^2 dt\, ds}_{\le \eps^2 \interleave \hat\xi\interleave _{1,2}^2\norm{\eta}_{L^2_{H_1}}^2}
\\
   &\le\eps^2 \norm{\eta}_{L^2_{H_1}}^2 .
\end{split}
\end{equation*}
Since $\eps$ was arbitrary the limit, as $h\to 0$, vanishes.
\\
\emph{Summand 2.}
To see that summand 2 vanishes
note that since $\phi\in C^1$
for any given $\eps>0$ there exists $h_0>0$ such that
$\norm{d\phi|_{\xi_s+h\hat\xi_s}-d\phi|_{\xi_s}}_{\Ll(H_1)}\le\eps$
whenever $h\in(0,h_0)$, $\interleave \hat\xi\interleave _{1,2}\le 1$, and $s\in\R$.
Thus
$$
   \sup_{\interleave \hat\xi\interleave _{1,2}^2+\norm{\hat\eta}_{L^2_{H_1}}^2\le 1}
   \norm{
   d\phi|_{\xi+h\hat\xi}\hat\eta
   -d\phi|_\xi\hat\eta}_{L^2_{H_1}}^2
   \le \sup_{\norm{\hat\eta}_{L^2_{H_1}}\le 1} \eps^2
   \norm{\hat\eta}_{L^2_{H_1}}^2
   \le\eps^2 .
$$
Since $\eps$ was arbitrary the limit, as $h\to 0$, vanishes.
This proves Step~3.
\end{proof}

\medskip
\noindent
\textbf{Step~4 (Differential is continuous).}
The differential
\begin{equation*}
\begin{split}
   d(T\Phi)\colon \Uu\times L^2_{H_1}
   &\to  \Ll((W^{1,2}_{H_1}\cap L^2_{H_2})\times L^2_{H_1})
\\
   (\xi,\eta)&\mapsto d(T\Phi)|_{(\xi,\eta)}
\end{split}
\end{equation*}
is continuous.

\begin{proof}
The map $d(T\Phi)$, see~(\ref{eq:cand-dTPhi-NEW}), has two components.
Continuity of the first component
is shown in Step~4 in the proof of
Theorem~\ref{thm:B-Phi}.
Hence it suffices to show continuity of the second component, namely

Pick $(\xi,\eta)\in \Uu\times L^2_{H_1}$.
Let $(v,w)$ and $(\hat\xi,\hat\eta)$ be elements of
$(W^{1,2}_{H_1}\cap L^2_{H_2})\times L^2_{H_1}$.
Since $\phi\in C^2$ the following holds.
Given $\eps>0$, there exists $\delta_\eps>0$, such that
$\interleave v\interleave _{1,2}<\delta_\eps$ implies
$$
   \norm{d\phi|_{\xi_s+v_s}
   -d\phi|_{\xi_s}}_{\Ll(H_1)}
   ,
   \norm{d^2\phi|_{\xi_s+v_s}
   -d^2\phi|_{\xi_s}}_{\Ll(H_1,H_1;H_1)}
   <\eps
$$
whenever $s\in\R$.
Moreover, maybe after shrinking $\delta>0$,
we can assume that there is a constant $c_\xi$ such that
$\norm{d^2\phi|_{\xi_s+v_s}}_{\Ll(H_1,H_1;H_1)}\le c_\xi$.
We assume that
$$
   \interleave v\interleave _{1,2}\le\delta_\eps
   ,\qquad
   \norm{w}_{L^2_{H_1}}<\eps .
$$
We estimate
\begin{equation*}
\begin{split}
   &\norm{
   d^2\phi|_{\xi+v}(\hat\xi,\eta+w)+d\phi|_{\xi+v}\hat{\eta}
   -d^2\phi|_{\xi}(\hat\xi,\eta)-d\phi|_{\xi}\hat{\eta}
   }_{L^2_{H_1}}^2
\\
   &=\int_\R
   \Abs{d^2\phi|_{\xi_s+v_s}(\hat\xi_s,\eta_s+w_s)+d\phi|_{\xi_s+v_s}\hat{\eta}_s
   -d^2\phi|_{\xi_s}(\hat\xi_s,\eta_s)-d\phi|_{\xi_s}\hat{\eta}_s}_{H_1}^2 ds
\\
   &\le2\int_\R 
   \Abs{d^2\phi|_{\xi_s+v_s}(\hat\xi_s,\eta_s)
   -d^2\phi|_{\xi_s}(\hat\xi_s,\eta_s)}_{H_1}^2
   ds
   \\
   &\quad
   +2\int_\R
   2\Abs{d\phi|_{\xi_s+v_s}\hat{\eta}_s-d\phi|_{\xi_s}\hat{\eta}_s}_{H_1}^2
   +2\Abs{d^2\phi|_{\xi_s+v_s} (\hat\xi_s,w_s)}_{H_1}^2
   ds
\\
   &\le
   2\eps^2\interleave \hat\xi\interleave _{1,2}^2\norm{\eta}_{L^2_{H_1}}^2
   +4\eps^2\norm{\hat\eta}_{L^2_{H_1}}^2
   +4c_\xi^2  \interleave \hat\xi\interleave _{1,2}^2 \norm{w}_{L^2_{H_1}}^2
\\
   &\le \eps^2 \left( 2 \norm{\eta}_{L^2_{H_1}}^2
   +4c_\xi^2  +4 \right) 
   \left(\interleave \hat\xi\interleave _{1,2}^2+\norm{\hat\eta}_{L^2_{H_1}}^2\right) .
\end{split}
\end{equation*}
This proves continuity of the map $d(T\Phi)$.
This proves Step~4.
\end{proof}

This concludes the proof of Theorem~\ref{thm:B-TPhi}.
\end{proof}

\boldmath
%%%%%%%%%%%%%%%%%%%%%%%%%%%%%%%%%%%
%%%%%%% Subsection  %%%%%%%%%%%%%%%%%%%
%%%%%%%%%%%%%%%%%%%%%%%%%%%%%%%%%%%
\subsection{Theorem~\ref{thm:B} parametrized}
\label{sec:ThmB-p}
\unboldmath

Let $\O_1$ be an open subset of $\R\times H_1$
and let $O_2$ be its part in $\R\times H_2$.

\begin{theorem}\label{thm:B-TPhi-parametrized}
Let $\phi\colon O_1\to H_1$ be asymptotically constant parametrized tame.
Let $\Uu:=\{\xi\in W^{1,2}_{H_1}\cap L^2_{H_2}
\mid \text{$(s,\xi_s)\in O_1$, $\forall s\in\R$} \}$.
Then the weak tangent map
\begin{equation}\label{eq:weak-tangent-map}
\begin{split}
   T\Phi
   \colon
   \Uu \times L^2_{H_1}
   &\to
   (W^{1,2}_{H_1}\cap L^2_{H_2})\times L^2_{H_1}
   \\
   (\xi,\eta)
   &\mapsto
   \left(\Phi(\xi),d\Phi|_\xi \eta\right)
\end{split}
\end{equation}
defined by
$$
   \Phi(\xi)(s)
   :=\phi_s(\xi_s)
   ,\qquad
   \left(d\Phi|_\xi \eta\right) (s)
   :=d\phi_s|_{\xi_s} \eta_s ,
$$
is well defined and continuously differentiable
where $\phi_s(x):=\phi(s,x)$.
\end{theorem}

\boldmath
%%%%%%%%%%%%%%%%%%%%%%%%%%%%%%%%%%%
%%%%%%% Subsubsection  %%%%%%%%%%%%%%%%
%%%%%%%%%%%%%%%%%%%%%%%%%%%%%%%%%%%
\subsubsection[Base component]{Base component $\Phi$}
\label{sec:Thm.B-Phi-parametrized}
\unboldmath

\begin{theorem}\label{thm:B-Phi-parametrized}
The base component $\Phi\colon\Uu\to W^{1,2}_{H_1}\cap L^2_{H_2}$,
$\Phi(\xi)(s):=\phi_s(\xi_s)$,
of the map $T\Phi$ in Theorem~\ref{thm:B-TPhi-parametrized}
is well defined and continuously differentiable.
\end{theorem}

\begin{proof}
The proof follows the same scheme as the proof of Theorem~\ref{thm:B-Phi},
just replace $\phi$ by $\phi_s$.
Due to the time dependency of our maps $\phi_s$, some additional terms
arise which we explain how to treat.
\\
We discuss the four steps of the proof of Theorem~\ref{thm:B-Phi}.
We use the abbreviations~(\ref{eq:abbreviate-1}).

\medskip
\noindent
\textbf{Step~1 (Well defined).}

\begin{proof}
The proof consists of two sub-steps A and B.
Namely, one has to show A) that $\Phi(\xi)$
lies in $ L^2_{H_2}\subset L^2_{H_1}$ and B) that its derivative
$\frac{d}{ds}\Phi(\xi)$ lies in $ L^2_{H_1}$.
\\
Step~A. This is precisely the same argument as in Theorem~\ref{thm:B-Phi}
by using that $\phi$ is parametrized tame and therefore asymptotically
constant, in particular zero is mapped to zero asymptotically.
\\
Step~B. Here in
$\frac{d}{ds}(\Phi(\xi))_s=d\phi_s|_{\xi_s} \frac{d}{ds}\xi_s
+\underline{\dot\phi_s(\xi_s)}$
a \underline{new term} arises. But $\dot\phi_s\equiv 0$
whenever $\abs{s}>T$, by parametrized tameness.
Since $\xi$ is in $W^{1,2}_{H_1}$ it is continuous,
and therefore the composed map $s\mapsto \dot\phi_s(\xi_s)$ is
continuous as well.
Since this map vanishes for $\abs{s}\ge T$ it has compact image.
Therefore there exists a constant $c>0$, independent of $s$,
such that $\abs{\dot\phi_s(\xi_s)}_1\le c$.
Hence we estimate
$
   \int_{-\infty}^\infty\abs{\dot\phi_s(\xi_s)}_1^2\, ds
   =\int_{-T}^T\abs{\dot\phi_s(\xi_s)}_1^2\, ds
   \le 2T c^2
$.
The first term in the sum is estimated precisely as in the
unparametrized case Theorem~\ref{thm:B-Phi}.

This finishes the proof of Step~B and Step~1.
\end{proof}

\medskip
\noindent
\textbf{Step~2 (Candidate).}
Given $\xi\in\Uu$, a natural candidate for the derivative is
$$
   \left(d\Phi|_\xi\hat\xi\right)(s)
   =d\phi_s|_{\xi_s}\hat\xi_s
$$
whenever $s\in\R$ and $\hat\xi\in W^{1,2}_{H_1}\cap L^2_{H_2}$.
The map $d\Phi|_\xi$ lies in $\Ll(W^{1,2}_{H_1}\cap L^2_{H_2})$.

\begin{proof}
The proof consists of two sub-steps A and B.
Step~A. This is precisely the same argument as in
Theorem~\ref{thm:B-Phi}.
\\
Step~B. In~(\ref{eq:S2-B}) now arises a third term, namely
$\int_\R\bigl|d\dot\phi_s|_{\xi_s}\hat\xi_s\bigr|_1^2\, ds$.
Since $\phi_s$ is asymptotically constant we have
$d\dot\phi_s\equiv 0$ for $\abs{s}\ge T$.
Since $\xi\colon\R\to H_1$ is continuous there exists a constant $c>0$
such that $\norm{d\dot\phi_s|_{\xi_s}}_{\Ll(H_1)}\le c$.
Therefore $\int_\R\bigl|d\dot\phi_s|_{\xi_s}\hat\xi_s\bigr|_1^2\, ds
\le c^2\int_\R \abs{\hat\xi_s}_1^2\, ds
\le c^2\norm{\hat\xi}_{W^{1,2}_{H_1}\cap L^2_{H_2}}^2$.
Thus in the parametrized case the candidate for the derivative
$d\Phi|_\xi$ still lies in $\Ll(W^{1,2}_{H_1}\cap L^2_{H_2})$.
This proves Step~2.
\end{proof}

\medskip
\noindent
\textbf{Step~3 (Candidate from Step~2 is derivative of \boldmath$\Phi$).}

\begin{proof}
The norm $\interleave \cdot\interleave _{1,2}$
on $W^{1,2}_{H_1}\cap L^2_{H_2}$ is given
by~(\ref{eq:1,2-triple}).

As in~(\ref{eq:S3-Phi}) there are two summands.
In summand 1 there are no $s$-derivatives.
In summand 2 we get the following additional term
in $G(s)$, namely
\begin{equation*}
\begin{split}
   G_2(s):
   &=\dot\phi_s (\xi_s+h\eta_s)
   -\dot\phi_s (\xi_s)
   -h\, d\dot\phi_s|_{\xi_s} \eta_s
\\
   &=\int_0^1 \tfrac{d}{ds} \dot\phi_s(\xi_s+th\eta_s)\, dt
   -h\, d\dot\phi_s|_{\xi_s} \eta_s
\\
   &=h\int_0^1\left(d\dot\phi_s|_{\xi_s+th\eta_s}
   -d\dot\phi_s|_{\xi_s}\right) \eta_s\, dt
\end{split}
\end{equation*}
where we used the fundamental theorem of calculus.
Square this identity and integrate to get
\begin{equation*}
\begin{split}
   &\sup_{\interleave \eta\interleave _{1,2}\le 1}\tfrac{1}{h^2}\int_\R\abs{G_2(s)}_{1}^2 ds
   \\
   &\le
   \sup_{\interleave \eta\interleave _{1,2}\le 1}
   \int_\R \biggl|\left(d\dot\phi_s|_{\xi_s+h\eta_s} -d\dot\phi_s|_{\xi_s}
   \right) \eta_s \biggr|^2_{1} ds .
\end{split}
\end{equation*}
As in~(\ref{eq:hvgvh49}),
for every $\eps>0$ there exists $h_0=h_0(\eps)>0$ such that
\begin{equation}\label{eq:hvgvh49-p}
\begin{split}
   \norm{d\dot\phi_s|_{\xi_s+th\eta_s}
   -d\dot \phi_s|_{\xi_s}}_{\Ll(H_1)}
   &<\eps
\end{split}
\end{equation}
whenever $h\in[0,h_0]$.
Therefore, if $h\in[0,h_0]$, we estimate
\begin{equation*}
\begin{split}
   &\sup_{\interleave \eta\interleave _{1,2}\le 1}
   \int_\R \Abs{\left(d\dot\phi_s|_{\xi_s+h\eta_s} -d\dot\phi_s|_{\xi_s}
   \right) \eta_s}^2_{1} ds
\\
%1
   &\le\sup_{\interleave \eta\interleave _{1,2}\le 1}
   \int_\R \norm{d\dot\phi_s|_{\xi_s+h\eta_s}
     -d\dot\phi_s|_{\xi_s}}_{\Ll(H_1)}^2
   \abs{\eta_s}^2_{1} ds
\\
%2
   &\le \eps^2 \sup_{\interleave \eta\interleave _{1,2}\le 1}
   \int_\R \abs{\eta_s}^2_{1} ds
\\
%5
   &\le \eps^2 .
\end{split}
\end{equation*}
Since $\eps>0$ was arbitrary, the limit of the extra term, as $h\to 0$, is
zero as well.
This proves Step~3.
\end{proof}

\medskip
\noindent
\textbf{Step~4 (Differential is continuous).}
The differential
$$
   d\Phi\colon \Uu\to \Ll(W^{1,2}_{H_1}\cap L^2_{H_2})
   ,\quad
   \xi\mapsto d\Phi|_{\xi}
$$
is continuous.

\begin{proof}
In~(\ref{eq:S4-Phi}) there arises the new term
$
   \int_\R \abs{\bigl(d\dot\phi_s|_{\xi_s}-d\dot\phi_s|_{\tilde\xi_s}\bigr)
   \eta_s}^2_{1} ds
$ which we call term~5.
To estimate term 5 note that by the continuity of the map
$d\dot\phi$, given $\eps>0$ there exists $\delta>0$ such that
$$
   \norm{\tilde\xi-\xi}_{1,2}\le\delta
   \quad\Rightarrow\quad
   \norm{d\dot\phi_s|_{\xi_s}-d\dot\phi_s|_{\tilde\xi_s}}_{\Ll(H_1)}
  \le\eps .
$$
In particular, for every $\tilde\xi$ in the $W^{1,2}$ $\delta$-ball
around $\xi$ it holds that
$$
   \int_\R \abs{\bigl(d\dot\phi_s|_{\xi_s}
   -d\dot\phi_s|_{\tilde\xi_s}\bigr)\eta_s}_{1}^2 ds
   \le\eps^2\int_\R \abs{\eta_s}_{1}^2 \, ds
   \le\eps^2\interleave \eta\interleave _{1,2}^2 .
$$
So the conclusion~(\ref{eq:S4-concl}), hence Step~4,
also holds in the parametrized case.
\end{proof}

This concludes the proof of Theorem~\ref{thm:B-Phi-parametrized}.
\end{proof}

\boldmath
%%%%%%%%%%%%%%%%%%%%%%%%%%%%%%%%%%%
%%%%%%% Subsubsection  %%%%%%%%%%%%%%%%
%%%%%%%%%%%%%%%%%%%%%%%%%%%%%%%%%%%
\subsubsection[Weak tangent map]{Weak tangent map $T\Phi$}
\label{sec:weak-tangent}
\unboldmath

\begin{proof}[Proof of Theorem~\ref{thm:B-TPhi-parametrized}]
The first component map $\xi\mapsto\phi\circ\xi$ is of class $C^1$ by
Theorem~\ref{thm:B-Phi-parametrized}.
The second component is a map
$\Uu\times L^2_{H_1}\to L^2_{H_1}$, $(\xi,\eta)\mapsto d\Phi|_\xi\eta$.
Note that the target space $L^2_{H_1}$ does not involve any
$s$-derivative and takes values in level one $H_1$, and not in the
analytically tricky level two $H_2$.

Line-by-line inspection of the proof of Theorem~\ref{thm:B-TPhi}
shows that no new terms arise.
Thus the proof of Theorem~\ref{thm:B-TPhi}
also proves Theorem~\ref{thm:B-TPhi-parametrized}.
\end{proof}

%\newpage%.\newpage
%%%%%%%%%%%%%%%%%%%%%%%%%%%%%%%%%%%
%%%%%%%%%%%%%%%%%%%%%%%%%%%%%%%%%%%
%%%%%%% Section  %%%%%%%%%%%%%%%%%%%%%%
%%%%%%%%%%%%%%%%%%%%%%%%%%%%%%%%%%%
%%%%%%%%%%%%%%%%%%%%%%%%%%%%%%%%%%%
\section{Coordinate charts}\label{sec:coordinate-charts}

Let $H_2\subset H_1$ be separable Hilbert spaces with a
compact dense inclusion of $H_2$ in $H_1$.
Assume that $(X_1,X_2)$ is a tame $(H_1,H_2)$-two-level manifold;
see Section~\ref{sec:two-level-mfs}.
We denote the unique maximal tame two-level atlas by
$$
   \Aa=\{\psi\colon H_1\supset U_1\to V_1\subset X_1\}
$$
see Remark~\ref{rem:max-atlas}.
Suppose further that $x_-,x_+$ are elements of $X_2$.
The space of paths $\Pp_{x_-x_+}$ will be defined
as a subset of the space of continuous paths
\begin{equation}\label{eq:Cc-4}
   \Cc_{x_-x_+}
   :=\left\{ x\in C^0(\R,X_1)\mid
   \lim_{s\to\mp\infty} x(s)=x_\mp \right\} .
\end{equation}

\smallskip
In Section~\ref{sec:path-spaces} 
we endow the space of paths $\Pp_{x_-x_+}$
with the structure of a $C^1$ Hilbert manifold
modeled on the Hilbert space
$W^{1,2}_{H_1}\cap L^2_{H_2}$ where we abbreviate
$$
   W^{1,2}_{H_1}:=W^{1,2}(\R,H_1)
   ,\qquad
   L^2_{H_2}:=L^2(\R,H_2) .
$$
In Subsection~\ref{sec:loc-par-4}
we define local parametrizations for $\Pp_{x_-x_+}$.
In Subsection~\ref{sec:def-path-space}
we provide an atlas for $\Pp_{x_-x_+}$.
In Section~\ref{sec:transition-maps}
we show that the transition maps are $C^1$ diffeomorphisms.

In Section~\ref{sec:weak-tangent bundle}
we define the weak tangent bundle $\Ee_{x_-x_+}\to \Pp_{x_-x_+}$
and show that it is a $C^1$ Hilbert manifold.

\boldmath
%%%%%%%%%%%%%%%%%%%%%%%%%%%%%%%%%%%
%%%%%%% Subsection  %%%%%%%%%%%%%%%%%%%
%%%%%%%%%%%%%%%%%%%%%%%%%%%%%%%%%%%
\subsection{Path spaces}\label{sec:path-spaces}
\unboldmath

\boldmath
%%%%%%%%%%%%%%%%%%%%%%%%%%%%%%%%%%%
%%%%%%% Subsubsection  %%%%%%%%%%%%%%%%
%%%%%%%%%%%%%%%%%%%%%%%%%%%%%%%%%%%
\subsubsection{Local parametrizations}
\label{sec:loc-par-4}
\unboldmath

\boldmath
%%%%%%%%%%%%%%%%%%%%%%%%%%%%%%%%%%%
%%%%%%% Subsubsection  %%%%%%%%%%%%%%%%
%%%%%%%%%%%%%%%%%%%%%%%%%%%%%%%%%%%
\subsubsection*{Basic paths and basic coverings}
\unboldmath

\begin{definition}\label{def:basic-path-4}
Given $x_-,x_+\in X_2$, consider a $C^2$ path $x\colon\R\to X_2$ such that
$$
   x(s)=
   \begin{cases}
      x_-&\text{, $s\le -T$,}
      \\
      x_+&\text{, $s\ge T$,}
   \end{cases}
$$
for some $T>0$.
Such paths are called \textbf{basic paths}.
\end{definition}

We will construct local charts for $\Pp_{x_-x_+}$
around basic paths.

\begin{definition}[Basic covering]\label{def:basic-covering-4}
Assume that $x\colon\R\to X_2$ from $x_-$ to $x_+$ is a
basic path.\footnote{
  Note that the closure of the path image $x(\R)\subset X_2$ is compact
  and that a basic path $x$ arrives at its endpoints $x_\mp$ already
  in finite time $\mp T$.
  }
A \textbf{basic covering of \boldmath$x$} is an ordered finite collection
$$
   \{\psi_i\colon H_1\supset U^i\to V^i\subset X_1\}_{i=1}^k
   \subset\Aa
$$
of a finite number $k$ of local parametrizations (homeomorphisms)
$\psi_i\colon U^i\to V^i$
whose images cover $x(\R)$ and such that the following is true.
\begin{itemize}\setlength\itemsep{0ex} 
\item[1)]
  There are times
  $
     t_1<t_2<\dots<t_{k-1}
  $
  with the following properties, firstly,
  \begin{equation*}
  \begin{split}
     (-\infty,t_1]&\subset x^{-1}(V^1),\\
     [t_1,t_2]&\subset x^{-1}(V^2),\\
     \vdots\quad &\;\;\vdots\quad\vdots\\
     [t_{k-2},t_{k-1}]&\subset x^{-1}(V^{k-1}),\\
     [t_{k-1},\infty)&\subset x^{-1}(V^k),
  \end{split}
  \end{equation*}
  and, secondly, $x(t_j)\not= x(t_{j+1})$ for every $j=1,\dots,k-2$.
\item[2)]
  It holds that
  $$
     d(\psi_{i+1}^{-1}\circ\psi_i)|_{\psi_i(x(t_i))}=\Id 
     ,\quad
     i=1,\dots,k-1.
  $$
\end{itemize}
\end{definition}

\begin{remark}[Case $k=1$]
There is just one map $\psi_1\colon U^1\to V^1$
and, by 1), we have $-\infty=:t_0<t_1:=\infty$ and
$x^{-1}(V^1)\supset (-\infty,t_1]\cup [t_0,\infty)=\R$.
Thus one coordinate patch $V$ already covers the entire path $x$.
Part 2) is void.
\end{remark}

\begin{figure}[h]
  \centering
  \includegraphics%[%width=0.99\textwidth]
                             %[height=5.5cm]
                             {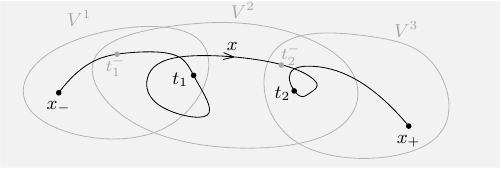}
  \caption{Basic covering of $x$ by $k=3$ local parametrizations
    $\psi_i\colon U^i\to V^i$}
   \label{fig:fig-basic-covering-4}
\end{figure}

\begin{lemma}\label{le:basic-cover-exist-4}
For each basic path $x$ a basic covering exists.
\end{lemma}

\begin{proof}
Choose a finite collection of local parametrizations
$\{\psi_i\colon H_1\supset U^i\to V^i\subset X_1\}_{i\in I}\subset\Aa$ 
which satisfies condition 1).
If $k=1$ we are done. Let $k\ge 2$.

To see that condition 2) can be achieved, too, we modify
our charts inductively as follows.
We replace $U^2$ by $\tilde U^2$ defined as
$$
   \tilde U^2
   := d(\psi_2^{-1}\circ\psi_1)|_{\psi_1(x(t_1))}^{-1} U^1
$$
and consider the modified chart $\tilde\psi_2\colon \tilde U^2\to V^2$
defined by
$$
   \tilde\psi_2
   :=\psi_2\circ
   d(\psi_2^{-1}\circ\psi_1)|_{\psi_1(x(t_1))} .
$$
Note that $\tilde\psi_2$ is indeed a chart, i.e. $\tilde\psi_2\in\Aa$.
To see this observe that by tameness the linear map
$d(\psi_2^{-1}\circ\psi_1)|_{\psi_1(x(t_1))}$ is element of $\Ll(H_1)\cap\Ll(H_2)$.
Since this map is linear its second derivative vanishes, hence
condition~(\ref{eq:tame}) is void. In particular, the composed map
$\tilde\psi_2$ is tame since by Theorem~\ref{thm:tame} tameness is
preserved under composition.
We compute
\begin{equation*}
\begin{split}
   d(\tilde\psi_2^{-1}\circ\psi_1)|_{\psi_1(x(t_1))}
   &=d\left(
   d(\psi_2^{-1}\circ\psi_1)|_{\psi_1(x(t_1))}^{-1}
   \psi_2^{-1}\circ\psi_1
   \right) |_{\psi_1(x(t_1))}
\\
   &=d(\psi_2^{-1}\circ\psi_1)|_{\psi_1(x(t_1))}^{-1}
   d(\psi_2^{-1}\circ\psi_1)|_{\psi_1(x(t_1))}
\\
   &=\Id .
\end{split}
\end{equation*}
Then we modify $U^3$ accordingly and so on.
\end{proof}

\boldmath
%%%%%%%%%%%%%%%%%%%%%%%%%%%%%%%%%%%
%%%%%%% Subsubsection  %%%%%%%%%%%%%%%%
%%%%%%%%%%%%%%%%%%%%%%%%%%%%%%%%%%%
\subsubsection*{Convexity setup on target manifold $X_1$}
\unboldmath

Given a basic path $x\colon\R\to X_2$ from $x_-$ to $x_+$,
let $\{\psi_i\colon H_1\supset U^i\to V^i\subset X_1\}_{i=1}^k$
be a basic covering of $x$. In case $k\ge 2$
we construct for every overlap $V^j\cap V^{j+1}\subset X_1$
a subset which, under both coordinate charts $\psi_j^{-1}$ and
$\psi_{j+1}^{-1}$, has a convex image in $H_1$.
Convexity will be needed for convex interpolation~(\ref{eq:chi_j-4}).
In case $k=1$ there is no overlap and the following is void.

We use the following abbreviations.
For $j=1,\dots,k-1$ we define sets
$$
   U^j_+:=\psi_j^{-1}(V^j\cap V^{j+1})\subset U^j
   ,\qquad
   U^{j+1}_-:=\psi_{j+1}^{-1}(V^j\cap V^{j+1})\subset U^{j+1} ,
$$
as illustrated by Figure~\ref{fig:chart-local-convexity-4}.
\begin{figure}[h]
  \centering
  \includegraphics[width=11cm]    %[width=0.99\textwidth]
                             %[height=5.5cm]
                             {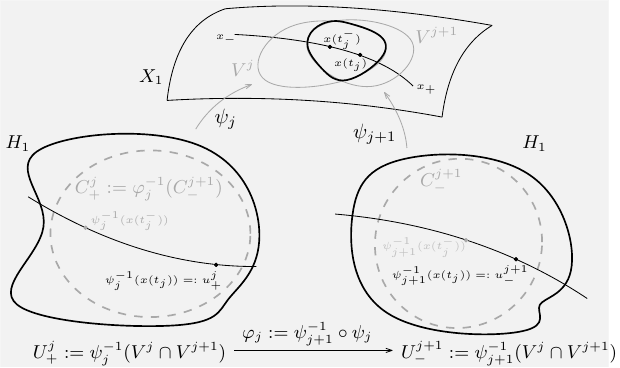}
  \caption{Manifold $X_1$ and convex parametrization domains
                $C^j_+,C^{j+1}_-\subset H_1$}
   \label{fig:chart-local-convexity-4}
\end{figure}

We define \textbf{basic covering transition maps}
(these are $C^2$ since $X_1$ is)
\begin{equation}\label{eq:psi-small-4}
   \varphi_j
   :=\psi_{j+1}^{-1}\circ\psi_j|_{U^j_+}
   \colon H_1\supset U^j_+\to U^{j+1}_-\subset H_1 ,
\end{equation}
and points
$$
   u^j_+:=\psi_j^{-1}(x(t_j))
   \in U^j_+\cap H_2
   ,\qquad
   u^{j+1}_-
   :=\psi_{j+1}^{-1}(x(t_j))
   \in U^{j+1}_-\cap H_2 .
$$
Note that basic paths lie in $X_2$, hence the intersection with $H_2$ above.
Note that by definition of a basic covering we have that
\begin{equation}\label{eq:7687gyuhu-4}
   d\varphi_j|_{u^j_+}
   =\Id .
\end{equation}

\begin{lemma}
For each $j=1,\dots,k-1$ there exists an open
\emph{ball} $C^{j+1}_-$ in $U^{j+1}_-\subset H_1$
centered at $u^{j+1}_-$
satisfying the following conditions
\begin{itemize}\setlength\itemsep{0ex} 
\item[a)]
  the pre-image $C^j_+:=\varphi_j^{-1}(C^{j+1}_-)\subset U^j_+\subset H_1$
  is convex;
\item[b)]
  at every point $q$ of the ball $C^{j+1}_-$ it holds the estimate
    \begin{equation}\label{eq:hgjkgf5322}
     \norm{d(\varphi_j^{-1})|_q-\Id}_{\Ll(H_1)}\le \tfrac12 ;
    \end{equation}
\item[c)]
  for $j\ge 2$ the closures of the subsets
  $C^{j}_-,C^{j}_+\subset U^j\subset H_1$ are disjoint;
\item[d)]
  there exists a constant $c_j>0$ such that
  for all $\eta\in H_2$
  and $y\in C^{j+1}_-\cap H_2$ there holds the inequality
  \begin{equation}\label{eq:(1)}  % (1)
     \abs{d\varphi_j^{-1}|_y\eta-\eta}_2
     \le\tfrac12\abs{\eta}_2+c_j \left(\abs{y}_2+1\right)\abs{\eta}_1 .
  \end{equation}
\end{itemize}
\end{lemma}

\begin{proof}
By choosing the ball $C^{j+1}_-$ small enough, condition a) (and c))
can always be achieved since $\varphi_j^{-1}$ is in $C^2$.
Therefore its derivative is in particular locally Lipschitz and,
moreover, since the derivative is invertible at every point, it is in
particular injective. Thus the two conditions
in~\cite[Thm.\,2.1]{Polyak:2001a} are satisfied
and $\varphi_j^{-1}$ maps a sufficiently small convex set to a convex set.
b) Note that from~(\ref{eq:7687gyuhu-4})
it follows that
\begin{equation}\label{eq:ggjk957}
   d(\varphi_j^{-1})|_{u^{j+1}_-}
   =\Id.
\end{equation}
Hence b) follows by continuity of $d(\varphi_j^{-1})$
possibly after shrinking the ball again.
\\
d) This follows from Lemma~\ref{le:tame-3} as follows.
Since $\varphi_j^{-1}\colon U_-^{j+1}\to U^j_+$ is tame
and using~(\ref{eq:ggjk957})
there exists an open neighborhood $W$ of $u^{j+1}_-$
and a constant $\kappa=\kappa(j)>0$ such that for all
$y\in W\cap H_2$ and $\eta\in H_2$ there holds the estimate
\begin{equation*}
\begin{split}
   \abs{d\varphi_j^{-1}|_y\eta-\eta}_2
   &\le\kappa \left(\abs{y-u_-^{j+1}}_1\abs{\eta}_2
   +\abs{y-u_-^{j+1}}_2\abs{\eta}_1\right)
\\
   &\quad+\tfrac{\kappa}{2}\left(\abs{y}_2+\abs{u_-^{j+1}}_2\right)
   \abs{y-u_-^{j+1}}_1\abs{\eta}_1 .
\end{split}
\end{equation*}
Maybe after shrinking the ball $C^{j+1}_-$ centered at $u^{j+1}_-$
again, we can assume that $C^{j+1}_-\subset W$.
Denote by $\eps$ the radius of the ball $C^{j+1}_-$.
Then for every $y\in C^{j+1}_-\cap H_2$,
and by the triangle inequality,
the above estimate simplifies to
\begin{equation*}
\begin{split}
   \abs{d\varphi_j^{-1}|_y\eta-\eta}_2
   &\le\kappa \eps\abs{\eta}_2
   +\kappa\abs{y}_2\abs{\eta}_1
   +\kappa \abs{u_-^{j+1}}_2\abs{\eta}_1
   +\eps\tfrac{\kappa}{2}\left(\abs{y}_2+\abs{u_-^{j+1}}_2\right)\abs{\eta}_1
\\
  &=\eps\kappa\abs{\eta}_2
   +\left(\kappa+\eps\tfrac{\kappa}{2}\right)\abs{y}_2\abs{\eta}_1
   +\kappa \abs{u_-^{j+1}}_2\left(1+\tfrac{\eps}{2}\right)\abs{\eta}_1.
\end{split}
\end{equation*}
Maybe after shrinking the ball $C^{j+1}_-$ a last time,
we can further assume that $\eps\kappa<1/2$.
Under this assumption the assertion~(\ref{eq:(1)}) follows from the
above inequality for $c_j=\frac{5}{4}\kappa\max\{1,\abs{u^{j+1}_-}_2\}$.
\end{proof}

\boldmath
%%%%%%%%%%%%%%%%%%%%%%%%%%%%%%%%%%%
%%%%%%% Subsubsection  %%%%%%%%%%%%%%%%
%%%%%%%%%%%%%%%%%%%%%%%%%%%%%%%%%%%
\subsubsection*{Interpolation}
\unboldmath

Let $\{\psi_i\colon H_1\supset U^i\to V^i\subset X_1\}_{i=1}^k$
be a basic covering of a basic path $x$ from $x_-$ to $x_+$.
Our goal is to define, on a small neighborhood $\Uu$ of $0$, an
injection
\begin{equation}\label{eq:hjj56g6-4}
   \Psi
   \colon W^{1,2}_{H_1}\cap L^2_{H_2}\supset \Uu
   \to\Cc_{x_- x_+}
   ,\qquad
   \text{$\Psi=\Psi^x$, $\Uu=\Uu^x$,}
\end{equation}
which takes the zero map to the basic path $x$.
We abbreviate $x_s:=x(s)$.

\begin{definition}[Domain $\Uu^x$]\label{def:Uu-n-4}
Let $\{\psi_i\colon H_1\supset U^i\to V^i\subset X_1\}_{i=1}^k$
be a basic covering of a path $x$ from $x_-$ to $x_+$.
If $k=1$ set $t_0:=-\infty$ and $t_1^-:=\infty$.
If $k\ge 2$ choose times $t_j^-$ for $j=1,\dots,k-1$
(see Figure~\ref{fig:chart-local-convexity-4}) such that (i) it holds
$$
   \underbrace{-\infty}_{=: t_0}
   <-T<\overbrace{t_1^-<t_1}^{\text{interpolate}}
   <\overbrace{t_2^- <t_2}^{\text{interpolate}}
   <\dots
   <\overbrace{t_{k-1}^-<t_{k-1}}^{\text{interpolate}}
   <T
   <\underbrace{+\infty}_{=:t_k^-}
$$
and such that (ii) the path $x$ along any \textbf{interpolation interval}
$\overline{[t_j^-,t_j]}$ is via the local coordinate chart $\psi_j^{-1}$ taken
into the convex set $C^j_+\subset H_1$, in symbols
$$
   s\in [t_j^-,t_j]
   \qquad\Rightarrow\qquad
   \psi_j^{-1}(x_s)\in C^j_+ ,
$$
or equivalently 
{\color{gray}(since by definition $C^j_+=\psi_j^{-1}\circ\psi_{j+1}(C^{j+1}_-)$)}
$$
   s\in [t_j^-,t_j]
   \qquad\Rightarrow\qquad
   \psi_{j+1}^{-1}(x_s)\in C^{j+1}_- .
$$
(iii)~Let $B_R(y)$ be the open radius-$R$ ball in $H_1$ centered at $y$.
Fix $R>0$ with
\begin{equation}\label{eq:B_R-C_-}
   s\in [t_j^-,t_j]
   \qquad\Rightarrow\qquad
   B_R(\psi_{j+1}^{-1}(x_s))\subset C^{j+1}_- ,
\end{equation}
whenever $j\in\{1,\dots,k-1\}$ and
$$
   s\in [t_j,t_{j+1}^-]
   \qquad\Rightarrow\qquad
   B_R(\psi_{j+1}^{-1}(x_s))\subset U^{j+1} ,
$$
whenever $j\in\{0,\dots,k-1\}$.
Here use the convention $x(\pm\infty):=x_{\pm \infty}$.
Such $R$ exists since the time intervals are compact,
hence their images in $H_1$ under $\psi_{j+1}^{-1}\circ x$ are
compact, now we use an elementary consideration in set theoretic
topology.\footnote{
  For a compact subset $K$ of an open set $U$ there exists $R>0$ such
  that $B_R(x)\subset U$ for every $x\in K$.
  To see this pick $y\in U$, then by openness there exists $R_y>0$
  such that $B_{R_y}(y)\subset U$. Hence, by the triangle inequality, 
  $z\in B_{R_y/2}(y)$ $\Rightarrow$ $B_{R_y/2}(z)\subset U$.
  By compactness the open cover $\{B_{R_y/2}(y)\mid y\in K\}$ of $K$
  admits a finite subcover, let $R$ be the smallest of these radii.
  }
(Here we profit from our choice to work with basic paths, thus $x$ is constant
outside the \emph{compact} time interval $[-T,T]$.)

\smallskip
\noindent
(iv)~Using $R$ from (iii) we define the domain of the local
parametrization by
\begin{equation}\label{eq:Uu_R}
\boxed{
   \Uu^x=\Uu_R^x:=\left\{\text{$\xi\in W^{1,2}_{H_1}\cap L^2_{H_2}$ such that
   $\abs{\xi_s}_1<R$ $\forall s\in\R$}\right\} .
}
\end{equation}
By definition of $R$ for any $\xi\in\Uu^x$ the following is true
\begin{itemize}\setlength\itemsep{0ex} 
%\item[1.]
%  for $s\in(-\infty,t_1^-]$ we have
%  $\psi_1^{-1}(x_s) +\xi_s\in U^1$
\item[1.]
  for $s\in[t_j^-,t_j]$ we have
  $\psi_j^{-1}(x_s) +\xi_s\in C^j_+$\hfill{\small\color{gray} $j=1,\dots,k-1$}
\item[2.]
  for $s\in[t_j^-,t_j]$ we have
  $\psi_{j+1}^{-1}(x_s) +\xi_s\in C^{j+1}_-$\hfill{\small\color{gray} $j=1,\dots,k-1$}
\item[3.]
  for $s\in[t_j,t_{j+1}^-]$ we have
  $\psi_{j+1}^{-1}(x_s) +\xi_s\in U^{j+1}$\hfill{\small\color{gray} $j=0,\dots,k-1$}
%\item[5.]
%  for $s\in[t_{k-1},\infty)$ we have
%  $\psi_k^{-1}(x_s) +\xi_s\in U^k$.
\end{itemize}
\end{definition}
\begin{figure}[h]
  \centering
  \includegraphics%[%width=0.99\textwidth]
                             %[height=5.5cm]
                             {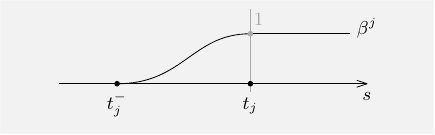}
  \caption{Cutoff function $\beta^j$ along interpolation
                 interval $[t_j^-,t_j]$}
   \label{fig:fig-chart-interpol-cutoff-4}
\end{figure}

\begin{remark}[Case $k=1$]
There is just one map $\psi_1\colon U^1\to V^1$. Only~3.
is non-void. It yields $\psi_1^{-1}(x_s)+\xi_s\in U^1$
$\forall s\in (-\infty,\infty)$.
\end{remark}

\begin{definition}[Local parametrization $\Psi^x$ near basic path $x$]
In Definition~\ref{def:Uu-n-4}
for each $j=1,\dots,k-1$ pick a monotone smooth cutoff function
$\beta^j\colon\R\to[0,1]$ such that $\beta^j\equiv 0$ on $(-\infty,t_j^-]$
and $\beta^j\equiv 1$ on $[t_j,\infty)$;
see Figure~\ref{fig:fig-chart-interpol-cutoff-4}.
For $\xi\in\Uu^x$ we define a
\textbf{local parametrization} 
\begin{equation}\label{eq:loc-param-Psi}
   \Psi\colon W^{1,2}_{H_1}\cap L^2_{H_2}\supset \Uu\to\Cc_{x_-x_+}
   ,\qquad
   \text{$\Psi=\Psi^x$, $\Uu=\Uu^x$,}
\end{equation}
centered at the basic path $x$ from $x_-$ to $x_+$ by
\begin{equation}\label{eq:loc-par-form-4}
\boxed{
   \left(\Psi(\xi)\right)(s)
   :=\begin{cases}
%      \psi_1\bigl(\psi_1^{-1}(x_s)+\xi_s\bigr)
%         &\text{, $s\in (-\infty,t_1^-]$}
%   \\
      \chi_j(s;\xi)&\text{, $s\in[t_j^-,t_j]$, {\color{gray}\small$j=1,\dots,k-1$}}
   \\
      \psi_{j+1}\bigl(\psi_{j+1}^{-1}(x_s)+\xi_s\bigr)
         &\text{, $s\in [t_j,t_{j+1}^-]$, {\color{gray}\small$j=0,\dots,k-1$}}
   \\
%      \psi_k\bigl(\psi_k^{-1}(x_s)+\xi_s\bigr)
%         &\text{, $s\in [t_{k-1},\infty)$.}
%   \\
  \end{cases}
}
\end{equation}
for every $s\in\R$.
Here the map $\chi_j(s;\xi)$, for $j=1,\dots,k-1$,
is defined by convex interpolation in the convex sets $C^j_+$ and $C^{j+1}_-$
and for $s\in\R$ as follows
\begin{equation}\label{eq:chi_j-4}
\begin{split}
   &\chi_j(s;\xi)\\
   :&=\psi_j\biggl(
   (1-\beta^j_s)\Bigl(\underbrace{\psi_j^{-1}(x_s)+\xi_s}_{\in C^j_+}\Bigr)
   +\beta^j_s\underbrace{
   \overbrace{\psi_j^{-1}\circ\psi_{j+1}}^{\stackrel{\text{(\ref{eq:psi-small-4})}}{=}\varphi_j^{-1}}
    \Bigl(\overbrace{\psi_{j+1}^{-1}(x_s)+\xi_s}^{\in C^{j+1}_-}\Bigr)}
      _{\in\varphi_j^{-1}(C^{j+1}_-)\subset C^j_+}
   \biggr)\\
   &=\psi_j\circ \Ss_s^j (\xi_s)
\end{split}
\end{equation}
where we abbreviated the map~(\ref{eq:A.3.2-S}) for
$\varphi=\varphi_j^{-1}\in C^2$ and $x_0=\psi^{-1}_{j+1}(x_s)$ by
\begin{equation}\label{eq:def-S-4}
   \Ss_s^j (\xi_s)
   :=\Ss^{\varphi^{-1}_j}_{\beta^j_s,\psi^{-1}_{j+1}(x_s)}(\xi_s) .
\end{equation}
\end{definition}

By Remark~\ref{rem:A.3.2.}, cf. Figure~\ref {fig:chart-no-exp-44},
the \textbf{interpolation map}
\begin{equation}\label{eq:def-S-44}
   \Ss_s^j\colon H_1\supset B_R(0)=:B_R\to H_1
\end{equation}
is a $C^2$ diffeomorphism onto its image,
in particular it is injective.

\begin{remark}[Case $k=1$]
There is just one map $\psi_1\colon U^1\to V^1$ and
$$
   \left(\Psi(\xi)\right)(s)
   =\psi_1\underbrace{\bigl(\psi_1^{-1}(x_s)+\xi_s\bigr)}_{\in U^1}\in V^1
   ,\quad
   \forall s\in (-\infty,\infty).
$$
\end{remark}

\begin{proposition}\label{prop:injective-4}
The map $\Psi\colon\Uu\to\Cc_{x_-x_+}$ is injective and $\Psi(0)=x$
reproduces the basic path $x$ used in the construction of $\Psi$.
\end{proposition}

\begin{proof}
We first check that $\Psi$ is injective.
Hence assume that there exist $\xi$ and $\tilde\xi$ such that
$\Psi(\xi)=\Psi(\tilde \xi)$.
In particular $\Psi(\xi)(s)=\Psi(\tilde \xi)(s)$ for every $s\in\R$.
We show that this implies $\xi(s)=\tilde \xi(s)$ for every $s\in\R$.
The only times where this is not obvious
is when $s$ lies in an interpolation interval $[t_j^-,t_j]$
for $j=1,\dots,k-1$. In this case by~(\ref{eq:chi_j-4}) there are the identities
$$
   \psi_j\circ \Ss_s^j(\xi_s)
   =\chi_j(s;\xi)
   =\Psi(\xi)(s)
   =\Psi(\tilde\xi)(s)
   =\chi_j(s;\tilde\xi)
   =\psi_j\circ \Ss_s^j(\tilde\xi_s) .
$$
But $\psi_j$ is injective, since it is a chart map,
and the map $\Ss_s^j$ is a diffeomorphism, as explained in
Appendix~\ref{sec:quantitative-IFT}.
Hence $\xi_s=\tilde\xi_s$ for every $s\in\R$,
that is $\xi=\tilde\xi$. This finishes the proof that $\Psi$
is injective.

Now we check that $\Psi(0)$ is a basic path.
We call $I=\cup_{j=1}^{k-1}[t_j^-,t_j]$ the \textbf{interpolation region}.
If $s\in\R\setminus\I$, i.e. $s$ is in the non-interpolation region,
then $\Psi(0)(s)=x(s)$ is immediate from~(\ref{eq:loc-par-form-4}).
If $s\in I$, i.e. there exists a $j\in\{1,\dots,k-1\}$ such that
$s\in[t_j^-,t_j]$, then by definition of $\Ss_s^j$,
see~(\ref{eq:chi_j-4}), we get
\begin{equation}\label{eq:Ss-0}
   \Ss_s^j(0)
   =\biggl((1-\beta^j_s)\psi_j^{-1}(x_s)
   +\beta^j_s\psi_j^{-1}\circ\psi_{j+1}\circ
    \psi_{j+1}^{-1}(x_s)\biggr)
   =\psi_j^{-1}(x_s) .
\end{equation}
Therefore the interpolation map for $\xi_s=0$ reproduces the basic
path, namely
\begin{equation*}
   \chi_j(s;0)
   =\psi_j\circ \Ss_s^j(0)
   =x_s
\end{equation*}
for every $s\in\R$.
This proves Proposition~\ref{prop:injective-4}.
\end{proof}

\boldmath
%%%%%%%%%%%%%%%%%%%%%%%%%%%%%%%%%%%
%%%%%%% Subsubsection  %%%%%%%%%%%%%%%%
%%%%%%%%%%%%%%%%%%%%%%%%%%%%%%%%%%%
\subsubsection{Definition of path space}\label{sec:def-path-space}
\unboldmath

\begin{definition}
Given $x_-,x_+\in X_2$, we define the path space $\Pp_{x_-x_+}$
as the subset of $\Cc_{x_-x_+}$ consisting of the images
of all local parametrizations $\Psi$ constructed in~(\ref{eq:loc-param-Psi})
above. We denote the set of all local parametrizations centered at a
basic path $x\colon\R\to X_2$ from $x_-$ to $x_+$ by
\begin{equation}\label{eq:AP}
   \Aa\Pp_{x_-x_+}
   :=\{\text{local parametrizations
   $\Psi^x\colon W^{1,2}_{H_1}\cap L^2_{H_2}\supset\Uu^x\to\Cc_{x_-x_+}$}\} .
\end{equation}
Then the \textbf{space of paths} in $X=X_1$ from $x_-$ to $x_+$ is
defined by
$$
   \Pp_{x_-x_+}
   :=\bigcup_{\Psi^x\in \Aa\Pp_{x_-x_+}} \Psi(\Uu)
   ,\qquad
   \text{\color{gray}$\Psi=\Psi^x$, $\Uu=\Uu^x$.}
$$
Note that $\Pp_{x_-x_+}\subset\Cc_{x_-x_+}$.
We define a \textbf{topology} on the set $\Pp_{x_-x_+}$ as follows.
A subset $\Vv$ of $\Pp_{x_-x_+}$ is \textbf{open}
if and only if $\Psi^{-1}(\Vv\cap\Psi(\Uu))$ is open in $\Uu$
for all $\Psi\in\Aa\Pp_{x_-x_+}$.
\end{definition}

\begin{theorem}[$C^1$ atlas]\label{thm:path-mf-4}
The set of local parametrizations $\Aa\Pp_{x_-x_+}$ is a $C^1$ atlas
for the path space $\Pp_{x_-x_+}$, in particular $\Pp_{x_-x_+}$ is a $C^1$
Hilbert manifold modeled on the Hilbert space $W^{1,2}_{H_1}\cap L^2_{H_2}$
with inner product
$$
   \langle\langle\langle
   \xi,\eta
   \rangle \rangle \rangle
%   \INNER{\xi}{\eta}^{1,2,change}
   :=\int_{-\infty}^\infty \INNER{\dot\xi(s)}{\dot\eta(s)}_{H_1} ds
   +\int_{-\infty}^\infty \INNER{\xi(s)}{\eta(s)}_{H_2} ds .
$$
\end{theorem}

\begin{proof}
To prove Theorem~\ref{thm:path-mf-4} we need to show that the transition
maps are $C^1$ diffeomorphisms.
This will be carried out in the next section and
Theorem~\ref{thm:path-mf-4} follows from Theorem~\ref{thm:Psi-diffeo-4}.
\end{proof}

\boldmath
%%%%%%%%%%%%%%%%%%%%%%%%%%%%%%%%%%%
%%%%%%% Subsubsection  %%%%%%%%%%%%%%%%
%%%%%%%%%%%%%%%%%%%%%%%%%%%%%%%%%%%
\subsubsection{Transition maps}\label{sec:transition-maps}
\unboldmath

We show that transition maps are $C^1$ diffeomorphisms
on the Hilbert space $W^{1,2}_{H_1}\cap L^2_{H_2}$.
Assume that $x$ and $\tilde x$ are basic paths from $x_-$ to $x_+$.
Let
$$
   T
   :=\max\{T_x, T_{\tilde x}\} >0
$$
be the maximum of the two times
that come with the basic paths $x$ and $\tilde x$, respectively; see
Definition~\ref{def:basic-path-4}.

Pick a basic covering of $x$,
notation $\{\psi_i\colon H_1\supset U^i\to V^i\subset X_1\}_{i=1}^k$,
and a basic covering of $\tilde x$, notation $\{\tilde \psi_i\colon
H_1\supset \tilde U^i\to \tilde V^i\subset X_1\}_{i=1}^{\tilde k}$.

For the basic coverings of $x$ and $\tilde x$
choose $t_1,\dots,t_{k-1}$, respectively $\tilde t_1,\dots,\tilde t_{\tilde k-1}$
according to 1) and 2) in Definition~\ref{def:basic-covering-4}
for the common $T$.
For $j=1,\dots,k-1$, respectively $\tilde j=1,\dots,\tilde k-1$,
choose open neighborhoods $C^j_+$ of $u^j_+$ in $U^j_+$,
respectively $\tilde C^{\tilde j}_+$ of $\tilde u^{\tilde j}_+$ in
$\tilde U^{\tilde j}_+$,
satisfying a) b) c) right after~(\ref{eq:7687gyuhu-4}).
In addition, choose disjoint interpolation intervals
$[t_1^-,t_1]$, $\dots$ , $[t_{k-1}^-,t_{k-1}]$, respectively
$[\tilde t_1^-,\tilde t_1]$, $\dots$ , $[\tilde t_{\tilde k-1}^-,\tilde t_{\tilde k-1}]$,
satisfying (i) and (ii) in Definition~\ref{def:Uu-n-4}.
Now choose open subsets $\Uu=\Uu_R^x$, respectively
$\tilde\Uu=\Uu_{\tilde R}^{\tilde x}$,
of $W^{1,2}_{H_1}\cap L^2_{H_2}$ as in~(\ref{eq:Uu_R}).

\smallskip
After choosing cutoff functions
$\beta^1,\dots,\beta^{k-1}$, respectively
$\tilde\beta^1,\dots,\tilde\beta^{\tilde k-1}$,
for the interpolation intervals, see
Figure~\ref{fig:fig-chart-interpol-cutoff-4},
we define local parametrizations
\begin{equation*}\label{eq:hjj5hjhjk6g6-4}
   \Psi=\Psi^x\colon\Uu\to\Pp_{x_- x_+}
   ,\quad
   \tilde\Psi=\Psi^{\tilde x}\colon\tilde\Uu\to\Pp_{x_- x_+}
   ,\quad
   0\in\Uu,\tilde\Uu\subset W^{1,2}_{H_1}\cap L^2_{H_2} ,
\end{equation*}
by formula~(\ref{eq:loc-par-form-4}) for these choices.
In particular, definition~(\ref{eq:loc-par-form-4}) for
$\Psi^{\tilde x}$
involves $\tilde\chi_{\tilde j}=\chi_{\tilde j}^{\tilde x}$ defined by~(\ref{eq:chi_j-4}) with
$\sim$\,-\,quantities on the right hand side,~e.g.
\begin{equation}\label{eq:Ss-tilde}
   \tilde{\Ss}_s^{\tilde j}
   :=\Ss^{\tilde\varphi^{-1}_{\tilde j}}_{\tilde \beta^{\tilde j}_s,
   \tilde\psi^{-1}_{\tilde j+1}(\tilde x_s)}
   ,\quad \text{with $\Ss$ as in~(\ref{eq:A.3.2-S}).}
\end{equation}
\begin{definition}[Path space transition map]
We abbreviate
$\Vv:=\Psi(\Uu)$ and $\tilde \Vv:=\tilde \Psi(\tilde \Uu)$
and $\Uu_0:=\Psi^{-1}(\Vv\cap\tilde\Vv)$
and $\tilde \Uu_0:=\tilde \Psi^{-1}(\Vv\cap\tilde\Vv)$.
As illustrated by Figure~\ref{fig:chart-no-exp-4}, we define
the corresponding \textbf{path space transition map} by
\begin{equation}\label{eq:Psi-def-4}
\boxed{
   \Phi:=\tilde\Psi^{-1}\circ\Psi|_{\Uu_0}\colon\Uu_0\to\tilde\Uu_0 ,
}
   \qquad
   \Uu_0, \tilde \Uu_0 \subset W^{1,2}_{H_1}\cap L^2_{H_2}\ .
\end{equation}
In view of Proposition~\ref{prop:injective-4}
the map $\Phi\colon\Uu_0\to\tilde\Uu_0$ is a bijection
with inverse
\begin{equation}\label{eq:Phi-inverse}
\boxed{
   \Phi^{-1}=\Psi^{-1}\circ\tilde
   \Psi|_{\Uu_0}\colon\tilde\Uu_0\to\Uu_0 .
}
\end{equation}
\end{definition}
\begin{figure}[h]
  \centering
  \includegraphics%[%width=0.99\textwidth]
                             [height=4.5cm]
                             {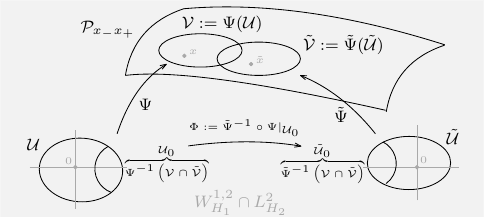}
  \caption{Transition map $\Phi$ for path space $\Pp_{x_-x_+}$ modeled
                 on $W^{1,2}_{H_1}\cap L^2_{H_2}$}
   \label{fig:chart-no-exp-4}
\end{figure}

\begin{theorem}\label{thm:Psi-diffeo-4}
The transition map $\Phi\colon \Uu_0\to \tilde \Uu_0$ is a $C^1$
diffeomorphism.
\end{theorem}

\begin{proof}
The proof has ten steps.

Denote the \textbf{interpolation region} for $x$, respectively $\tilde x$, by
$$
   I=\bigcup_{j=1}^{k-1}[t_j^-,t_j]
   ,\qquad
   \tilde I=\bigcup_{\tilde j=1}^{\tilde k-1}[\tilde t_{\tilde j}^-,\tilde t_{\tilde j}] .
$$
As in Definition~\ref{def:Uu-n-4} we use the 
conventions
$t_0:=-\infty=:\tilde t_0$ and $t_k^-:=\infty=:\tilde t_{\tilde k}^-$.
%$$
%   t_0:=-\infty=:\tilde t_0
%   ,\qquad
%   t_k^-:=\infty=:\tilde t_{\tilde k}^- .
%$$

\medskip\noindent
\textbf{Step 1.}
Pick $\xi\in\Uu_0$. Set $\eta:=\Phi(\xi)$.
By definition of $\Phi$ we have the identity
\begin{equation}\label{eq:Phi-4}
   \tilde\Psi(\eta)(s)=\Psi(\xi)(s)
%   ,\qquad
%   \eta(s)
%   =\left(\tilde\Psi^{-1}\circ \Psi(\xi)\right)(s)
%   =\Phi(\xi)(s) ,
\end{equation}
for every $s\in\R$.
We solve~(\ref{eq:Phi-4}) for
$\eta_s:=\eta(s)=\Phi(\xi)(s)$ considering four cases.

\begin{proof}%[Proof of Step~1.]
To prove Step~1 we consider four cases. 

\smallskip\noindent
\textbf{Case 1.}
Suppose $s\in \R\setminus(I\cup \tilde I)$,
that is $s$ is a non-interpolation time for both.
\newline
Then there is $j\in\{1,\dots,k\}$ such that $s$ lies in the
non-interpolation interval $(t_{j-1},t_j^-)$.
Moreover,  there is $\tilde j\in\{1,\dots,\tilde k\}$ such that $s$ lies in the
non-interpolation interval $(\tilde t_{\tilde j-1},\tilde t_{\tilde j}^-)$.
The two sides of~(\ref{eq:Phi-4}) take on the form
$$
   \tilde\psi_{\tilde j}\left(\tilde\psi_{\tilde j}^{-1}(\tilde x_s)+\eta_s\right)
   =\psi_j\left(\psi_j^{-1}(x_s)+\xi_s\right)
$$
which we resolve for
$$
   \eta_s
   =\tilde\psi_{\tilde j}^{-1}\circ\psi_j\left(\psi_j^{-1}(x_s)+\xi_s\right)
   -\tilde\psi_{\tilde j}^{-1}(\tilde x_s).
$$
%
%\smallskip\noindent
\textbf{Case 2.}
Suppose $s\in I\cap (\R\setminus\tilde I)$.
\newline
There is $j\in\{1,\dots,k-1\}$ such that $s$ is in the
interpolation interval $[t_j^-,t_j]$. There is 
$\tilde j\in\{1,\dots,\tilde k\}$ such that $s$ lies in the
non-interpolation interval $(\tilde t_{\tilde j-1},\tilde t_{\tilde j}^-)$.
The two sides of~(\ref{eq:Phi-4}), using~(\ref{eq:def-S-4}) in the end, take on the form
\begin{equation*}
\begin{split}
   &\tilde\psi_{\tilde j}\left(\tilde\psi_{\tilde j}^{-1}(\tilde x_s)+\eta_s\right)\\
   &=\chi_j(s;\xi)\\
   :&=\psi_j\biggl(
   (1-\beta^j_s)\Bigl(\psi_j^{-1}(x_s)+\xi_s\Bigr)
   +\beta^j_s\underline{\psi_j^{-1}\circ\psi_{j+1}}
    \Bigl(\underline{\psi_{j+1}^{-1}}(x_s)+\xi_s\Bigr)
   \biggr)\\
   &=\psi_j\biggl(
   (1-\beta^j_s)\Bigl(\psi_j^{-1}(x_s)+\xi_s\Bigr)
   +\beta^j_s\varphi_j^{-1}
    \Bigl(\varphi_j\circ\psi_j^{-1}(x_s)+\xi_s\Bigr)
   \biggr)\\
   &=\psi_j\circ \Ss_s^j(\xi_s) .
\end{split}
\end{equation*}
Observe that $\psi_j^{-1}\psi_{j+1}=\varphi_j^{-1}$ is the
basic covering transition map~(\ref{eq:psi-small-4}).
Then
$$
   \eta_s
   =\tilde\psi_{\tilde j}^{-1}\circ\psi_j\circ\Ss_s^j(\xi_s)
   -\tilde\psi_{\tilde j}^{-1}(\tilde x_s) .
$$
%
%\smallskip\noindent
\textbf{Case 3.}
Suppose $s\in (\R\setminus I)\cap \tilde I$.
\newline
There is 
$j\in\{1,\dots,k\}$ such that $s$ lies in the
non-interpolation interval $(t_{j-1}, t_j^-)$. There is 
$\tilde j\in\{1,\dots,\tilde k-1\}$ such that $s$ lies in the
interpolation interval $[\tilde t_{\tilde j}^-,\tilde t_{\tilde j}]$.
The two sides of~(\ref{eq:Phi-4}) take on the form
\begin{equation*}
\begin{split}
   \tilde\psi_{\tilde j}\biggl(
   (1-\tilde\beta^{\tilde j}_s)\Bigl(\tilde\psi_{\tilde j}^{-1}(\tilde x_s)+\eta_s\Bigr)
   +\tilde\beta^{\tilde j}_s \tilde\varphi_{\tilde j}^{-1}
   \Bigl(
      \tilde\varphi_{\tilde j}\circ\tilde\psi_{\tilde j}^{-1}(\tilde x_s)+\eta_s
   \Bigr)
   \biggr)
   &=:\tilde \chi_{\tilde j}(s;\eta)\\
   &=\psi_j\left(\psi_j^{-1}(x_s)+\xi_s\right)
\end{split}
\end{equation*}
where $\tilde\varphi_{\tilde j}$ is the basic covering transition map
from~(\ref{eq:psi-small-4}). Equivalently we have
\begin{equation*}
\begin{split}
   (1-\tilde\beta^{\tilde j}_s)\Bigl(\tilde\psi_{\tilde j}^{-1}(\tilde x_s)+\eta_s\Bigr)
   +\tilde\beta^{\tilde j}_s \tilde\varphi_{\tilde j}^{-1}
    \Bigl(
       \tilde\varphi_{\tilde j}\circ\tilde \psi_{\tilde j}^{-1}(\tilde x_s)+\eta_s
    \Bigr)
   &=\tilde\psi_{\tilde j}^{-1}\circ\psi_j\left(\psi_j^{-1}(x_s)+\xi_s\right).
\end{split}
\end{equation*}
We wish to resolve for $\eta_s$. With $\tilde\Ss_s^{\tilde j}$
as in~(\ref{eq:Ss-tilde}) we obtain the identity
\begin{equation*}%\label{eq:def-tilde-S-4}
   \tilde\Ss_s^{\tilde j} (\eta_s)
%   =\Ss^{\tilde\varphi^{-1}_{\tilde j}}_{\tilde\beta^{\tilde
%       j}_s,\tilde\psi^{-1}_{\tilde j}(x_s)}(\eta_s)
   =\tilde\psi_{\tilde j}^{-1}\circ \psi_j\left(\psi_j^{-1}(x_s)+\xi_s\right).
\end{equation*}
But $\tilde\Ss_s^{\tilde j}$ is invertible, by Remark~\ref{rem:A.3.2.},
and we obtain the formula
$$
   \eta_s=(\tilde\Ss_s^{\tilde j})^{-1}
   \left(
   \tilde\psi_{\tilde j}^{-1}\circ \psi_j\left(\psi_j^{-1}(x_s)+\xi_s\right)
   \right).
$$
\smallskip\noindent
\textbf{Case 4.}
Suppose $s\in I\cap \tilde I$.
\newline
There is 
$j\in\{1,\dots,k-1\}$ such that $s$ lies in the
interpolation interval $[t_j^-,t_j]$.
There is  $\tilde j\in\{1,\dots,\tilde k-1\}$ such that $s$ lies in the
interpolation interval $[\tilde t_{\tilde j}^-,\tilde t_{\tilde j}]$.
The two sides of~(\ref{eq:Phi-4}) take on the form
\begin{equation*}
\begin{split}
   &\tilde\psi_{\tilde j}\biggl(
   (1-\tilde\beta^{\tilde j}_s)\Bigl(\tilde\psi_{\tilde j}^{-1}(\tilde x_s)+\eta_s\Bigr)
   +\tilde\beta^{\tilde j}_s \tilde\varphi_{\tilde j}^{-1}
    \Bigl(\tilde\varphi_{\tilde j}\circ\tilde \psi_{\tilde j}^{-1}(\tilde x_s)+\eta_s\Bigr)
   \biggr)\\
   &=:\tilde \chi_{\tilde j}(s;\eta)\\
   &=\chi_j(s;\xi)\\
   :&=\psi_j\biggl(
   (1-\beta^j_s)\Bigl(\psi_j^{-1}(x_s)+\xi_s\Bigr)
   +\beta^j_s\varphi_j^{-1}
    \Bigl(\varphi_j\circ\psi_j^{-1}(x_s)+\xi_s\Bigr)
   \biggr) .
\end{split}
\end{equation*}
Similarly as in Case~3 with $\tilde\Ss_s^{\tilde j}$ as in~(\ref{eq:Ss-tilde})
we resolve for $\eta_s$, namely
$$
   \tilde\Ss_s^{\tilde j} (\eta_s)
   =\tilde\psi_{\tilde j}^{-1}\circ\psi_j\circ
   \Ss_s^j (\xi_s)
   ,\qquad
   \eta_s
   =(\tilde\Ss_s^{\tilde j})^{-1}\circ \tilde\psi_{\tilde j}^{-1}\circ\psi_j\circ
   \Ss_s^j (\xi_s) .
$$
This concludes the proof of Step~1.
\end{proof}

\smallskip
\noindent
\textbf{Step 2.}
%The map $\Phi\colon\Uu_0\to\tilde \Uu_0$
%defined by~(\ref{eq:Psi-def-4}) is $C^1$.
We find a subset $O\subset\R\times H_1$ and a map
$
   \varphi\colon\R\times H_1\supset O\to H_1
$
such that we can write
\begin{equation}\label{eq:Psi-map-proof-4}
\boxed{
   (\Phi(\xi))(s)
   =\varphi(s,\xi(s))=:\varphi_s(\xi(s))
}
\end{equation}
for all $s\in\R$ and $\xi\in\Uu_0$, see Figure~\ref{fig:chart-no-exp-4}.

\begin{proof}%[Proof of Step~2.]
To prove Step~2 we first describe the slices $U^s$ of $O$ such that
$$
   O=\bigcup_{s\in\R}\left(\{s\}\times U^s\right) .
$$
We consider the same four cases as in Step~1.
The map $\Ss_s^j$ in the composition
$$
   H_1\supset B_R:=B_R(0)\stackrel{\Ss_s^j}{\longrightarrow}
   U^j \stackrel{\psi^j}{\longrightarrow}
   V^j\subset X_1
$$
is defined in~(\ref{eq:def-S-44}).

\smallskip\noindent
\textbf{Case 1.}
For $s\in \R\setminus(I\cup \tilde I)$ let
$$
   U^s
   :=\psi_j^{-1}
   \left(
   \psi_j(B_R(u_s^j))\cap\tilde{\psi}^{\tilde j}(B_{\tilde R}(\tilde{u}_s^{\tilde j}))
   \right)
   -u_s^j
   ,\qquad
\boxed{
   u_s^j=\psi_j^{-1}(x_s) .
}
$$

\smallskip\noindent
\textbf{Case 2.}
For $s\in I\cap (\R\setminus\tilde I)$ let
$$
   U^s
   :=(\Ss_s^j)^{-1}\circ\psi_j^{-1}
   \left(
   \left(\psi_j\circ\Ss_s^j\right)(B_R)
   \cap
   \tilde{\psi}^{\tilde j}(B_{\tilde R}(\tilde{u}_s^{\tilde j})
   \right) .
$$

\smallskip\noindent
\textbf{Case 3.}
For $s\in (\R\setminus I)\cap \tilde I$ let
$$
   U^s
   :=\psi_j^{-1}
   \left(
   \psi_j(B_R(u_s^j))
   \cap
   \bigl(\tilde\psi_{\tilde j}\circ\tilde\Ss_s^{\tilde j}\bigr)(B_{\tilde R})
   \right)
   -u_s^j .
$$

\smallskip\noindent
\textbf{Case 4.}
For $s\in I\cap \tilde I$ let
$$
   U^s:=
   (\Ss_s^j)^{-1}\circ\psi_j^{-1}
   \left(
   \left(\psi_j\circ\Ss_s^j\right)(B_R)
   \cap
   \bigl(\tilde\psi_{\tilde j}\circ\tilde\Ss_s^{\tilde j}\bigr)(B_{\tilde R})
   \right) .
$$

\smallskip\noindent
For $s\in\R$ and $v\in U^s$
we define (juxtaposition means composition)
\begin{equation}\label{eq:rho-4}
\boxed{
   \varphi(s,v)
   :=\begin{cases}
      \tilde\psi_{\tilde j}^{-1}\psi_j\left(\psi_j^{-1}(x_s)+v\right)
      -\tilde\psi_{\tilde j}^{-1}(\tilde x_s)
         &\text{\small, $s\in (t_{j-1},t_j^-)\cap (\tilde t_{\tilde j-1},\tilde t_{\tilde j}^-)$,}
   \\
      \tilde\psi_{\tilde j}^{-1}\psi_j\Ss_s^j(v)
      -\tilde\psi_{\tilde j}^{-1}(\tilde x_s)
         &\text{\small, $s\in [t_j^-,t_j]\cap (\tilde t_{\tilde j-1},\tilde t_{\tilde j}^-)$,}
   \\
      (\tilde\Ss_s^{\tilde j})^{-1}
      \left(
      \tilde\psi_{\tilde j}^{-1} \psi_j\left(\psi_j^{-1}(x_s)+v\right)
      \right)
         &\text{\small, $s\in (t_{j-1}, t_j^-)\cap [\tilde t_{\tilde j}^-,\tilde t_{\tilde j}]$,}
   \\
      (\tilde\Ss_s^{\tilde j})^{-1} \tilde\psi_{\tilde j}^{-1}\psi_j
      \Ss_s^j (v)
         &\text{\small, $s\in[t_j^-,t_j]\cap [\tilde t_{\tilde j}^-,\tilde t_{\tilde j}]$.}
  \end{cases}
}
\end{equation}
For $s\le-T$ it holds $s\in (t_0,t_1^-)\cap (\tilde t_0,\tilde t_1^-)$,
so by the first case of~(\ref{eq:rho-4}) we get
\begin{equation}\label{eq:as-cnst_-}
   \varphi(s,v)
   =\tilde\psi_1^{-1}\psi_1\left(\psi_1^{-1}(x_-)+v\right)
      -\tilde\psi_1^{-1}(x_-)
   ,\qquad
   \varphi(s,0)=0 .
\end{equation}
Similarly for $s\ge T$ we get
\begin{equation}\label{eq:as-cnst_+}
   \varphi(s,v)
   =\tilde\psi_{\tilde k}^{-1}\psi_k\left(\psi_k^{-1}(x_+)+v\right)
      -\tilde\psi_{\tilde k}^{-1}(x_+)
   ,\qquad
   \varphi(s,0)=0 .
\end{equation}
Figure~\ref{fig:chart-no-exp-44} illustrates
the definition of $\varphi_s:=\varphi(s,\cdot)$ in case 2. The other three cases are
similar, in the figure one just needs to interchange the interpolation
maps $\Ss_s^j$ and the translation maps $T$.

By construction of the map $\varphi$ the identity~(\ref{eq:Psi-map-proof-4}) holds.
This concludes the proof of Step~2.
\end{proof}
\begin{figure}[h]
  \centering
  \includegraphics%[%width=0.99\textwidth]
                             [height=6.5cm]
                             {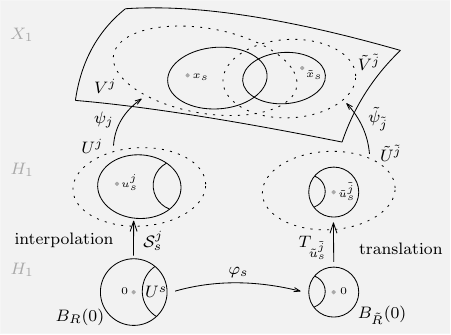}
  \caption{The map $\varphi_s\colon H_1\supset U^s\to H_1$
                 in~(\ref{eq:rho-4}) in case 2}
   \label{fig:chart-no-exp-44}
\end{figure}

\smallskip\noindent
\textbf{Outlook on the next few steps.}
In the next few steps we show that $\varphi$ is parametrized tame,
see Definition~\ref{def:tame-parametrized}.
This then allows us to conclude with the help
of Theorem~\ref{thm:B-Phi-parametrized}
that $\Phi$ in~(\ref{eq:Psi-def-4}) is $C^1$.

\medskip\noindent
\textbf{Step 3 (Tameness).}
The interpolation map $\Ss_s^j\colon H_1\supset B_R\to H_1$,
see~(\ref{eq:def-S-44}) and Figure~\ref{fig:chart-no-exp-44},
is tame and has a tame inverse $(\Ss_s^j)^{-1}\colon \Ss_s(B_R)\to B_R$,
$\forall j$.

\begin{proof}
The proof of Step~3 is a rather lengthy undertaking.

\smallskip\noindent
\textbf{Tameness of \boldmath$\Ss_s^j$.}
Since $\varphi_j^{-1}$ is tame, the map $\Ss_s^j$ is tame as well,
in particular of class $C^2$ on $H_1$ and on $H_2$;
for the definition of $\Ss_s^j$ see~(\ref{eq:def-S-4}) and~(\ref{eq:A.3.2-S}).

\smallskip\noindent
\textbf{Tameness of \boldmath$(\Ss_s^j)^{-1}$.}
Let $s$ be in the interpolation interval $[t_j^-,t_j]$
and let $v\in H_1$ be in the ball $B_R\subset H_1$ 
centered at $0$.

\begin{lemma}\label{le:Fredholm}
The restriction to $H_2$ of the linearized interpolation map
$$
   T:=d\Ss_s^j|_v|_{H_2}\colon H_2\to H_2
$$
at $v\in B_R\cap H_2$
is an isomorphism on $H_2$ and there is a constant
$\mu>0$ such~that
\begin{equation}\label{eq:(3)}
\begin{split}
   \abs{(d\Ss_s^j|_v)^{-1} \xi}_2
   \le \mu\left(\abs{\xi}_2+\abs{v}_2 \abs{\xi}_1\right)
\end {split}
\end{equation}
for every $\xi\in H_2$.
\end{lemma}

\begin{proof}[Proof of Lemma~\ref{le:Fredholm}]
Pick $v\in B_R\cap H_2$.
By~(\ref{eq:B_R-C_-}) it follows that $v+\psi^{-1}_{j+1}(x_s)\in C_-^{j+1}$.
By definition~(\ref{eq:def-S-44}) of $\Ss_s^j$
and by identity~(\ref{eq:hfgjk7799}) we obtain
\begin{equation}\label{eq:vfghvh326}
\begin{split}
   d\Ss_s^j|_v-\Id
   &=d \Ss^{\varphi^{-1}_j}_{\beta^j_s,\psi^{-1}_{j+1}(x_s)}|_v-\Id
\\
   &=\beta^j_s\left(d\varphi_j^{-1}|_{\psi^{-1}_{j+1}(x_s)+v}-\Id\right).
\end {split}
\end {equation}
$\bullet$
A first consequence of this identity, using~(\ref{eq:hgjkgf5322}) and
$\beta_s^j\le 1$, is the estimate
$$
   \norm{d\Ss_s^j|_v-\Id}_{\Ll(H_1)}
   \le\tfrac12.
$$
By the Neumann series,
see e.g.~\cite[Rmk.\,B.5]{Frauenfelder:2024e},
we get the uniform bound
\begin{equation}\label{eq:jghhj567u}
  \norm{(d\Ss_s^j|_v)^{-1}}_{\Ll(H_1)}
   \le\sum_{k=0}^\infty \norm{\Id-d\Ss_s^j|_v}_{\Ll(H_1)}^k
   \le\sum_{k=0}^\infty\tfrac{1}{2^k}
   =\tfrac{1}{1-\frac12}
   = 2 .
\end {equation}
$\bullet$
A second consequence of~(\ref {eq:vfghvh326}) is the following.
Assume that $v\in B_R\cap H_2$. Since basic paths are mapped to $H_2$,
it follows that $\psi^{-1}_{j+1}(x_s)+v \in C_-^{j+1}\cap H_2$.
Hence, by~(\ref{eq:(1)}) and using that $\beta_s^j\le 1$,
we obtain for every $\eta\in H_2$ the estimate
\begin{equation*}
\begin{split}
   \abs{d \Ss_s^j|_v\eta-\eta}_2
   &\le\tfrac12\abs{\eta}_2
   +c_j \left(\abs{\psi^{-1}_{j+1}(x_s)+v}_2+1\right)\abs{\eta}_1
\\
   &\le\tfrac12\abs{\eta}_2
   +c_j
   \left(\abs{v}_2+\abs{\psi^{-1}_{j+1}(x_s)}_2+1\right)\abs{\eta}_1 .
\end {split}
\end {equation*}
Since the basic path $x$ is continuous as a map to $H_2$, the
following maximum 
$$
\boxed{
   C:=\max_{j=1,\dots,k-1}
   \max_{s\in[t_j^-,t_j]}
   \{c_j(1+\abs{\psi^{-1}_{j+1}(x_s)}_2)\}
}
$$
is well defined. Therefore we obtain the inequality
\begin{equation*}
\begin{split}
   \abs{d \Ss_s^j|_v\eta-\eta}_2
   &\le\tfrac12\abs{\eta}_2
   +C \left(\abs{v}_2+1\right)\abs{\eta}_1 .
\end{split}
\end {equation*}
With this inequality we estimate
\begin{equation*}
\begin{split}
   \abs{d \Ss_s^j|_v\eta}_2=\abs{\eta+(d\Ss_s^j|_v -\Id)\eta}_2
   &\ge \abs{\eta}_2-\abs{(d \Ss_s^j|_v -\Id)\eta}_2\\
   &\ge \abs{\eta}_2-\tfrac12\abs{\eta}_2
   -C \left(\abs{v}_2+1\right)\abs{\eta}_1\\
   &=\tfrac12\abs{\eta}_2 -C \left(\abs{v}_2+1\right)\abs{\eta}_1 .
\end{split}
\end{equation*}
We rewrite this inequality as
\begin{equation}\label{eq:semi-F}
   \abs{\eta}_2
   \le 2\abs{d\Ss_s^j|_v\eta}_2
   +2C \left(\abs{v}_2+1\right)\abs{\eta}_1 .
\end{equation}
Note that since the restriction of the linearized transition map
$d\varphi_j^{-1}|_{\psi^{-1}_{j+1}(x_s)+v}$ to $H_2$
takes values in $H_2$, by ~(\ref {eq:vfghvh326}) so does the
restriction of $d\Ss_s^j|_v$ to $H_2$.
The level map
$
   d\Ss_s^j|_v\colon H_2\to H_2
$
is bounded linear again by~(\ref {eq:vfghvh326}).
Using that the inclusion $\iota\colon H_2\INTO H_1$ is compact,
estimate~(\ref{eq:semi-F}) implies that the level map 
$$
   T:=d\Ss_s^j|_v|_{H_2}\colon H_2\to H_2
$$
is a semi-Fredholm operator; see\cite[Le.\,A.1.1]{mcduff:2004a}
for $D=d\Ss_s^j|_v$, $K=\iota$, $X=Y=H_2$, and $Z=H_1$.
In the terminology of~\cite[\S 16]{Muller:2007a}
it is an upper semi-Fredholm operator.

\smallskip\noindent
$\bullet$
Next we compute the semi-Fredholm index of $T$
as an element of $\Z\cup\{-\infty\}$.

\smallskip\noindent
For this purpose we define the bounded linear operator
$$
   T_\beta
   :=(1-\beta)\Id+\beta d\varphi_j^{-1}|_{\psi^{-1}_{j+1}(x_s)+v}|_{H_2}
   \colon H_2\to H_2.
$$
for $\beta\in\R$. In view of~(\ref{eq:vfghvh326})
we have $T_{\beta_s^j}=T$ and $T_0=\Id$.
If $\beta\in[0,1]$ the same arguments which
established~(\ref{eq:semi-F}) show that
\begin{equation*}%\label{eq:semi-F-new}
   \abs{\eta}_2
   \le 2\abs{T_\beta\eta}_2
   +2C \left(\abs{v}_2+1\right)\abs{\eta}_1
\end{equation*}
for every $\eta\in H_2$.
Indeed note that in the derivation of~(\ref{eq:semi-F})
we only used that $\beta_s^j\le 1$.
In particular $T_\beta$ is a semi-Fredholm operator
for any $\beta\in[0,1]$.
Since $T_0=\Id$ we have $\INDEX(T_0)=0$.
Since the semi-Fredholm index is locally constant,
see~\cite[\S18 Cor.\,3]{Muller:2007a},
follows that $\INDEX(T_\beta)=0$ for any $\beta\in[0,1]$.
In particular $0=\INDEX(T):=\dim\ker T-\dim\coker T$.
Since $d\Ss_s^j|_v\colon H_1\to H_1$ is an isomorphism,
it has trivial kernel, hence its restriction $T$ has trivial kernel as well.
Since the index of $T$ vanishes it follows that its cokernel is
trivial as well, in particular $T$ is bijective and therefore, by the
open mapping theorem, it is an isomorphism from $H_2$ to $H_2$.

In particular, the operator $(d\Ss_s^j|_v)^{-1}\colon H_1\to H_1$
restricts to a bounded linear isomorphism
\begin{equation}\label{eq:(4565)}
   T^{-1}=(d\Ss_s^j|_v)^{-1}|_{H_2}\colon H_2\stackrel{\cong}{\longrightarrow} H_2 .
\end{equation}
Consequently for any $\xi\in H_2$ there exists $\eta\in H_2$
such that $\xi=d\Ss_s^j|_v\eta$.
Using~(\ref{eq:semi-F}) we compute
\begin{equation*}%\label{eq:(3)-a}
\begin{split}
   \abs{(d\Ss_s^j|_v)^{-1} \xi}_2
   &\le 2\abs{\xi}_2
   +2C \left(\abs{v}_2+1\right)\abs{(d\Ss_s^j|_v)^{-1}\xi}_1
\\
   &\le 2\abs{\xi}_2
   +2C
   \left(\abs{v}_2+1\right)\norm{(d\Ss_s^j|_v)^{-1}}_{\Ll(H_1)}\abs{\xi}_1
\\
   &\le 2\abs{\xi}_2
   +4C\left(\abs{v}_2+1\right)\abs{\xi}_1
\\
   &\le \mu\left(\abs{\xi}_2+\abs{v}_2 \abs{\xi}_1\right)
   ,\qquad \boxed{\mu:= 2(1+2C)>2}
\end {split}
\end{equation*}
for every $\xi\in H_2$. Here step three is
by~(\ref{eq:jghhj567u}), the final step by~(\ref{eq:12}).

This proves~(\ref{eq:(3)}) and concludes the proof of Lemma~\ref{le:Fredholm}.
\end{proof}

We continue the proof of Step~3,
more precisely the proof of tameness of the inverse $(\Ss_s^j)^{-1}$
of the interpolation map.

\smallskip\noindent
$\bullet$
Continuity of the second derivative of $(\Ss_s^j)^{-1}$
with respect to the $C^2$ topology.

\smallskip\noindent
The map $\Ss_s^j\colon B_R\cap H_2 \to \Ss_s^j(B_R)\cap H_2$ is $C^2$
since it is tame.
It follows from~(\ref{eq:(4565)}) by the implicit function theorem,
see e.g.~\cite[Thm.\,A.3.1]{mcduff:2004a}
that $(\Ss_s^j)^{-1}\colon \Ss_s^j(B_R)\cap H_2 \to B_R\cap H_2$
is $C^2$ as well.

\smallskip\noindent
$\bullet$
Recall from~(\ref{eq:def-S-44}) that
$\Ss_s^j\colon H_1\supset B_R\to H_1$
is a $C^2$ diffeomorphism onto its image.
Linearize the identity $\id=(\Ss_s^j)^{-1}\circ\Ss_s^j$
at $v\in B_R\subset H_1$ to get
\begin{equation}\label{eq:777trghhbu77}
   \Id_{H_1}=d (\Ss_s^j)^{-1}|_{\Ss_s^j(v)} d \Ss_s^j|_v .
\end{equation}
Linearizing the displayed identity, we resolve that
\begin{equation}\label{eq:sec-deriv}
   d^2(\Ss_s^j)^{-1}|_{\Ss_s^j(v)}\left(\xi,\eta\right)
   =-\left(d\Ss_s^j|_v\right)^{-1}
   d^2\Ss_s^j|_v
   \left(\left(d\Ss_s^j|_v\right)^{-1}\xi, \left(d\Ss_s^j|_v\right)^{-1}\eta\right)
\end{equation}
for all $\xi,\eta\in H_1$.

\smallskip\noindent
$\bullet$
We show $(\Ss_s^j)^{-1}\colon H_1\supset \Ss_s^j(B_R)\to B_R$
satisfies the tameness estimate~(\ref{eq:tame}).

\smallskip\noindent
By~(\ref{eq:(4565)}) and since $\Ss_s^j$ maps $B_R\cap H_2$ to
$H_2$ it follows that $(\Ss_s^j)^{-1}|_{\Ss_s^j(B_R)\cap H_2}$
maps $\Ss_s^j(B_R)\cap H_2$ onto $B_R\cap H_2$.
Hence if $\Ss_s^j(v_0)$ is in $\Ss_s^j(B_R)\cap H_2$
it follows that $v_0\in B_R\cap H_2$.
Since $\Ss_s^j$ is tame there exists an $H_1$-open neighborhood
of $v_0$ and a constant $\kappa>0$, notation
\begin{equation}\label{eq:V}
   V\subset B_R
   ,\qquad
\boxed{
   \kappa>0,
}
\end{equation}
such that for every $v\in V\cap H_2$ and all $\xi,\eta\in H_2$
it holds the estimate
\begin{equation}\label{eq:tamehgfhh}
\begin{split}
   \Abs{d^2 \Ss_s^j|_v (\xi,\eta)}_2 
   &\le \kappa\bigl(\abs{\xi}_1 \abs{\eta}_2+\abs{\xi}_2\abs{\eta}_1
   +\abs{v}_2 \abs{\xi}_1 \abs{\eta}_1\bigr) .
\end{split}
\end{equation}
Since $\Ss_s^j$ is $C^2$ on $B_R$, maybe after shrinking $V$ and
enlarging $\kappa$, we can additionally assume that
\begin{equation}\label{eq:bhhjhjhgj}
   \Abs{d^2 \Ss_s^j|_v (\xi,\eta)}_1\le \kappa \abs{\xi}_1
   \abs{\eta}_1 .
\end{equation}
Hence for $\Ss_s^j(v)\in \Ss_s^j(V)\cap H_2$ we estimate
\begin{equation*}
\begin{split}
   &\Abs{d^2 (\Ss_s^j)^{-1}|_{\Ss_s^j(v)} (\xi,\eta)}_2
\\
   &=\Abs{\left(d\Ss_s^j|_v\right)^{-1}
   d^2\Ss_s^j|_v
   \left(\left(d\Ss_s^j|_v\right)^{-1}\xi, \left(d\Ss_s^j|_v\right)^{-1}\eta\right)
   }_2
\\
   &\le \mu \Abs{d^2\Ss_s^j|_v
   \left(\left(d\Ss_s^j|_v\right)^{-1}\xi, \left(d\Ss_s^j|_v\right)^{-1}\eta\right)
   }_2
   \\
   &\quad+\mu\abs{v}_2 \Abs{d^2\Ss_s^j|_v
   \left(\left(d\Ss_s^j|_v\right)^{-1}\xi, \left(d\Ss_s^j|_v\right)^{-1}\eta\right)
   }_1
\\
   &\le\mu \kappa
   \Abs{\left(d\Ss_s^j|_v\right)^{-1}\xi}_1\Abs{\left(d\Ss_s^j|_v\right)^{-1}\eta}_2
   +\mu \kappa
   \Abs{\left(d\Ss_s^j|_v\right)^{-1}\xi}_2\Abs{\left(d\Ss_s^j|_v\right)^{-1}\eta}_1
   \\
   &\quad+2\mu \kappa\abs{v}_2
   \Abs{\left(d\Ss_s^j|_v\right)^{-1}\xi}_1\Abs{\left(d\Ss_s^j|_v\right)^{-1}\eta}_1
\\
   &\le 2\mu^2\kappa\abs{\xi}_1
   \left(\abs{\eta}_2+\abs{v}_2\abs{\eta}_1\right)
   +2\mu^2\kappa\abs{\eta}_1
   \left(\abs{\xi}_2+\abs{v}_2\abs{\xi}_1\right)
   +8\mu\kappa\abs{v}_2\abs{\xi}_1\abs{\eta}_1
\\
   &\le 4\mu\kappa(\mu+2)\left(
   \abs{\xi}_1\abs{\eta}_2+\abs{\xi}_2\abs{\eta}_1
   +\abs{v}_2\abs{\xi}_1\abs{\eta}_1\right) .
\end{split}
\end{equation*}
Here step 1 is by~(\ref{eq:sec-deriv})
and step 2 by~(\ref{eq:(3)}).
Step 3 follows by~(\ref{eq:tamehgfhh}) and~(\ref{eq:bhhjhjhgj}).
Step 4 is by~(\ref{eq:jghhj567u}) and~(\ref{eq:(3)}) again.
\newline
To summarize, with 
$$
\boxed{
   \kappa_*:=4\mu\kappa(\mu+2)>32\kappa
}
$$
and whenever $v\in V\cap H_2$ we have
\begin{equation}\label{eq:7ggf7775gh}
\begin{split}
   \Abs{d^2 (\Ss_s^j)^{-1}|_{\Ss_s^j(v)} (\xi,\eta)}_2
   &\le \kappa_*\left(
   \abs{\xi}_1\abs{\eta}_2+\abs{\xi}_2\abs{\eta}_1
   +\abs{v}_2\abs{\xi}_1\abs{\eta}_1\right)
\end{split}
\end{equation}
for all $\xi,\eta\in H_2$.
This estimate is not yet the tameness estimate~(\ref{eq:tame}).
The reason is that in place of $\abs{v}_2$ we need
$\abs{\Ss_s^j(v)}_2$.
To this end we show that, maybe after shrinking $V$, there exists
a constant $C$ such that
\begin{equation}\label{eq:7ggf7775gffghh}
\boxed{
    \abs{v}_2\le C\left(\abs{\Ss_s^j(v)}_2 +1\right)
}
\end{equation}
whenever $v\in V\cap H_2$.
For that purpose we use~(\ref{eq:(3)}) inductively.

\smallskip\noindent
\textbf{Special case.}
To simplify the computation we first derive~(\ref{eq:7ggf7775gffghh})
in the special case where $v_0=0$ and $\Ss_s^j(0)=0$.
Maybe after shrinking $V$ we can assume that
$\Ss_s^j(V)$ is convex and contained in the radius-$\frac{1}{2\mu}$ ball in
$H_1$ about the origin.
For $v\in V\cap H_2$ we estimate (juxtaposition means composition)
\begin{equation*}
\begin{split}
   \abs{v}_2
   &=\abs{(\Ss_s^j)^{-1}\Ss_s^j (v)}_2
\\
   &\stackrel{2}{=}\Bigl|\int_0^1 
   \frac{d}{dt_1}(\Ss_s^j)^{-1}\left( t_1\Ss_s^j (v)\right)
   dt_1
   \Bigr|_2
\\
   &\stackrel{3}{=}\Bigl|\int_0^1 
   d(\Ss_s^j)^{-1}|_{ t_1\Ss_s^j (v)} \Ss_s^j(v)
   \, dt_1
   \Bigr|_2
\\
   &\stackrel{4}{=}\Bigl|\int_0^1 
   (d\Ss_s^j|_{(\Ss_s^j)^{-1}(t_1\Ss_s^j (v))})^{-1} \Ss_s^j (v)
   \, dt_1 
   \Bigr|_2
\\
   &\stackrel{5}{\le}
   \int_0^1\Abs{(d\Ss_s^j|_{(\Ss_s^j)^{-1}(t_1\Ss_s^j (v))})^{-1}\Ss_s^j (v)}_2 dt_1
\\
   &\stackrel{6}{\le}
   \int_0^1\mu\Bigl(\Abs{\Ss_s^j (v)}_2
   +\Abs{(\Ss_s^j)^{-1}(t_1\Ss_s^j (v))}_2
   \underbrace{\Abs{\Ss_s^j(v)}_1}_{\le 1/2\mu} \Bigr)\, dt_1
\\
   &\le %\stackrel{5}{\le}
   \mu \Abs{\Ss_s^j (v)}_2
   +\tfrac12 \int_0^1 \Abs{(\Ss_s^j)^{-1}(t_1\Ss_s^j (v))}_2 dt_1 .
\end{split}
\end{equation*}
Step 2 is by the fundamental theorem of calculus.
Step 3 is by the chain rule.
Step 4 is by~(\ref{eq:sec-deriv}).
Step 5 is by monotonicity of the integral.
Step 6 is by~(\ref{eq:(3)}).

The integrand in the second summand we similarly estimate further
\begin{equation*}
\begin{split}
   &\Abs{(\Ss_s^j)^{-1}(t_1\Ss_s^j (v))}_2
\\
   &=\Bigl|\int_0^1 
   \underbrace{\overbrace{d(\Ss_s^j)^{-1}|_{ t_1t_2\Ss_s^j (v)} t_1\Ss_s^j(v)}
      ^{\tfrac{d}{dt_2}(\Ss_s^j)^{-1}\left( t_1t_2\Ss_s^j (v)\right) }
      }_{(d\Ss_s^j|_{(\Ss_s^j)^{-1}(t_1t_2\Ss_s^j (v))})^{-1} t_1\Ss_s^j (v)}
   \, dt_2
   \Bigr|_2
\\
   &\stackrel{2}{\le}
   \int_0^1\Bigl|(d\Ss_s^j|_{(\Ss_s^j)^{-1}(t_1t_2\Ss_s^j (v))})^{-1}
   \underbrace{t_1}_{\le1}\Ss_s^j (v)\Bigr|_2 dt_2
\\
   &\stackrel{3}{\le}
   \int_0^1\mu\Bigl(\Abs{\Ss_s^j (v)}_2
   +\Abs{(\Ss_s^j)^{-1}(t_1t_2\Ss_s^j (v))}_2
   \underbrace{\Abs{\Ss_s^j(v)}_1}_{\le 1/2\mu} \Bigr)\, dt_2
\\
   &\le 
   \mu \Abs{\Ss_s^j (v)}_2
   +\tfrac12 \int_0^1 \Abs{(\Ss_s^j)^{-1}(t_1t_2\Ss_s^j (v))}_2 dt_2 .
\end{split}
\end{equation*}
Step 2 is by monotonicity of the integral.
Step 3 is by~(\ref{eq:(3)}).

Insert this latter estimate into the previous estimate to obtain
\begin{equation*}
\begin{split}
   \abs{v}_2
   &\le\left(1+\tfrac12\right)\mu \Abs{\Ss_s^j (v)}_2
   +\tfrac{1}{4} \int_0^1 \int_0^1 \Abs{(\Ss_s^j)^{-1}(t_1t_2\Ss_s^j(v))}_2 
   dt_2\, dt_1 .
\end{split}
\end{equation*}
Estimating the integrand in the second term inductively with the help
of~(\ref{eq:(3)}) we obtain for every $n\in\N$ the following estimate
\begin{equation*}
\begin{split}
   \abs{v}_2
   &\le
   \Bigl(\sum_{j=0}^{n-1} \frac{1}{2^j} \Bigr) \mu \Abs{\Ss_s^j(v)}_2
   +\frac{1}{2^n}\int_0^1\dots\int_0^1
   \Abs{(\Ss_s^j)^{-1}(t_1\dots t_n\Ss_s^j(v))}_2 dt_n\dots dt_1 .
\end{split}
\end{equation*}
We consider the limit $n\to\infty$.
First note that there exists a constant $c_V$ depending on $V$, but not
on $n$, such that $\abs{(\Ss_s^j)^{-1}(t\Ss_s^j(v))}_2\le c_V$ for
every $t\in[0,1]$. Hence we can estimate the second summand from above
by $c_V/2^n$. Therefore in the limit, as $n\to\infty$, the second summand vanishes.
Using that $\sum_{j=0}^{n-1} \frac{1}{2^j} =2$ we
get~(\ref{eq:7ggf7775gffghh}) with $C=2\mu$, namely
$$
   \Abs{v}_2
   \le 2 \mu \Abs{\Ss_s^j(v)}_2 .
$$

%\smallskip
\noindent
\textbf{General case.}
Maybe after shrinking $V$ we can assume that
$\Ss_s^j(V)$ is convex and contained in the radius-$\frac{1}{2\mu}$ ball in
$H_1$ around $\Ss_s^j (v_0)$. For $v\in V\cap H_2$ we estimate
by the same techniques as in the special case above
\begin{equation*}
\begin{split}
   \abs{v}_2
   &=\abs{v_0+v-v_0}_2
\\
   &\le
   \abs{v_0}_2+
   \abs{(\Ss_s^j)^{-1}\Ss_s^j (v)-(\Ss_s^j)^{-1}\Ss_s^j (v_0)}_2
\\
   &= %\stackrel{2}{=}
   \abs{v_0}_2+
   \Bigl|\int_0^1 
   \underbrace{\overbrace{d(\Ss_s^j)^{-1}|_{ t_1\Ss_s^j (v)+(1-t_1)\Ss_s^j (v_0)}
                        \left(\Ss_s^j(v)-\Ss_s^j(v_0)\right)}
      ^{\tfrac{d}{dt_1}(\Ss_s^j)^{-1}\left( t_1\Ss_s^j (v)+(1-t_1)\Ss_s^j (v_0)\right)}
      }_{(d\Ss_s^j|_{(\Ss_s^j)^{-1}(t_1\Ss_s^j (v)+(1-t_1)\Ss_s^j (v_0))})^{-1} (\Ss_s^j (v)-\Ss_s^j (v_0))}
   \, dt_1 
   \Bigr|_2
\\
   &\le %\stackrel{3}{\le}
   \abs{v_0}_2+
   \int_0^1\Abs{(d\Ss_s^j|_{(\Ss_s^j)^{-1}(t_1\Ss_s^j (v)+(1-t_1)\Ss_s^j (v_0))})^{-1}
   (\Ss_s^j (v)-\Ss_s^j (v_0))}_2dt_1
\\
   &\stackrel{5}{\le}
   \abs{v_0}_2+
   \int_0^1\mu(\Abs{\Ss_s^j (v)-\Ss_s^j (v_0)}_2 dt_1
   \\
   &\quad
   +\int_0^1\mu\Abs{(\Ss_s^j)^{-1}(t_1\Ss_s^j (v)+(1-t_1)\Ss_s^j (v_0))}_2
   \underbrace{\Abs{\Ss_s^j(v) -\Ss_s^j (v_0}_1}_{\le 1/2\mu} \, dt_1
\\
   &\le %\stackrel{5}{\le}
   \abs{v_0}_2+
   \mu \Abs{\Ss_s^j (v)-\Ss_s^j (v_0)}_2
   +\tfrac12 \int_0^1 \Abs{(\Ss_s^j)^{-1}(t_1\Ss_s^j (v) +(1-t_1)\Ss_s^j (v_0))}_2 dt_1
   .
\end{split}
\end{equation*}
Step 5 is by~(\ref{eq:(3)}).
The integrand in the second summand we estimate further
\begin{equation*}
\begin{split}
   &\Abs{(\Ss_s^j)^{-1}(t_1\Ss_s^j (v) +(1-t_1)\Ss_s^j (v_0))-v_0+v_0}_2
\\
   &=\Bigl|\int_0^1 
   \underbrace{\overbrace{d(\Ss_s^j)^{-1}|_{ t_1t_2\Ss_s^j (v) +(1-t_1t_2)\Ss_s^j (v_0)}
      t_1\left(\Ss_s^j(v)-\Ss_s^j(v_0)\right)}
      ^{\tfrac{d}{dt_2}(\Ss_s^j)^{-1}\left( t_1t_2\Ss_s^j (v)+ (1-t_1t_2)\Ss_s^j (v_0))\right) }
      }_{(d\Ss_s^j|_{(\Ss_s^j)^{-1}(t_1t_2\Ss_s^j (v) +(1-t_1t_2)\Ss_s^j (v_0))})^{-1}
         t_1\left(\Ss_s^j(v)-\Ss_s^j(v_0)\right)}
   \, dt_2 +v_0
   \Bigr|_2
\\
   &\stackrel{2}{\le}
   \abs{v_0}_2
   +\int_0^1\Bigl|(d\Ss_s^j|_{(\Ss_s^j)^{-1}(t_1t_2\Ss_s^j (v) + (1-t_1t_2)\Ss_s^j (v_0))})^{-1}
   \underbrace{t_1}_{\le1}\left(\Ss_s^j(v)-\Ss_s^j(v_0)\right)\Bigr|_2 dt_2
\\
   &\stackrel{3}{\le} 
   \abs{v_0}_2
   +\int_0^1\mu\biggl(\Abs{\Ss_s^j (v) -\Ss_s^j(v_0)}_2
   \\
   &\qquad\qquad\qquad
   +\Abs{(\Ss_s^j)^{-1}(t_1t_2\Ss_s^j (v) + (1-t_1t_2)\Ss_s^j (v_0))}_2
   \underbrace{\Abs{\Ss_s^j(v) -\Ss_s^j(v_0)}_1}_{\le 1/2\mu} \biggr)\, dt_2
\\
   &\le 
   \abs{v_0}_2
   +\mu \Abs{\Ss_s^j (v) -\Ss_s^j(v_0)}_2
   +\tfrac12\int_0^1\Abs{(\Ss_s^j)^{-1}(t_1t_2\Ss_s^j (v)+(1-t_1t_2)\Ss_s^j (v_0))}_2
   dt_2 .
\end{split}
\end{equation*}
Step 2 is by the triangle inequality and monotonicity of the integral.
Step 3 is by~(\ref{eq:(3)}).
Insert this latter estimate into the previous one to obtain
\begin{equation*}
\begin{split}
   \abs{v}_2
   &\le\left(1+\tfrac12\right)\abs{v_0}_2
   +\left(1+\tfrac12\right)\mu \Abs{\Ss_s^j (v) -\Ss_s^j(v_0)}_2
   \\
   &\quad
   +\tfrac{1}{4} \int_0^1 \int_0^1
   \Abs{(\Ss_s^j)^{-1}(t_1t_2\Ss_s^j (v)+(1-t_1t_2)\Ss_s^j (v_0))}_2 
   dt_2\, dt_1 .
\end{split}
\end{equation*}
Estimating the integrand in the second term inductively with the help
of~(\ref{eq:(3)}) we obtain for every $n\in\N$ the following estimate
\begin{equation*}
\begin{split}
   \abs{v}_2
   &\le \Bigl(\sum_{j=0}^{n-1} \frac{1}{2^j} \Bigr)\abs{v_0}_2
   +\Bigl(\sum_{j=0}^{n-1} \frac{1}{2^j} \Bigr) \mu \Abs{\Ss_s^j(v) -\Ss_s^j(v_0)}_2
   \\
   &\quad
   +\frac{1}{2^n}\int_0^1\dots\int_0^1
   \Abs{(\Ss_s^j)^{-1}(t_1\dots t_n\Ss_s^j (v)+(1-t_1\dots t_n)\Ss_s^j (v_0))}_2 
   dt_n\dots dt_1 .
\end{split}
\end{equation*}
We consider the limit $n\to\infty$.
First note that there exists a constant $C_V$ depending on $V$, but not
on $n$, such that $\abs{(\Ss_s^j)^{-1}(t\Ss_s^j(v)+(1-t) \Ss_s^j(v_0))}_2\le C_V$
for every $t\in[0,1]$.
Hence we can estimate the final (third) summand from above
by $c_V/2^n$. Therefore in the limit, as $n\to\infty$, the final summand vanishes.
Using that $\sum_{j=0}^{n-1} \frac{1}{2^j} =2$ we estimate
\begin{equation}\label{eq:K_s-1}
\begin{split}
   \Abs{v}_2
   &\le 2\Abs{v_0}_2+2 \mu \Abs{\Ss_s^j(v)-\Ss_s^j (v_0)}_2\\
   &\le 2 \mu \Abs{\Ss_s^j(v)}_2
   +2 \mu \Abs{\Ss_s^j(v_0)}_2
   +2\Abs{v_0}_2\\
   &\le K\left(\abs{\Ss_s^j(v)}_2 +1\right),
   \quad\boxed{K(s):=2\max\{\mu,\mu\abs{\Ss_s^j (v_0)}_2+\abs{v}_2\}.}
\end{split}
\end{equation}
%where $K=2\max\{\mu,\mu\abs{\Ss_s^j (v_0)}_2+\abs{v}_2\}$.
This proves the claim~(\ref{eq:7ggf7775gffghh}).

%\smallskip
\noindent
\textsc{Conclusion.}
We are now in position to prove the tameness estimate~(\ref{eq:tame})
for the inverse $(\Ss_s^j)^{-1}\colon H_1\supset \Ss_s^j(B_R)\to B_R$.
Plugging~(\ref{eq:7ggf7775gffghh}) into~(\ref{eq:7ggf7775gh}) leads to
\begin{equation}\label{eq:K_s-2}
\begin{split}
   &\Abs{d^2 (\Ss_s^j)^{-1}|_{\Ss_s^j(v)} (\xi,\eta)}_2
\\
   &\le \kappa_*\bigl(
   \abs{\xi}_1\abs{\eta}_2+\abs{\xi}_2\abs{\eta}_1
   +K\left(\abs{\Ss_s^j(v)}_2 +1\right)\abs{\xi}_1\abs{\eta}_1\bigr)
\\
   &\stackrel{2}{\le}\kappa_*\Bigl(\abs{\xi}_1\abs{\eta}_2+\abs{\xi}_2\abs{\eta}_1
   +K \abs{\Ss_s^j(v)}_2 \abs{\xi}_1\abs{\eta}_1
   +\tfrac{K}{2}\abs{\xi}_2\abs{\eta}_1
   +\tfrac{K}{2}\abs{\xi}_1\abs{\eta}_2
   \Bigr)
\\
   &=
   \kappa_*\left(\tfrac{K}{2}+1\right) \abs{\xi}_1\abs{\eta}_2
   +\kappa_*\left(\tfrac{K}{2}+1\right) \abs{\xi}_2\abs{\eta}_1
   +\kappa_* K \abs{\Ss_s^j(v)}_2 \abs{\xi}_1\abs{\eta}_1
\\
   &\le
   C_*\bigl(
   \abs{\xi}_1\abs{\eta}_2+\abs{\xi}_2\abs{\eta}_1
   +\abs{\Ss_s^j(v)}_2 \abs{\xi}_1\abs{\eta}_1
   \bigr)
\end{split}
\end{equation}
where the final step holds for $C_*(s):=\kappa_*\max\{K_s, \tfrac{K_s}{2}+1\}$.
Step 2 is by~(\ref{eq:12}).

\smallskip
This proves tameness of $(\Ss_s^j)^{-1}$ and completes the proof of Step~3.
\end{proof}

%\smallskip
\noindent
\textbf{Step 4.}
For any $s\in\R$ the map $\varphi_s\colon \{s\}\times U^s\to H_1$
defined by~(\ref{eq:rho-4}) is tame.

\begin{proof}
In~(\ref{eq:rho-4}) the maps
$$
   \tilde\psi_{\tilde j}^{-1}\circ\psi_j\colon
   \psi_j^{-1}(V_j\cap\tilde V_{\tilde {j}})
   \to\tilde\psi_{\tilde j}^{-1}(V_j\cap\tilde V_{\tilde {j}})
$$
are transition maps of a tame atlas and therefore tame.
Because by Theorem~\ref{thm:tame}
the composition of tame maps is tame,
tameness of $\varphi_s$ follows by~(\ref{eq:rho-4})
in view of Step~3.
This proves Step~4.
\end{proof}

%\smallskip
\noindent
\textbf{Step 5.}
The interpolation map,
see~(\ref{eq:def-S-44}) and Figure~\ref{fig:chart-no-exp-44},
$$
   \Ss^j\colon [t_j^-,t_j]\times B_R\to H_1
   ,\quad
   (s,v)\mapsto \Ss^j_s(v)
$$
is parametrized tame.

\begin{proof}
We need to check two conditions.
1)~The map $\Ss^j$ is $C^2$ and its restriction,
$$
   \Ss^j|_{[t_j^-,t_j]\times (B_R\cap H_2)}
   \colon [t_j^-,t_j]\times \left(B_R\cap H_2\right) \to H_2
$$
is $C^2$ as well.
Using that $\Ss_s^j$ is in~(\ref{eq:def-S-4}) defined by
$$
   \Ss_s^j
   :=\Ss^{\varphi^{-1}_j}_{\beta^j_s,\psi^{-1}_{j+1}(x_s)}
$$
together with smoothness of the cutoff function $\beta^j_s$ in the
$s$-variable and the fact that our basic path is a
$C^2$-map $s\mapsto x_s$ to the second manifold level $X_2$,
both $C^2$-claims follow directly from~(\ref{eq:A.3.2-S}).
\\
2)~The parametrized tameness estimate~(\ref{eq:tame-parametrized})
follows from tameness of $\phi_j^{-1}$
by noting that $\phi_j^{-1}$ does not depend on $s$.
This proves 2) and Step~5.
\end{proof}

%\smallskip
\noindent
\textbf{Step 6.}
The inverted interpolation map family
$$
   \Ss^j_-
   \colon
   \{(s,v)\in [t_j^-,t_j]\times H_1\mid v\in \Ss^j_s(B_R)\}
   \to B_R
   ,\quad
   (s,v)\mapsto (\Ss^j_s)^{-1}(v)
$$
is parametrized tame.

\begin{proof}
1)~That $\Ss^j_-$ is $C^2$ on level 1 and level 2 follows
from the fact that this is true for $\Ss^j$ by Step~5
and a version of the implicit function theorem as explained
in Proposition~\ref{prop:inversion} for $\Gg=\Ss^j_-$.
\\
2)~The tameness estimate~(\ref{eq:tame-parametrized})
follows from the tameness estimate~(\ref{eq:K_s-2})
by noting that the only $s$-dependence of the constant $C_*$
comes from the constant $K(s)$ as introduced
in~(\ref{eq:K_s-1}), but this constant can be chosen uniformly
in a little neighborhood of $s\in\R$
by using that $H_2$-norm of $\Ss_s^j(v_0)$
depends continuously on $s$. This proves Step~6.
\end{proof}

%\smallskip
\noindent
\textbf{Step 7.}
The map $\varphi\colon \R\times H_1\supset O\to H_1$
defined by~(\ref{eq:rho-4}) is asymptotically constant parametrized tame.

\begin{proof}
By Steps~5 and~6 the map $\varphi$ is a composition of parametrized
tame maps.
The same argument as in the proof of Theorem~\ref{thm:tame}
shows that the composition of parametrized tame maps
is again parametrized tame.
The map $\varphi$ is asymptotically constant
by~(\ref{eq:as-cnst_-}) and~(\ref{eq:as-cnst_+}).
This proves Step~7.
\end{proof}

%\smallskip
\noindent
\textbf{Step 8.}
The map $\Phi\colon\Uu_0\to\tilde \Uu_0$
defined by~(\ref{eq:Psi-def-4}) is $C^1$.

\begin{proof}
This follows from Step~7 and Theorem~\ref{thm:B-Phi-parametrized}.
\end{proof}

%\smallskip
\noindent
\textbf{Step 9.}
The inverse $\Phi^{-1}\colon\tilde\Uu_0\to\Uu_0$,
see~(\ref{eq:Phi-inverse}), is $C^1$.

\begin{proof}
In Step~8 interchange the roles of the charts $\Uu$
and $\tilde\Uu$.
\end{proof}

%\smallskip
\noindent
\textbf{Step 10.}
We prove Theorem~\ref{thm:Psi-diffeo-4}.

\begin{proof}
Theorem~\ref{thm:Psi-diffeo-4} follows by Step~8 and Step~9.
\end{proof}

The proof Theorem~\ref{thm:Psi-diffeo-4} is complete.
\end{proof}

\begin{corollary}\label{cor:T-Psi-diffeo}
The weak tangent map
$T\Phi\colon \Uu_0\times L^2_{H_1}\to \tilde \Uu_0\times L^2_{H_1}$
defined by~(\ref{eq:weak-tangent-map}) is a $C^1$ diffeomorphism.
\end{corollary}

\begin{proof}
By Step~7 in the proof of Theorem~\ref{thm:Psi-diffeo-4} 
the map $\varphi$ is asymptotically constant parametrized tame.
So by Theorem~\ref{thm:B-TPhi-parametrized}
the weak tangent map $T\Phi$ is $C^1$.
And interchanging the roles of $\Uu_0$ and $\tilde \Uu_0$ its inverse
is $C^1$ as well.
This proves Corollary~\ref{cor:T-Psi-diffeo}.
\end{proof}

%\newpage
%%%%%%%%%%%%%%%%%%%%%%%%%%%%%%%%%%%
%%%%%%%%%%%%%%%%%%%%%%%%%%%%%%%%%%%
%%%%%%% Subsection  %%%%%%%%%%%%%%%%%%%
%%%%%%%%%%%%%%%%%%%%%%%%%%%%%%%%%%%
%%%%%%%%%%%%%%%%%%%%%%%%%%%%%%%%%%%
\subsection{Weak tangent bundles}\label{sec:weak-tangent bundle}

We are now in position to define precisely what the weak tangent bundle is.
For a local parametrization $\Psi\in\Aa\Pp_{x_-x_+}$, see~(\ref{eq:AP}),
we denote by $\Uu_\Psi\subset W^{1,2}_{H_1}\cap L^2_{H_2}$
the domain of $\Psi$.
Given two local parametrizations $\Psi,\tilde\Psi\in\Aa\Pp_{x_-x_+}$,
the corresponding \textbf{weak tangent bundle transition map}
is defined by
$$
   \Phi_{\Psi\tilde\Psi}
   :={\tilde\Psi}^{-1}\circ\Psi|_{\Uu_{\Psi\tilde\Psi}}
   \colon
   \Uu_{\Psi\tilde\Psi}
   \to
   \Uu_{\tilde\Psi\Psi}
$$
where $\Uu_{\Psi\tilde\Psi}$ and $\Uu_{\tilde\Psi\Psi}$ are open
subsets of the Hilbert space
$\bigl(W^{1,2}_{H_1}\cap L^2_{H_2}\bigr)\times L^2_{H_1}$ and,
as illustrated by Figure~\ref{fig:chart-weak-tangent}, they
are defined by
$$
   \Uu_{\Psi\tilde\Psi}
   :=\Psi^{-1}\left(\Psi(\Uu_\Psi)\cap \tilde\Psi(\Uu_{\tilde\Psi})\right)
   ,\qquad
   \Uu_{\tilde\Psi\Psi}
   :=\tilde{\Psi}^{-1}\left(\Psi(\Uu_\Psi)\cap \tilde\Psi(\Uu_{\tilde\Psi})\right).
$$
By definition the total space of the \textbf{weak tangent bundle} is
the quotient space
$$
   \Ee=\Ee_{x_-x_+}
   :=\biggl(\bigcup_{\Psi\in\Aa\Pp_{x_-x_+}}
   \left(\Uu_\Psi\times L^2_{H_1}\right)
   \biggr)
   / \sim
$$
where two points $(x,v)\in\Uu_\Psi\times L^2_{H_1}$ and
$(\tilde x,\tilde v)\in \Uu_{\tilde\Psi}\times L^2_{H_1}$ are \textbf{equivalent}
$$
   (x,v)\sim (\tilde x,\tilde v)
   \quad:\Leftrightarrow\quad
   T\Phi_{\Psi\tilde\Psi}(x,v)
   =(\tilde x,\tilde v).
$$
It follows from Corollary~\ref{cor:T-Psi-diffeo}
that the transition maps are $C^1$. Hence $\Ee_{x_-x_+}$ is a
$C^1$ manifold.
\begin{figure}[h]
  \centering
  \includegraphics%[%width=0.99\textwidth]
                             [height=4.5cm]
                             {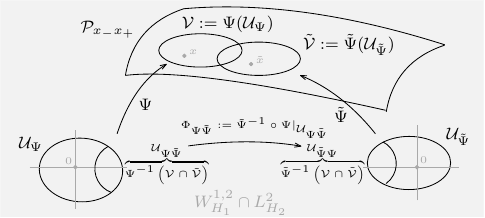}
  \caption{Transition map $\Phi_{\Psi\tilde\Psi}$
                 for path space $\Pp_{x_-x_+}$ modeled
                 on $W^{1,2}_{H_1}\cap L^2_{H_2}$}
   \label{fig:chart-weak-tangent}
\end{figure}

\begin{theorem}\label{thm:weak-TP}
The weak tangent bundle $\Ee_{x_-x_+}$ is a $C^1$ manifold modeled on
the Hilbert space $\bigl(W^{1,2}_{H_1}\cap L^2_{H_2}\bigr)\times L^2_{H_1}$.
\end{theorem}

\begin{proof}
By Corollary~\ref{cor:T-Psi-diffeo} the transition maps are $C^1$.
\end{proof}

\newpage
\appendix
%%%%%%%%%%%%%%%%%%%%%%%%%%%%%%%%%%%
%%%%%%%%%%%%%%%%%%%%%%%%%%%%%%%%%%%
%%%%%%% Section  %%%%%%%%%%%%%%%%%%%%%
%%%%%%%%%%%%%%%%%%%%%%%%%%%%%%%%%%%
%%%%%%%%%%%%%%%%%%%%%%%%%%%%%%%%%%%
\section{Hilbert space valued Sobolev spaces}\label{sec:Bochner-Sobolev}

In this appendix we collect and detail some results on 
Hilbert space valued Sobolev spaces, scattered throughout the
literature, which we used throughout this paper.
As required by the theorem of Pettis we assume that $H$ is a
\emph{separable} Hilbert space. We denote the induced norm
on $H$ by $\abs{\cdot}$.

%%%%%%%%%%%%%%%%%%%%%%%%%%%%%%%%%%%
%%%%%%% Section  %%%%%%%%%%%%%%%%%%%%%
%%%%%%%%%%%%%%%%%%%%%%%%%%%%%%%%%%%
\subsection{Measurability}\label{sec:measurability}

\begin{definition}\label{def:simple}
(i)~Let $X$ be a topological space and
$B(X)$ the \textbf{\boldmath Borel $\sigma$-algebra},
the smallest $\sigma$-algebra that contains the open sets;
cf.~Definition~\ref{def:Borel}.
\\
(ii)~A \textbf{measurable space} is a pair
$(A,\Aa)$ where $A$ is a set and $\Aa$ is a $\sigma$-algebra in $A$.
The elements of a $\sigma$-algebra are called \textbf{measurable sets}.
A map $f\colon(A,\Aa)\to(X,\Bb(X))$ 
is called a \textbf{measurable map}
if pre-images of measurable sets are measurable.
\\
(iii)
Let $(A,\Aa)$ be a measurable space.
A map $s\colon A\to H$ is called \textbf{simple}
%more precisely \textbf{\boldmath$\Aa$-simple},
if it is of the form $s=\sum_{k=1}^N \chi_{A_k} x_k$
with $A_k\in\Aa$ and $x_k\in H$. Here $\chi_{A_k}$ is the
\textbf{characteristic function} of the set $A_k$,
i.e. $\chi_{A_k}\equiv1$ on $A_k$ and $\chi_{A_k}\equiv0$ else.
\end{definition}

\begin{remark}\label{rem:meas}
(i)~Simple maps $f\colon A\to H$ are measurable.
(ii)~By the minimality of the Borel $\sigma$-algebra,
a map $f\colon(A,\Aa)\to(H,\Bb(H))$ is measurable iff
$f^{-1}(U)\in\Aa$ for all open sets $U\subset H$;
see e.g.~\cite[Thm.\,1.20]{Salamon:2016a}.
\end{remark}

\begin{definition}
A map $f\colon A\to H$ is \textbf{strongly measurable}
if there is a sequence of simple functions $s_k\colon A\to H$
with $\lim_{k\to\infty} s_k=f$ pointwise on $A$.
\end{definition}

\begin{theorem}[{Pettis~\cite{Pettis:1938a}}]\label{thm:Pettis-Neerven}
Let $H$ be a separable Hilbert space
and $(A,\Aa)$ a measurable space.
For a Hilbert space valued map $f: A\to H$
the following assertions are equivalent.
\begin{itemize}
\item[(1)]
  Every function $\inner{f}{x}\colon( A,\Aa)\to(\R,\Bb(\R))$, where $x\in H$,
  is measurable.
\item[(2)]
  The map $f \colon( A,\Aa)\to(H,\Bb(H))$ is measurable.
\item[(3)]
  The map $f \colon( A,\Aa)\to(H,\Bb(H))$ is strongly measurable.
\end{itemize}
\end{theorem}

\begin{proof}
We closely follow the arguments of Neerven~\cite{Neerven:2007a}.
\\
Let $f: A\to H$ be a Hilbert space valued map.

\smallskip\noindent
(1)$\Rightarrow$(3):
As $H$ is separable, so is the unit sphere, hence we can pick a dense
sequence $(e_n)_{n=1}^\infty$ in the unit sphere of $H$.
For any $x\in H$ the function
\begin{equation}\label{eq:meas-1}
   F_x\colon( A,\Aa)\to ([0,\infty),\Bb),\quad
   t\mapsto \abs{f(t)-x}
\end{equation}
is measurable. To prove this we write
\begin{equation*}
\begin{split}
   \abs{f(t)-x}
   =
   \sup_{\abs{e}=1} \abs{\INNER{f(t)-x}{e}}
   &=\sup_{n\in\N} \abs{\INNER{f(t)-x}{e_n}}\\
   &=\sup_{n\in\N} \abs{\INNER{f(t)}{e_n}-\INNER{x}{e_n}} .
\end{split}
\end{equation*}
The function $t\mapsto \INNER{f(t)}{e_n}$ is measurable by~(1).
Adding to a measurable function a constant term, here
$\INNER{-x}{e_n}$, preserves measurability.
Absolute value $\abs{\cdot}\colon\R\to[0,\infty)$ is continuous, thus
Borel measurable $\abs{\cdot}\colon(\R,\Bb)\to([0,\infty),\Bb)$.
Measurability is preserved under composition, hence
$$
   ( A,\Aa)\to (\R,\Bb)\stackrel{\abs{\cdot}}{\to} ([0,\infty),\Bb),\quad
   t\mapsto \abs{\INNER{f(t)}{e_n}-\INNER{x}{e_n}} 
$$
is measurable.
Taking the supremum of countably many measurable functions
produces a measurable function;
see e.g.~\cite[Thm.\,1.24\,(ii) and Def.\,1.17\,(ii)]{Salamon:2016a}.
This proves that the function~(\ref{eq:meas-1}) is measurable.

\smallskip
As $H$ is separable we can pick a dense
sequence $(x_n)_{n=1}^\infty$ in $H$.
Define a sequence of functions
$$
   d_n\colon H\to\{x_1,\dots,x_n\}
   ,\quad
   y\mapsto d_n(y):=x_{k(n,y)}
$$
as follows. For each $y\in H$ let $k(n,y)$ be the least integer
$k\in \{1,\dots,n\}$ which minimizes the distance to $y$, that is
$$
   \abs{y-x_k}=\min_{j\in\{1,\dots,n\}}\abs{y-x_j}
   ,\qquad
   \forall\ell<k\colon \abs{y-x_\ell}>\abs{y-x_k} .
$$
Observe that
\begin{equation}\label{eq:Pettis-1-3}
   \forall y \in H\colon\quad
   \lim_{n\to\infty}\abs{d_n(y)-y}=0
\end{equation}
since $(x_n)_{n=1}^\infty$ is dense in $H$.
Define a sequence of functions
\begin{equation*}
\begin{split}
   f_n\colon A&\to \{x_1,\dots,x_n\}\subset H
\\
   t&\mapsto d_n(f(t))=x_{k(n,f(t))} .
\end{split}
\end{equation*}
We show that these functions are a) simple and b) converge pointwise
to $f$. To prove a) observe that
for any $k\in\{1,\dots,n\}$ there is the identity
\begin{equation*}
\begin{split}
   A_k^n:
%1
    &=
   f_n^{-1}(x_k)
\\
%2
   &=\{t\in A\mid x_{k(n,f(t))}=x_k\}
\\
%3
   &=\biggl\{t\in A\;\Big|\;
   \abs{f(t)-x_k}
   =\min_{j\in\{1,\dots,n\}}\abs{f(t)-x_j}\biggr\}
   \\
   &\quad\bigcap
   \biggl\{t\in A\;\Big|\;
   \underbrace{\abs{f(t)-x_\ell}}_{=:F_{x_\ell}(t)}
   >\abs{f(t)-x_k}
   \text{ for $\ell=1,\dots,k-1$}\biggr\}
\\
%4
   &=\biggl( F_{x_k} - \min_{j\in\{1,\dots,n\}} F_{x_j}\biggr)^{-1}(0)
   \quad\cap \bigcap_{\ell=1}^{k-1}
   \biggl( F_{x_\ell} - F_{x_k}\biggr)^{-1}(0,\infty) .
\end{split}
\end{equation*}
Note that the sets on the right hand side are in the Lebesgue
$\sigma$-algebra, in symbols $A_k^n\in \Aa$:
Indeed functions of the form~(\ref{eq:meas-1}) are measurable.
Moreover, minima and differences of measurable functions are measurable,
see e.g.~\cite[Thm.\,1.24]{Salamon:2016a},
and the complement of $\{0\}$ as well as $(0,\infty)$ are open sets,
thus Borel measurable sets.
This proves that each function $f_n$ is simple.
To prove b) note that for each $t\in A$ we have
$$
   \lim_{n\to\infty}\abs{f_n(t)-f(t)}
   =\lim_{n\to\infty}\abs{d_n(f(t))-f(t)}
   \stackrel{\text{(\ref{eq:Pettis-1-3})}}{=}0 .
$$
This concludes the proof of assertion~(3) in the theorem.

\smallskip\noindent
(3)$\Rightarrow$(2):
To prove that $f^{-1}(B)\in \Aa$ for every $B\in \Bb(H)$
it suffices to show that $f^{-1}(U)\in\Aa$ for every open set $U$ in $H$;
see Remark~\ref{rem:meas}.
To this end let $U\subset H$ be open and, by~(3), choose a sequence
of simple functions $s_n\colon A\to H$ converging pointwise to $f$.
For $r>0$ define $U_r:=\{x\in U\mid \mathrm{dist}(x,U^{\rm c})>r\}$,
where the closed set $U^{\rm c}:=H\setminus U$ is the complement of $U$.
Then $s_n^{-1}(U_r)\in\Aa$ for each $n\ge1$, since $U_r$
is open in $H$ and simple functions are measurable. The equality
$$
   f^{-1}(U)
   =\bigcup_{m\ge1}\bigcup_{n\ge1}\bigcap_{k\ge n}
   \underbrace{s_k^{-1}(U_{\frac{1}{m}})}_{\in\Aa} ,
$$
implies that $f^{-1}(U)$ lies in $\Aa$ since the right hand side does:
Indeed $\sigma$-algebras are closed under countable unions
and intersections, see e.g.~\cite[\S\,1.1]{Salamon:2016a}.
It remains to prove the equality of sets.

`$\subset$' Let $t\in f^{-1}(U)$.
Since $U$ is open and $f(t)\in U$,
it holds that $d:=\mathrm{dist}(f(t),U^{\rm c})>0$.
Since $s_k(t)$ converges to $f(t)$, as $k\to\infty$, there exists $n_t\in\N$
such that $\mathrm{dist}(f(t),s_k(t))<\frac{d}{2}$
for every $k\ge n_t$. Choose $m_t$ such that $\frac{1}{m_t}<\frac{d}{2}$.
We claim that $s_k(t)\in U_{\frac{1}{m_t}}$ for every $k\ge n_t$.
To see this pick $x\in U^{\rm c}$ and let $k\ge n_t$.
Then by the triangle inequality
$$
   d
   \le \mathrm{dist}(f(t),x)
   \le \mathrm{dist}(f(t),s_k(t))+\mathrm{dist}(s_k(t),x)
   <\tfrac{d}{2}+\mathrm{dist}(s_k(t),x).
$$
Thus $\mathrm{dist}(s_k(t),x)>\frac{d}{2}$.
Since $x\in U^{\rm c}$ was arbitrary
$\mathrm{dist}(s_k(t),U^{\rm c})\ge \frac{d}{2}>\frac{1}{m_t}$.
Hence $s_k(t)\in U_{\frac{1}{m_t}}$ for every $k\ge n_t$.
In other words, the element $t$ lies in
$$
   t\in\bigcap_{k\ge n_t} s_k^{-1}(U_{\frac{1}{m_t}})
$$
for the particular $m_t$ and $n_t$,
therefore it lies in the union over all $m$ and $n$, in symbols
$$
   t\in \bigcup_{m\ge1}\bigcup_{n\ge1}\bigcap_{k\ge n}
   s_k^{-1}(U_{\frac{1}{m}})
$$
`$\supset$'
Let $t$ be element of the right hand side.
Then there exist $m_t$ and $n_t$ with
$$
   t\in\bigcap_{k\ge n_t} s_k^{-1}(U_{\frac{1}{m_t}}) .
$$
This means that $\mathrm{dist}(s_k(t),U^{\rm c})>\frac{1}{m_t}$
whenever $k\ge n_t$.
Since $\lim_{k\to\infty}s_k(t)=f(t)$ it follows that
$\mathrm{dist}(f(t),U^{\rm c})\ge\frac{1}{m_t}$.
But this means that $f(t)\in U$.

\smallskip\noindent
(2)$\Rightarrow$(1): 
True by two facts in measure theory:
(i) Continuous maps are Borel measurable,
here $\INNER{\cdot}{x}\colon (H,\Bb)\to (\R,\Bb)$.
(ii) The composition of a measurable map
with a measurable map, here with $f:(A,\Aa)\to(H,\Bb(H))$,
is measurable.

\smallskip
The proof of Theorem~\ref{thm:Pettis-Neerven} is complete.
\end{proof}

%\newpage
%%%%%%%%%%%%%%%%%%%%%%%%%%%%%%%%%%%
%%%%%%% Subsection  %%%%%%%%%%%%%%%%%%%
%%%%%%%%%%%%%%%%%%%%%%%%%%%%%%%%%%%
\subsection{Lebesgue measure and really simple functions}
\label{sec:Lebesgue}

%%%%%%%%%%%%%%%%%%%%%%%%%%%%%%%%%%%
%%%%%%% Subsection  %%%%%%%%%%%%%%%%%%%
%%%%%%%%%%%%%%%%%%%%%%%%%%%%%%%%%%%
\subsubsection*{Lebesgue measure}

\begin{definition}\label{def:Borel}
Let $\Ii$ be the set of all open intervals $(a,b)\subset \R$
with $a<b\in\R$.
The smallest $\sigma$-algebra containing $\Ii$ is the \textbf{Borel
\boldmath$\sigma$-algebra $\Bb=\Bb(\R)$},
in symbols $\Sigma_\Ii=\Bb$; see e.g.~\cite[Thm.\,IX.1.11]{amann:2009a}.
The Lebesgue measure $\lambda$ on the Borel $\sigma$-algebra $\Bb$ is
the unique translation invariant measure that satisfies
$\lambda((a,b)):=b-a$ whenever $a<b$ are real numbers;
see e.g.~\cite[Thm.\,2.1]{Salamon:2016a}.
\end{definition}

\begin{definition}\label{def:Lebesgue}
A subset $A$ of $\R$ is called \textbf{Lebesgue measurable}
if there exist $A_- \subset A\subset A_+$ where $A_-,A_+\in\Bb$
such that $\lambda(A_+\setminus A_-)=0$.
For a Lebesgue measurable set $A$ one defines its Lebesgue measure
$\lambda(A):=\lambda(A_-){\color{gray}\,=\lambda(A_+)}$.

The family $\Aa$ which consists of all Lebesgue measurable sets
is a $\sigma$-algebra on $\R$ which contains the Borel $\sigma$-algebra
$\Bb$. It is called the \textbf{Lebesgue \boldmath$\sigma$-algebra}.
The map $\lambda\colon\Aa\to[0,\infty]$ is a measure in the sense that
$
   \lambda(\cup_{i=1}^\infty A_i)
   =\sum_{i=1}^\infty \mu(A_i)
$
whenever $A_i$ is a sequence of pairwise disjoint Lebesgue measurable sets.
One calls $\lambda$ the \textbf{Lebesgue measure on \boldmath$\R$}.
\end{definition}

%%%%%%%%%%%%%%%%%%%%%%%%%%%%%%%%%%%
%%%%%%% Subsubsection  %%%%%%%%%%%%%%%%
%%%%%%%%%%%%%%%%%%%%%%%%%%%%%%%%%%%
\subsubsection*{Really simple functions}

Let $H$ be a separable Hilbert space with induced norm $\abs{\cdot}$.
A function $r\colon\R\to H$ is called \textbf{really simple}
if it is a finite sum of the form
$$
   r(x)=\sum_{k=1}^N \chi_{I_k} x_k
   ,\qquad
   I_k=(a_k,b_k),\quad x_k\in H,
   \quad k=1,\dots,N .
$$

Note that $\nu(I_k)=b_k-a_k$ is finite.
Since intervals are Lebesgue measurable
$r$ is in particular a simple function with respect to Lebesgue measurable
space $(\R,\Aa)$.
On the other hand, any simple function $s\colon(a,b)\to H$
(see Definition~\ref{def:simple} for $A=(a,b)$ and $\Aa_{(a,b)}:=\Aa\cap(a,b)$)
can be arbitrarily well approximated in norm by really simple functions.

\begin{theorem}\label{thm:really-simple}
For every simple function $s\colon(a,b)\to H$ and every $\eps>0$ there
exists a really simple function $r_\eps\colon(a,b)\to H$ such that
\begin{equation}\label{eq:r-s}
   \int_a^b\abs{s(t)-r_\eps(t)}\, dt
   <\eps.
\end{equation}
\end{theorem}

\begin{proof}
See e.g.~\cite[Thm.\,1.18]{lieb:2001a}
for real valued functions. There simple functions are of the form
$\sum_{k=1}^N \chi_{A_k} x_k$ and the sets $A_k$
are required to be of finite measure.
We choose as total space $A=(a,b)$, and not $A=\R$,
so that this condition is satisfied automatically.
Indeed $\lambda(A_k)\le\lambda((a,b))=b-a<\infty$.
Since simple functions take on only finitely many values,
the proof for Hilbert valued functions remains the same.
\end{proof}

%\newpage
%%%%%%%%%%%%%%%%%%%%%%%%%%%%%%%%%%%
%%%%%%% Section  %%%%%%%%%%%%%%%%%%%%%
%%%%%%%%%%%%%%%%%%%%%%%%%%%%%%%%%%%
\subsection{Bochner integral}\label{sec:Bochner}

In this section $H$ is a separable Hilbert space
and $I=(a,b)$ a bounded interval.
By $\Aa$ we denote the $\sigma$-algebra
of Lebesgue measurable subsets of $\R$
and $\Aa_{(a,b)}:=\Aa\cap(a,b)$.
By $\abs{\cdot}$ we denote the norm in $H$ induced by the inner product.

\begin{definition}\label{def:Bochner-int}
A map $f\colon\R\to H$ is called \textbf{Bochner integrable}
if it has the following two properties:
\begin{itemize}\setlength\itemsep{0ex} 
\item[(i)]
  $\forall x\in H$ the function $\inner{f}{x}\colon(\R,\Aa)\to(\R,\Bb(\R))$
  is measurable.
\item[(ii)]
  The integral $\int_\R \abs{f(t)}\, dt$ is finite.
\end{itemize}
\end{definition}

\begin{remark}
The map $f\colon(\R,\Aa)\to(H,\Bb(H))$ is measurable,
by property (i) and Theorem~\ref{thm:Pettis-Neerven}.
Since $\abs{\cdot}\colon H\to\R$ is continuous,
the composition
$$
   \abs{f(\cdot)} \colon
   (\R,\Aa)
   \stackrel{f}{\to} (H,\Bb(H))
   \stackrel{\abs{\cdot}}{\to} (\R,\Bb(\R))
$$
is measurable as well.
Therefore the usual scalar Lebesgue integral
$$
   \int_{-\infty}^\infty \abs{f(t)}\, dt
   \in [0,\infty)
$$
is well defined.
\end{remark}

In this section we explain how to define the
\textbf{Bochner integral \boldmath$\dashint$}
$$
   \dashint_I f\, dt \in H
   ,\quad
   I=(a,b)
$$
over bounded intervals ($a<b\in\R$) and show the estimate
$
   \abs{\dashint_I f\, dt}
   \le \int_I \abs{f(t)}\, dt
$.
The symbol $\dashint$ is merely meant to distinguish
the Hilbert valued Bochner integral from the usual real valued integral.

\begin{definition}[Bochner integral of simple functions]
\label{def:Bochner-s}
Consider a simple function
$s=\sum_{k=1}^N \chi_{A_k} x_k\colon\R\to H$ where $A_k\in \Aa$
and $x_k\in H$.
Its Bochner integral over a bounded interval $I=(a,b)$ is defined by
\begin{equation}\label{eq:Bochner-simple-T}
   \dashint_I s\, dt
   :=\dashint_\R \chi_I s\, dt
   :=\sum_{k=1}^N \lambda(A_k\cap I) x_k \in H .
\end{equation}
Here $\lambda$ is the Lebesgue measure on $\R$
and $\mu(A_k\cap I)\le\mu(I)=b-a$ is finite.
It is routine to show that the Bochner integral
of simple functions
\begin{itemize}\setlength\itemsep{0ex} 
\item[a)]
  is independent of the representation of the simple function~$s$;
%  see e.g.~\cite[p.\,81 Rmk.\,X.2.1\,(b)]{amann:2009a};
\item[b)]
  is linear on the real vector space of simple functions;
%  see e.g.~\cite[p.\,81 Rmk.\,X.2.1\,(c)]{amann:2009a};
\item[c)]
  satisfies the inequality
  $\Abs{\dashint_I s\, dt}
  \le \int_I \abs{s(t)}\, dt$.
\end{itemize}
Proofs are given e.g. in~\cite[p.\,81 Rmk.\,X.2.1\,(b,c,e)]{amann:2009a}.
\end{definition}

\begin{definition}[Bochner integral]\label{def:Bochner-f}
Let $f\colon\R\to H$ be Bochner integrable.
Given a bounded interval $I=(a,b)$, by
Theorem~\ref{thm:Pettis-Neerven} (1)$\Leftrightarrow$(3) for
$A=(a,b)$, there is a sequence of simple functions $s_n\colon(a,b)\to H$
such that $\lim_{n\to\infty} s_n=f$ pointwise. 
By Proposition~\ref{prop:Bochner-1} below the limit
\begin{equation}\label{eq:Bochner-int}
   \dashint_a^b f\, dt =\dashint_I f\, dt 
   :=\lim_{n\to\infty} \dashint_I s_n\, dt \in H
\end{equation}
exists in $H$ and is independent of the approximating simple function
sequence for $f$. It is the \textbf{Bochner integral of \boldmath$f$}
with respect to Lebesgue measure.
\end{definition}

\begin{proposition}\label{prop:Bochner-1}
The limit~(\ref{eq:Bochner-int}) exists in $H$ and
it is independent of the approximating
sequence $(s_n)$ for $f$.
\end{proposition}

\begin{remark}
As a consequence of Theorem~\ref{thm:really-simple}
the Bochner integral of a Bochner integrable
function $f\colon\R\to H$ over bounded intervals $I$,
is given as the limit of Bochner integrals of really simple
functions $r_n$, in symbols
\begin{equation}\label{eq:Bochner-int-r-R}
   \dashint_{-T}^T f\, dt
   =\lim_{n\to\infty} \dashint_{-T}^T r_n\, dt \in H .
\end{equation}
\end{remark}

The proof of Proposition~\ref{prop:Bochner-1}
is based on the following proposition.

\begin{proposition}\label{prop:Bochner-2}
In the presence of the strong measurability property
\begin{itemize}
\item[\rm (i)]
  $f\colon(a,b)\to H$ is a map and
  $s_n\colon(a,b)\to H$ is a sequence of simple functions
  such that $\lim_{n\to\infty} s_n=f$ pointwise
\end{itemize}
the following two properties are equivalent
$$
   \text{\rm (ii) } \int_a^b \abs{f(t)}\, dt <\infty
   \quad\Longleftrightarrow\quad
   \text{\rm (iii) }
   \lim_{n\to\infty}\int_a^b \abs{s_n(t)-f(t)}\, dt=0
$$
and in either case
\begin{equation}\label{eq:Bochner-inequality}
 \biggl| \dashint_a^b f\, dt\biggr|
   \le \int_a^b \abs{f(t)}\, dt .
\end{equation}
\end{proposition}

\begin{proof}
\smallskip\noindent
(i,iii)$\Rightarrow$(ii):
Add zero and use the triangle inequality and~(iii) to obtain
$$
   \int_a^b \abs{f(t)}\, dt
   \le \int_a^b \abs{f(t)-s_n(t)}\, dt
   +\int_a^b \underbrace{\abs{s_n(t)}}_{\le 1}\, dt
   \stackrel{n\to\infty}{\longrightarrow}
   b-a <\infty.
$$
%\smallskip\noindent
(i,ii)$\Rightarrow$(iii): 
The function defined by
$$
   \sigma_n
   :=\chi_{\{\abs{s_n}\le2\abs{f}\}} s_n
$$
is simple. Given $t$, for large $n$ one has
$\sigma_n(t)=s_n(t)$. Consequently $\sigma_n\to f$ pointwise.
Furthermore, there is the pointwise inequality
$$
   F_n(t)
   :=\abs{\sigma_n(t)}
   \le 2\abs{f(t)}=: G(t)
$$
between integrable (here (ii) enters) scalar functions.
Continuity of the norm yields that
$F_n(t)\to F(t):=\abs{f(t)}$ pointwise, as $n\to\infty$.
Thus the Lebesgue dominated convergence theorem,
see e.g.~\cite[Thm.~X.3.12]{amann:2009a}, applies and yields~(iii) as
well as the following identity to be used right below
$$
   \lim_{n\to\infty}\int_a^b \abs{\sigma_n(t)}\, dt
   =\int_a^b \abs{f(t)}\, dt .
$$

\smallskip\noindent
Final estimate: Suppose~(i,ii).
Step one in what follows is by~(\ref{eq:Bochner-int}),
step two is by continuity of the norm,
for the final inequality see Definition~\ref{def:Bochner-s} c)
\begin{equation*}
\begin{split}
   \biggl| \dashint_a^b f\, dt\biggr|
   &=\biggl|\lim_{n\to\infty}\dashint_a^b \sigma_n\, dt\biggr|
\\
   &=\lim_{n\to\infty}\biggl|\dashint_a^b \sigma_n\, dt\biggr|
\\
   &\stackrel{3}{\le}\lim_{n\to\infty}\int_a^b \abs{\sigma_n(t)}\, dt .
\end{split}
\end{equation*}
The previous displayed identity proves the estimate and
Proposition~\ref{prop:Bochner-2}.
\end{proof}

\begin{proof}[Proof of Proposition~\ref{prop:Bochner-1}]
The proof has two steps.

\medskip\noindent
\textbf{Step 1.}
  The elements $y_{s_n}:=\dashint_a^b s_n\, dt$
  form a Cauchy sequence in $H$.

\smallskip\noindent
To see this observe that
\begin{equation*}
\begin{split}
   \abs{y_{s_k}-y_{s_\ell}}
   &=\biggl|\dashint_a^b \left( s_k-s_\ell\right)  dt\biggr|
\\
   &\stackrel{\rm c)}{\le}\int_a^b \abs{s_k(t)-s_\ell(t)}\, dt
\\
   &\stackrel{\Delta}{\le}\int_a^b \abs{s_k(t)-f(t)}\, dt
   +\int_a^b \abs{f(t)-s_\ell(t)}\, dt
\end{split}
\end{equation*}
The assertion of Step~1 now follows by part~(iii) of
Proposition~\ref{prop:Bochner-2}.

\medskip\noindent
\textbf{Step 2.}
Let $s_n$ and $\sigma_n$ be sequences of
simple functions both converging pointwise to $f$.
Then both sequences of Bochner integrals
$y_{s_n}$ and $z_{\sigma_n}$ converge in $H$ and have the same limit.

\smallskip\noindent
By Step~1 there are elements $y$ and $z$ of $H$ such that
$y_{s_n}\to y$ and $z_{\sigma_n}\to z$, as $n\to\infty$.
Observe that by continuity of the norm we get inequality one
\begin{equation*}
\begin{split}
   \abs{y-z}
   &\le\lim_{n\to\infty}
   \biggl|\dashint_a^b s_n\, dt-\dashint_a^b \sigma_n\, dt\biggr|
\\
   &\stackrel{\text{c)}}{\le}\lim_{n\to\infty}\int_a^b
   \abs{s_n(t)-\sigma_n(t)} \, dt
\\
   &\stackrel{\Delta}{\le}
   \lim_{n\to\infty}\int_a^b\abs{s_n(t)-f(t)} \, dt
   +\lim_{n\to\infty}\int_a^b\abs{f(t)-\sigma_n(t)} \, dt
\\
   &\stackrel{\text{(iii)}}{=}0 .
\end{split}
\end{equation*}
The penultimate step is by the triangle inequality after adding zero
and by linearity of the real-valued integral.
The ultimate step is by Proposition~\ref{prop:Bochner-2}.
%x
This proves Step~2 and Proposition~\ref{prop:Bochner-1}.
\end{proof}

\begin{lemma}\label{le:hhbhj6478}
Let $f\colon \R\to H$ be Bochner integrable, $x\in H$, and $a<b\in\R$.
Then
$$
   \INNER{\dashint_a^b f \, dt}{x}
   =\int_a^b\INNER{f(t)}{x} dt .
$$
\end{lemma}

\begin{proof}
By linearity of the inner product this is true for simple functions,
hence the Lemma follows by approximating $f$ by simple functions
in view of the continuity of the inner product.
\end{proof}

%\newpage
%%%%%%%%%%%%%%%%%%%%%%%%%%%%%%%%%%%
%%%%%%% Subsection  %%%%%%%%%%%%%%%%%%%
%%%%%%%%%%%%%%%%%%%%%%%%%%%%%%%%%%%
\subsection{Hilbert space valued Lebesgue and Sobolev spaces}
\label{sec:Lebesgue-Sobolev}

\begin{definition}
Let $p\in[1,\infty)$.
The space $L^p=L^p(\R,H)$ consists of equivalence classes $[f]$ of
functions $f\colon\R\to H$ satisfying the following two conditions.
\begin{itemize}\setlength\itemsep{0ex} 
\item[(i)]
  Every function $\inner{f}{x}\colon(\R,\Aa)\to(\R,\Bb)$, where $x\in H$,
  is measurable.
\item[(ii)]
  The \textbf{\boldmath$L^p$-norm}
  $\norm{f}_p:=\norm{f}_{L^p(\R,H)}:=\left(\int_\R\abs{f(t)}^p dt\right)^{1/p}<\infty$
  is finite.
\end{itemize}
Two functions $f$ and $g$ satisfying (i) and (ii) are called
\textbf{equivalent} if they coincide outside a set of Lebesgue measure zero.
To simplify notation we write the equivalence class $[f]$ still as $f$.
On $L^2(\R,H)$ there is the inner product
$$
   \INNER{f}{g}_{L^2(\R,H)}
   :=\int_\R \INNER{f(t)}{g(t)} \, dt .
$$
The vector space $L^2(\R,H)$ endowed with this inner product becomes a
Hilbert space and the $(L^p,\norm{\cdot}_p)$ are Banach spaces;
see e.g.~\cite[App.\,A.2]{Frauenfelder:2021b}.
\end{definition}

\begin{lemma}
For $p\in[1,\infty)$ every $f\in L^p(\R,H)$ restricted to a bounded
interval $I=(a,b)$ is Bochner integrable along $(a,b)$.
\end{lemma}

\begin{proof}
By H\"older
$
   \int_\R \abs{f(t)\chi_{(a,b)}(t)}\, dt
   =\int_a^b \abs{f(t)}\, dt
   \le (b-a)^{1-\frac{1}{p}} \norm{f}_p
   <\infty
$.
\end{proof}

\begin{definition}[Sobolev space]
Let $W^{1,2}(\R,H)$ be the vector
space of all $f\in L^2(\R,H)$ for which there exists
an element $v\in L^2(\R,H)$ such that
\begin{equation}\label{eq:weak-derivative}
     \int_\R\INNER{f(t)}{\dot \varphi(t)} dt
     =-\int_\R\INNER{v(t)}{\varphi(t)} dt
\end{equation}
for every $\varphi\in C^\infty_{\rm c}(\R,H)$.
If such a map $v$ exists, then it is unique and called
the \textbf{weak derivative} of $f$.
We denote $v$ by the symbol $\dot f$ or $f^\prime$.
\\
The vector space $W^{1,2}(\R,H)$ is endowed with the
\textbf{\boldmath$W^{1,2}$-inner product}
\begin{equation}\label{eq:W12-inner}
   \INNER{f}{g}_{W^{1,2}(\R,H)}
   :=\inner{f}{g}_{L^2(\R,H)}
   +\inner{\dot f}{\dot g}_{L^2(\R,H)} .
\end{equation}
and the induced \textbf{\boldmath$W^{1,2}$-norm},
notation
$\norm{f}_{W^{1,2}(\R,H)}:=(\INNER{f}{g}_{W^{1,2}(\R,H)})^{\frac12}$.
\\
The vector space $W^{1,2}(\R,H)$ endowed with this inner product
is a Hilbert space; see e.g.~\cite[Prop.\,A.8]{Frauenfelder:2021b}.
\end{definition}

%%%%%%%%%%%%%%%%%%%%%%%%%%%%%%%%%%%
%%%%%%% Subsection  %%%%%%%%%%%%%%%%%%%
%%%%%%%%%%%%%%%%%%%%%%%%%%%%%%%%%%%
\subsubsection*{The Banach space \boldmath$L^\infty(\R,H)$}

\begin{definition}[Essentially bounded functions]
Let $\Ll^\infty(\R,H)$ be the set of all functions $f\colon\R\to H$
which have the following two properties:
\begin{itemize}\setlength\itemsep{0ex} 
\item[(a)]
  $\forall x\in H$ the function $\inner{f}{x}\colon(\R,\Aa)\to(\R,\Bb(\R))$
  is measurable.
\item[(b)]
  There exists $r\in[0,\infty)$ such that
  $\{\norm{f}>r\}$ is of Lebesgue measure zero.
\end{itemize}
The functions $f\in \Ll^\infty(\R,H)$
are called \textbf{essentially bounded}.
\end{definition}

\begin{proposition}\label{prop:Ll-infty}
The space $\Ll^\infty(\R,H)$ is linear and complete with respect to the
semi-norm $\norm{f}_\infty=\norm{f}_{\Ll^\infty(\R,H)}$ defined as the
infimum of essential bounds
$$
   \norm{f}_\infty
   :=\inf\bigl\{r\ge0\mid
   \text{the set $\{\norm{f}>r\}$ has Lebesgue measure zero}\bigr\}
   .
$$
\end{proposition}

\begin{proof}
To see that $\Ll^\infty(\R,H)$ is a real vector space
note that property~(a) follows from the fact that the space of measurable
functions $\inner{f}{x}\colon(\R,\Aa)\to(\R,\Bb)$ is a vector space.
Concerning property~(b), given $\alpha\in\R$ and $f,g\in \Ll^\infty(\R,H)$, then
$\norm{\alpha f+g}_\infty\le\Abs{\alpha}\norm{f}_\infty+\norm{g}_\infty$;
see e.g.~\cite[Rmk.\,X.4.1\,(c)]{amann:2009a}.
For a proof that $\Ll^\infty(\R,H)$ is complete
see e.g.~\cite[Thm.\,X.4.6]{amann:2009a}.
% REAL-VALUED: \cite[Thm.\,4.9]{Salamon:2016a}.
\end{proof}

On $\Ll^\infty(\R,H)$ consider the equivalence relation where $f\sim g$ if the two
maps are equal outside a set of Lebesgue measure zero. On the quotient space
$$
   L^\infty(\R,H):=\Ll^\infty(\R,H)/\sim
$$
the semi-norm $\norm{\cdot}_\infty$ is a norm.
Hence $L^\infty(\R,H)$ is a Banach space by Proposition~\ref{prop:Ll-infty}.
To ease notation we still denote the elements of $L^\infty(\R,H)$ by~$f$.

%\newpage% .\newpage
%%%%%%%%%%%%%%%%%%%%%%%%%%%%%%%%%%%
%%%%%%% Subsection  %%%%%%%%%%%%%%%%%%%
%%%%%%%%%%%%%%%%%%%%%%%%%%%%%%%%%%%
\subsection{Differentiable compactly supported approximation}
\label{sec:diff-app}

%%%%%%%%%%%%%%%%%%%%%%%%%%%%%%%%%%%
%%%%%%% Subsubsection  %%%%%%%%%%%%%%%%
%%%%%%%%%%%%%%%%%%%%%%%%%%%%%%%%%%%
\subsubsection*{Convolution}

Fix a smooth function $\rho\colon\R\to[0,\infty)$
which is supported in $[-1,1]$, is symmetric $\rho(\cdot)=\rho(-\cdot)$,
and satisfies $\int_\R\rho(t)\, dt =1$;
see e.g.~\cite[Ex.\,X.7.12\,b)]{amann:2009a}.
Define
$$
   \rho_\mu(t):=\tfrac{1}{\mu}\rho(\tfrac{s}{\mu})
$$
for $\mu>0$. This function
is supported in $[-\mu,\mu]$ and has the properties that
\begin{equation}\label{eq:rho-mu}
   \int_\R\rho_\mu(t)\, dt
   =\norm{\rho_\mu}_{L^1(\R)}
   =\norm{\rho}_{L^1(\R)}
   =1
\end{equation}
and that $\int_{\R\setminus(-r,r)}\rho_\mu(t)\, dt\to 0$,
as $\mu\to 0$, for any given $r>0$.

\begin{definition}
Let $f\colon \R\to H$ be a function which is Bochner integrable.
The \textbf{convolution} of $f$
by $\rho_\mu$ at time $t\in\R$ is the Bochner integral
\begin{equation}\label{eq:convolution}
   f_\mu(t)
   :=(f*\rho_\mu)(t)
   :=\dashint_{-\mu}^\mu f(t-s) \rho_\mu(s) \, ds
   =\dashint_{t-\mu}^{t+\mu} f(s) \rho_\mu(t-s) \, ds .
\end{equation}
\end{definition}

\begin{lemma}
The Bochner integral~(\ref{eq:convolution}), thus convolution, is well defined.
\end{lemma}

\begin{proof}
(i)~The integrand is measurable.
Indeed, given $t\in\R$ and $x\in H$, then the map
$(\R,\Aa)\to(\R,\Bb)$, $s\mapsto \INNER{\rho_\mu(s)  f(t-s)}{x}
=\rho_\mu(s) \INNER{f(t-s)}{x}$,
is measurable since $\rho_\mu$ is continuous, and therefore
measurable, the second term is measurable by assumption that $f$ is
Bochner integrable, see Definition~\ref{def:Bochner-int},
and the product of measurable functions is measurable.
(ii)~By Young's inequality for the convolution of real valued
functions the norm
$
   \norm{\rho_\mu*f}_{L^1(\R,H)}
   \le\norm{\rho_\mu}_{L^1(\R,H)} \norm{f}_{L^1(\R,H)}
   =\norm{f}_{L^1(\R,H)}
$
is finite.
\end{proof}

\begin{lemma}\label{le:gfhjhgf475}
For $f\in L^2(\R,H)$ the convolution
$f_\mu:=f*\rho_\mu$ lies in $C^1(\R,H)$ and the derivative
is $\frac{d}{dt} (f*\rho_\mu)=f*\frac{d}{dt}\rho_\mu$ whenever $\mu>0$.
\end{lemma}

\begin{proof}
To see this let $\mu>0$. 
Since $\rho$, and therefore $\rho_\mu$, is $C^1$ and compactly
supported its derivative $\dot\rho_\mu$ is uniformly continuous,
so the following is true: Given $\eps>0$, there exists $\delta_\eps>0$
such that if $\abs{\sigma-\tilde\sigma}<\delta_\eps$, then 
\begin{equation}\label{eq:gh77ff}
   \abs{\dot\rho_\mu(\sigma)-\dot\rho_\mu(\tilde\sigma)}<\eps.
\end{equation}
Choose a nonzero number $h\in(-\delta_\eps,\delta_\eps)$.
For $t\in\R$ we calculate
\begin{equation*}
\begin{split}
    \Abs{\tfrac{\rho_\mu(t+h)-\rho_\mu(t)}{h}
   -\dot\rho_\mu(t)}
   &=\biggl|\frac{1}{h} \int_0^1
   \underbrace{\tfrac{d}{d\tau} \rho_\mu(t+\tau
     h)}_{\dot\rho_\mu(t+\tau h) \cdot h} d\tau
   -\dot\rho_\mu(t)\biggr|
\\
   &\le\int_0^1
   \underbrace{\Abs{\dot\rho_\mu(t+\tau h)-\dot\rho_\mu(t)}}
      _{\text{$<\eps$ by (\ref{eq:gh77ff})}}
   d\tau
\\
   &<\eps .
\end{split}
\end{equation*}
Use first estimate~(\ref{eq:Bochner-inequality})
and then the previous estimate to obtain
\begin{equation*}
\begin{split}
   &\Abs{
   \tfrac{f_\mu(t+h)-f_\mu(t)}{h}-f*\dot\rho_\mu(t)}
\\
   &=\Abs{
   \dashint_{t-\mu-\delta}^{t+\mu+\delta} f(s)\left(\tfrac{\rho_\mu(t+h-s)-\rho_\mu(t-s)}{h}
   -\dot\rho_\mu(t-s)\right) ds}
\\
   &\le
   \int_{t-\mu-\delta}^{t+\mu+\delta}\Abs{f(s)}
   \underbrace{
   \Abs{\tfrac{\rho_\mu(t+h-s)-\rho_\mu(t-s)}{h}-\dot\rho_\mu(t-s)}}_{<\eps}
   ds
\\
   &<\eps \sqrt{2(\mu+\delta)} \norm{f}_{L^2(\R,H)} .
\end{split}
\end{equation*}
In the last step we used H\"older's inequality for the product integrand
$1\cdot\abs{f(s)}$.
This proves that the derivative of $f_\mu$ is given by $f*\dot\rho_\mu$.

It remains to show continuity of the derivative.
Given $t\in\R$ and $h\in(-\delta_\eps,\delta_\eps)$ as above,
we calculate
\begin{equation*}
\begin{split}
   &\Abs{\dot f_\mu(t+h)-\dot f_\mu(t)}
\\
   &=\Abs{\dashint_{t-\mu-\delta}^{t+\mu+\delta}
   f(s)\bigl(\dot\rho_\mu(t+h-s)-\dot\rho_\mu(t-s)\bigr) ds
   }
\\
   &\le \int_{t-\mu-\delta}^{t+\mu+\delta}\Abs{f(s)}
   \underbrace{\Abs{\dot\rho_\mu(t+h-s)-\dot\rho_\mu(t-s)}}
      _{\text{$<\eps$ by (\ref{eq:gh77ff})}}
   ds
\\
   &<\eps \sqrt{2(\mu+\delta)} \norm{f}_{L^2(\R,H)} .
\end{split}
\end{equation*}
This proves continuity of $\dot f_\mu$ and concludes the proof of Claim~1.
\end{proof}

%\newpage
%%%%%%%%%%%%%%%%%%%%%%%%%%%%%%%%%%%
%%%%%%% Subsubsection  %%%%%%%%%%%%%%%%
%%%%%%%%%%%%%%%%%%%%%%%%%%%%%%%%%%%
\subsubsection*{Lebesgue space}

\begin{theorem}[$C^1_{\rm c}$ approximation of $L^2$]\label{thm:Leb-dense}
For any $f\in L^2(\R,H)$ it holds that
\begin{equation}\label{eq:h66fgh34}
   \lim_{\mu\to 0}\norm{f_\mu-f}_{L^2(\R,H)}=0
   ,\qquad
   f_\mu\stackrel{\text{\rm(\ref{eq:convolution})}}{:=}f*\rho_\mu
   \in C^1(\R,H) .
\end{equation}
More is true, namely $C^1_{\rm c}(\R,H)$
is a dense subset of $L^2(\R,H)$.
\end{theorem}

One can replace $C^1_{\rm c}(\R,H)$ by $C^\infty_{\rm c}(\R,H)$ with
some more work, but for our purposes this is not needed.

\begin{proof}
The proof has six steps.
For $I\subset\R$ we often abbreviate $L^1_I:=L^1(I,H)$.

\smallskip
\noindent
\textbf{Step~1 (Reduction to {\boldmath$L^1$}).}
It suffices to prove approximation~(\ref{eq:h66fgh34}) in $L^1$.

\smallskip
\noindent
To see this pick $f\in L^2(\R,H)$ and $\eps>0$.
Then there exists $h=h_\eps>0$ such that
$\norm{f}_{L^2(\{\abs{f}> h\})}<\frac{\eps}{4}$.
For the truncated function $f^h:=\chi_{\norm{f}<h} f$ we get
\begin{equation*}
\begin{split}
   \norm{f^h-f}_{L^2_\R}
   &=\left(\int_{\{\abs{f}> h\}} f^2(t)\, dt\right)^{1/2}
   <\frac{\eps}{4} .
\end{split}
\end{equation*} 
By assumption approximation holds in $L^1$, so there exists
$\mu=\mu_\eps>0$ such that
$$
   \norm{(f^h)_\mu-f^h}_1<\frac{\eps^2}{8h} .
$$
Note that $\norm{f^h}_\infty\le h$. Now use this estimate
to get for every $t\in\R$ the estimate
$
   \abs{(\rho_\mu*f^h)(t)}
   \le\int_{-\mu}^\mu \rho_\mu(s)\abs{f^h(t-s)} \, ds
   \le h \norm{\rho_\mu}_1=h
$.
With this we get
\begin{equation*}
\begin{split}
   \norm{(f^h)_\mu-f^h}_2^2
   &=\int_\R \abs{(f^h)_\mu(t)-f^h(t)}^2\, dt
\\
   &\le \left(\norm{(f^h)_\mu}_\infty + \norm{f^h}_\infty\right)
   \int_\R \abs{(f^h)_\mu(t)-f^h(t)}\, dt
\\
   &\le 2h \norm{(f^h)_\mu-f^h}_1 .
\end{split}
\end{equation*}
Using the above three displayed inequalities we estimate
\begin{equation*}
\begin{split}
   \norm{f_\mu-f}_2
   &\le\norm{\rho_\mu*(f-f^h)}_2
   +\norm{(f^h)_\mu-f^h}_2
   +\norm{f^h-f}_2
\\
   &\le\norm{\rho_\mu}_1 \norm{f-f^h}_2
   +\sqrt{2h\norm{(f^h)_\mu-f^h}_1}
   +\norm{f^h-f}_2
\\
   &\le\tfrac{\eps}{2}+\sqrt{\tfrac{\eps^2}{4}}
\\
   &=\eps .
\end{split}
\end{equation*}
In step two we used Young's inequality
$\norm{gh}_r\le\norm{g}_p\norm{h}_q$ 
where $1+\frac{1}{r}=\frac{1}{p}+\frac{1}{q}$ for $r=q=2$ and $p=1$.
In step three we used~(\ref{eq:rho-mu}).

\smallskip
\noindent
\textbf{Step~2 (\boldmath$C^1$ approximation near {\boldmath$\infty$}).}
Pick $f\in L^2(\R,H)$ and $\eps>0$.
There exists $T=T(\eps)>1$ such that
$
   \norm{f_\mu-f}_{L^1_{\R\setminus[-T,T]}}<\tfrac{\eps}{2}
$
whenever $\mu\in(0,1]$.

\smallskip
\noindent
Since $f\in L^1$ there exists $T>1$ such that the integral near infinity
is small
\begin{equation}\label{eq:hfgh34974}
   \norm{f}_{L^1_{\R\setminus[-T+1,T-1]}}<\tfrac{\eps}{4}.
\end{equation}
We estimate
\begin{equation*}%\label{eq:hf88gh34}
\begin{split}
   \norm{f_\mu}_{L^1_{\R\setminus[-T,T]}}
   &=\norm{\rho_\mu*f}_{L^1_{\R\setminus[-T,T]}}
\\
   &=\norm{\rho_\mu*f\chi_{\R\setminus[-T+1,T-1]}}_{L^1_{\R\setminus[-T,T]}}
\\
   &\le \norm{\rho_\mu*f\chi_{\R\setminus[-T+1,T-1]}}_{L^1_{\R}}
\\
   &\le\norm{\rho_\mu}_{L^1_{\R}} \norm{f\chi_{\R\setminus[-T+1,T-1]}}_{L^1_{\R}}
\\
   &=\norm{f}_{L^1_{\R\setminus[-T+1,T-1]}}
\\
   &<\tfrac{\eps}{4}.
\end{split}
\end{equation*}
Step two uses that $\mu\le 1$ and we multiplied by a characteristic
function which is identically $1$ on that domain.
Step three is by monotonicity of the integral.
Step four is by Young's inequality.
Step five uses that $\norm{\rho_\mu}_{L^1(\R)}=1$.
The final step six is by~(\ref{eq:hfgh34974}).

To conclude the proof of Step~2 we estimate
\begin{equation*}%\label{eq:hf88gh34}
\begin{split}
   \norm{f_\mu-f}_{L^1_{\R\setminus[-T,T]}}
   \le\norm{f_\mu}_{L^1_{\R\setminus[-T,T]}}+\norm{f}_{L^1_{\R\setminus[-T,T]}}
   < \tfrac{\eps}{4} +\tfrac{\eps}{4}=\tfrac{\eps}{2}.
\end{split}
\end{equation*}
Here we used~(\ref{eq:hfgh34974}) and the
subsequent estimate

\smallskip
\noindent
\textbf{Step~3 (Approximate \boldmath$f$ by really simple function $r$).}
There is a really simple function $r\colon[-T-1,T+1]\to H$
such that $\norm{f-r}_{L^1([-T-1,T+1],H)}<\tfrac{\eps}{6}$.

\smallskip
\noindent
By Proposition~\ref{prop:Bochner-2}~(iii)
we approximate $f$ by a simple function $s$
which, by~(\ref{eq:r-s}), we approximate by a really simple function $r$.

\smallskip
\noindent
\textbf{Step~4 (\boldmath$C^1$ approximation of $r$).}
There exists a constant $\mu_\eps\in(0,1]$
such that $\norm{r_\mu-r}_{L^1(\R,H)}<\tfrac{\eps}{6}$
whenever $\mu\in(0,\mu_\eps]$.

\smallskip
\noindent

A really simple function
$r\colon [-T,T]\to H$ is of the form
$r=\sum_{k=1}^N \chi_{I_k} x_k$
where each $I_k$ is an interval.
Note that $\abs{(\rho_\mu*\chi_{I_k})-\chi_{I_k}}\le 1$
and that for every $t$ outside of the intervals $(a_k-\mu,a_k+\mu)$ and
$(b_k-\mu,b_k+\mu)$ the function
$\abs{(\rho_\mu*\chi_{I_k})(t)-\chi_{I_k}(t)}=0$ vanishes.
Consequently
$
   \int_{-\infty}^\infty
   \abs{(\rho_\mu*\chi_{I_k})-\chi_{I_k}}\, dt
   \le 4\mu 
$
and therefore $\norm{r_\mu-r}_{L^1(\R,H)}\le 4\mu N \kappa$
where $\kappa:=\max\{\norm{x_1},\dots, \norm{x_k}$.
Hence Step~4 follows by choosing $\mu_\eps<\eps/(24 N\kappa)$.

\smallskip
\noindent
\textbf{Step~5 (\boldmath$C^1$ approximation of $f$).}
Given $\mu\in(0,\mu_\eps]\subset(0,1]$,
then we have $\norm{f_\mu-f}_{L^1(\R,H)}<\eps$.
Equivalently, this proves~(\ref{eq:h66fgh34}).

\smallskip
\noindent
To prove this we estimate
\begin{equation}\label{eq:hfgh34}
\begin{split}
   \norm{f_\mu-r_\mu}_{L^1_{[-T,T]}}
   &=\norm{\rho_\mu*(f-r)}_{L^1_{[-T,T]}}
\\
   &=\norm{\rho_\mu*(f-r)\chi_{[-T-1,T+1]}}_{L^1_{[-T,T]}}
\\
   &\le \norm{\rho_\mu*(f-r)\chi_{[-T-1,T+1]}}_{L^1_\R}
\\
   &\le\norm{\rho_\mu}_1 \norm{(f-r)\chi_{[-T-1,T+1]}}_{L^1_\R}
\\
   &=\norm{f-r}_{L^1_{[-T-1,T+1]}}
\\
   &<\tfrac{\eps}{6}.
\end{split}
\end{equation}
Step two uses that $\supp \rho_\mu\subset [-\mu,\mu]\subset[-1,1]$
and we multiplied by a characteristic function which is identically
$1$ on that domain.
Step four is by Young's inequality.
Step five uses that $\norm{\rho_\mu}_1=\norm{\rho}_1=1$.
The final step six is by Step~3.

Now we decompose $\R=\R\setminus[-T,T]\cup [-T,T]$
and use Step~2 to obtain
\begin{equation*}
\begin{split}
   \norm{f_\mu-f}_{L^1(\R,H)}
   &=\norm{f_\mu-f}_{L^1_{\R\setminus[-T,T]}}
   +\norm{f_\mu-f}_{L^1_{[-T,T]}}
\\
   &\le\tfrac{\eps}{2}
   +\norm{f_\mu-r_\mu}_{L^1_{[-T,T]}}
   +\norm{r_\mu-r}_{L^1_{[-T,T]}}
   +\norm{r-f}_{L^1_{[-T,T]}}
\\
   &<\tfrac{\eps}{2}+\tfrac{\eps}{6}+\tfrac{\eps}{6}+\tfrac{\eps}{6}
   =\eps.
\end{split}
\end{equation*}
The final inequality is by~(\ref{eq:hfgh34})
and Steps three and four.

\smallskip
\noindent
Since $\eps>0$ was arbitrary we get
$\lim_{\mu\to0}\norm{f_\mu-f}_{L^1(\R,H)}=0$.
This also proves $L^2$ convergence by Step~1.
The proof of Step~5 is complete.

\smallskip
\noindent
\textbf{Step~6 (Compact support).}
For any $f\in L^{1,2}(\R,H)$ the sequence
$F_k\in C^1_0(\R,H)$ defined prior to~(\ref{eq:hgfd35g})
converges to $f$ in $L^2$.

\smallskip
\noindent
In~(\ref{eq:hgfd35g}) replace the $W^{1,2}$ norm by the $L^2$ norm.

\smallskip
\noindent
This concludes the proof of Theorem~\ref{thm:Leb-dense}.
\end{proof}

%%%%%%%%%%%%%%%%%%%%%%%%%%%%%%%%%%%
%%%%%%% Subsubsection  %%%%%%%%%%%%%%%%
%%%%%%%%%%%%%%%%%%%%%%%%%%%%%%%%%%%
\subsubsection*{Sobolev space}

\begin{theorem}[$C^1$ approximation of $W^{1,2}$]\label{thm:Sob-dense}
The set $C^1_{\rm c}(\R,H)$ of smooth compactly supported maps
is a dense subset of the Hilbert space $W^{1,2}(\R,H)$.
\end{theorem}

The proof of the theorem uses the following lemma.

\begin{lemma}\label{le:conv-deriv}
For $f\in W^{1,2}(\R,H)$ it holds
$(f*\rho_\mu)^\prime=f^\prime*\rho_\mu$ whenever $\mu>0$.
\end{lemma}

\begin{proof}
Pick $\varphi\in C^\infty_{\rm c}(\R,H)$.
Fix $T>0$ such that $\supp\varphi\subset[-T,T]$.
Abbreviate $I:=[-T-\mu,T+\mu]$.
We compute by definition~(\ref{eq:convolution}) of convolution
\begin{equation*}
\begin{split}
   \int_{\R_t}\INNER{(f^\prime*\rho_\mu)(t)}{\varphi(t)} dt
   &\stackrel{1}{=}\int_{I_t}\int_{I_s}\INNER{f^\prime(s)\rho_\mu(t-s)}{\varphi(t)} dsdt
\\
   &\stackrel{2}{=}\int_{I_s}\int_{I_t}\INNER{f^\prime(s)}{\rho_\mu(t-s)\varphi(t)} dtds
\\
   &\stackrel{3}{=}\int_{I_s}\INNER{f^\prime(s)}{(\rho_\mu*\varphi)(s)} ds
\\
   &\stackrel{4}{=}-\int_{I_s}\INNER{f(s)}
   {\underbrace{(\rho_\mu*\varphi)^\prime(s)}_{(\rho_\mu*\dot \varphi)(s)}} ds
\\
   &\stackrel{5}{=}-\int_{I_s}\int_{I_t}
   \underbrace{\INNER{f(s)}{\rho_\mu(s-t)\dot\varphi(t)}}
   _{\INNER{f(s) \rho_\mu(s-t)}{\dot\varphi(t)}} dtds
\\
   &\stackrel{6}{=}-\int_{\R_t}\INNER{(f*\rho_\mu)(t)}{\dot\varphi(t)} dt .
\end{split}
\end{equation*}
Steps 1, 3, 5, and 6 are by Lemma~\ref{le:hhbhj6478}.
Step 2 is by the Theorem of Fubini, see e.g.~\cite[Thm.\,X.6.16]{amann:2009a},
which applies since the integrand is absolutely integrable: Indeed the integral
\begin{equation*}
\begin{split}
   \int_{I_t}\int_{I_s}\Abs{\INNER{f^\prime(s)\rho_\mu(t-s)}{\varphi(t)}} dsdt
   &\le2(T+\mu)\tfrac{1}{\mu} \norm{\varphi}_\infty\int_{I_s}\Abs{f^\prime(s)}  \, ds
\\
   &\le (2T+2\mu)^{\frac{3}{2}}\tfrac{1}{\mu} \norm{\varphi}_\infty\norm{f^\prime}_2
\end{split}
\end{equation*}
is finite.
Step 3 also uses that $\rho(t)=\rho(-t)$ has been chosen symmetric.
Step 4 is by definition~(\ref{eq:weak-derivative}) of weak derivative.
We then used Lemma~\ref{le:gfhjhgf475} for $H=\R$.
Since $\varphi$ was an arbitrary test function this proves
Lemma~\ref{le:conv-deriv}.
\end{proof}

\begin{proof}[Proof of Theorem~\ref{thm:Sob-dense}]
The proof has two steps.

\smallskip
\noindent
\textbf{Step~1 (\boldmath$C^1$ approximation).}
$C^1(\R,H)\cap W^{1,2}(\R,H)$ is dense in $W^{1,2}(\R,H)$.

\medskip
\noindent
To prove this pick $f\in W^{1,2}(\R,H)$.
In particular $f$ and its weak derivative $f^\prime$ are in $L^2(\R,H)$.
Hence, applying twice Theorem~\ref{thm:Leb-dense}, we have convergence
$$
   f_\mu\stackrel{L^2}{\longrightarrow} f
   ,\qquad
   (f^\prime)_\mu\stackrel{L^2}{\longrightarrow}  f^\prime
   ,\qquad \text{as $\mu\to 0$}.
$$
Since $(f^\prime)_\mu =(f_\mu)^\prime$,
by Lemma~\ref{le:conv-deriv},
this shows that 
\begin{equation}\label{eq:h68776fgh34}
   \lim_{\mu\to 0}\norm{f_\mu-f}_{W^{1,2}(\R,H)}=0
   ,\qquad
   f_\mu\stackrel{\text{\rm(\ref{eq:convolution})}}{:=}f*\rho_\mu
   \in C^1(\R,H) .
\end{equation}

\smallskip
\noindent
\textbf{Step~2 (\boldmath$C^1$ approximation with compact support).}
For any $f\in W^{1,2}(\R,H)$ there is a sequence
$F_k\in C^1_0(\R,H)$ which converges to $f$ in $W^{1,2}$.

\smallskip
\noindent
To see this pick $\phi\in C^\infty_{\rm c}(\R,[0,1])$ supported in
$[-2,2]$ and with $\phi\equiv 1$ on $[-1,1]$.
For $k\in\N$ set $\phi_k(t):=\phi(\tfrac{t}{k})$ and $g_k:=\phi_k f$.
We claim that $g_k\to f$ in $W^{1,2}(\R,H)$. Indeed since
$\phi_k\equiv1$ on $\{\abs{t}\le k\}$
the integral domain reduces to
\begin{equation*}
\begin{split}
   \norm{f-g_k}_2^2
   &=\int_{\abs{t}> k}
   \underbrace{(1-\phi_k(t))}_{\le 1}(1-\phi_k(t))\cdot \abs{f(t)}^2 dt
\\
   &\le \int_{\abs{t}> k}\abs{f(t)}^2 dt
   \to 0 \text{, as $k\to 0$,}
\end{split}
\end{equation*}
and since $\dot g_k(t)=\phi_k(t)\dot f(t)+k\dot\phi(k t) f(t)$ we get
\begin{equation*}
\begin{split}
   \norm{\dot f-\dot g_k}_2^2
   &=\int_\R
   \abs{(1-\phi(\tfrac{t}{k}))\dot f(t)-\tfrac{1}{k}\dot\phi(\tfrac{t}{k}) f(t)}^2 dt
\\
   &\le2\int_{\abs{t}> k}
   \abs{\dot f(t)}^2 dt
   +\frac{2}{k^2} \norm{\dot\phi}_\infty\norm{f}_2^2
   \to 0 \text{, as $k\to \infty$.}
\end{split}
\end{equation*}
The sequence of compactly supported smooth functions
$
    F_k
   :=\phi_k (f*\rho_k)
$
approximates $f$ in $W^{1,2}$, indeed
\begin{equation}\label{eq:hgfd35g}
\begin{split}
   \norm{f-F_k}_{1,2}
   &=\norm{f-\phi_k f +\phi_k(f-f*\rho_k)}_{1,2}
\\
   &\le
   \underbrace{\norm{f-g_k}_{1,2}}_{\to 0  \text{ shown above}}
   +\underbrace{\norm{f-f*\rho_k}_{1,2}}_{\to 0  \text{ by~(\ref{eq:h68776fgh34})}} .
\end{split}
\end{equation}
In step~1 we added zero and step~2 uses the triangle inequality and
that $\phi_k\le 1$.
\\
This proves Step~2 and concludes the proof of Theorem~\ref{thm:Sob-dense}.
\end{proof}

%%%%%%%%%%%%%%%%%%%%%%%%%%%%%%%%%%%
%%%%%%% Section  %%%%%%%%%%%%%%%%%%%%%
%%%%%%%%%%%%%%%%%%%%%%%%%%%%%%%%%%%
\subsection{Sobolev embedding}\label{sec:Sobolev}

\begin{theorem}\label{thm:Sob-estimate}
Any element $v\in W^{1,2}(\R,H)$ satisfies the estimate
\begin{equation}\label{eq:Sob-estimate}
   \norm{v}_\infty\le \norm{v}_{1,2} .
\end{equation}
\end{theorem}

\begin{proof}
By Theorem~\ref{thm:Sob-dense} we can assume without loss of
generality that $v\in C^1_{\rm c}(\R,H)$.
By the fundamental theorem of calculus in combination with compact
support at any time $t\in\R$ we estimate
\begin{equation*}
\begin{split}
   \abs{v(t)}^2
   &=\int_{-\infty}^t \underbrace{\tfrac{d}{d\sigma}\abs{v(\sigma)}^2}
      _{=2\INNER{v^\prime(\sigma)}{v(\sigma)}} \, d\sigma
\\
   &\stackrel{2}{\le}\int_{-\infty}^t \left(\abs{v^\prime(\sigma)}^2
   +\abs{v(\sigma)}^2\right) \, d\sigma
\\
   &\stackrel{3}{\le}\int_{-\infty}^\infty \left(\abs{v^\prime(\sigma)}^2
   +\abs{v(\sigma)}^2\right) \, d\sigma
\\
   &=\norm{v}_{W^{1,2}(\R,H)}^2 .
\end{split}
\end{equation*}
In step 2 we used Cauchy-Schwarz and then Young's inequality $ab\le
a^2/2+b^2/2$ whenever $a,b\ge 0$.
Since $t\in\R$ was arbitrary, estimate~(\ref{thm:Sob-estimate}) follows.
\end{proof}

\begin{remark}\label{rem:Sob-estimate}
The proof shows that if $\tau\in\R$
and $v\in W^{1,2}((-\infty,\tau),H)$, then
instead of estimate~(\ref{thm:Sob-estimate}) we have
$$
   \norm{v}_{L^\infty((-\infty,\tau),H)}
   \le \norm{v}_{W^{1,2}((-\infty,\tau),H)} .
$$
To see this in step 3 of the estimate just replace
$\int_{-\infty}^\infty$ by $\int_{-\infty}^\tau$
Similarly, if $v\in W^{1,2}((\tau,\infty),H)$, then we obtain
$$
   \norm{v}_{L^\infty((\tau,\infty),H)}
   \le \norm{v}_{W^{1,2}((\tau,\infty),H)} .
$$
\end{remark}

%\newpage %.\newpage %.\newpage
%%%%%%%%%%%%%%%%%%%%%%%%%%%%%%%%%%%
%%%%%%%%%%%%%%%%%%%%%%%%%%%%%%%%%%%
%%%%%%% Section  %%%%%%%%%%%%%%%%%%%%%%
%%%%%%%%%%%%%%%%%%%%%%%%%%%%%%%%%%%
%%%%%%%%%%%%%%%%%%%%%%%%%%%%%%%%%%%
\section{Implicit Function Theorem}\label{sec:IFT}

%%%%%%%%%%%%%%%%%%%%%%%%%%%%%%%%%%%
%%%%%%%%%%%%%%%%%%%%%%%%%%%%%%%%%%%
%%%%%%% Subsection  %%%%%%%%%%%%%%%%%%%
%%%%%%%%%%%%%%%%%%%%%%%%%%%%%%%%%%%
%%%%%%%%%%%%%%%%%%%%%%%%%%%%%%%%%%%
\subsection{Quantitative}
\label{sec:quantitative-IFT}

We denote by $B_r(x;X)$ the open ball of radius $r$ centered at $x$ in
a Banach space $X$. We often abbreviate $B_r(x):=B_r(x;X)$ and
$B_r:=B_r(0;X)$.

\begin{lemma}[{\hspace{-.001cm}\cite[Le.\,A.3.2]{mcduff:2004a}}]
\label{le:A.3.2}
Let $\gamma<1$ and $R$ be positive real numbers.
Let $X$ be a Banach space, $x_0\in X$, and $\varphi\colon B_R(x_0)\to X$
be a continuously differentiable map such that
\[
   \norm{\Id-d\varphi(x)}\le\gamma
\]
for every $x\in B_R(x_0)$.
Then the following holds.
The map $\varphi$ is injective and $\varphi$ maps $B_R(x_0)$
into an open set in $X$ such that
\begin{equation}\label{eq:psi-incl}
   B_{R(1-\gamma)}(\varphi(x_0))
   \subset\varphi(B_R(x_0))
   \subset B_{R(1+\gamma)}(x_0) .
\end{equation}
The inverse $\varphi^{-1}\colon\varphi(B_R(x_0))\to B_R(x_0)$
is continuously differentiable and
\begin{equation}\label{eq:psi-inverse}
   d(\varphi^{-1})|_y
   =(d\varphi|_{\varphi^{-1}(y)})^{-1} .
\end{equation}
\end{lemma}

\begin{corollary}[Higher differentiability $C^k$]\label{cor:C2}
Under the assumption of Lemma~\ref{le:A.3.2}
assume in addition that $\varphi$ is $C^k$.
Then the inverse is $C^k$ as well.
\end{corollary}

\begin{proof}
This follows inductively by the chain rule from~(\ref{eq:psi-inverse}).
\end{proof}

\begin{lemma}[Family version]\label{le:A.3.2-family}
Let $\gamma<1$ and $R$ be positive real numbers.
Let $X$ be a Banach space and $x_0\in X$.
Assume that there exist\ $\eps>0$ and a $C^k$ map
$\varphi\colon(-\eps,\eps)\times B_R(x_0)\to X$
such that for every $s\in(-\eps,\eps)$ the map
$\varphi_s:=\varphi(s,\cdot)\colon B_R(x_0)\to X$
satisfies the estimate
\[
   \norm{\Id-d\varphi_s(x)}\le\gamma
\]
at every point $x\in B_R(x_0)$.
Then the following holds.
The $C^k$ map defined by
$$
   \Phi\colon (-\eps,\eps)\times B_R(x_0)\to (-\eps,\eps)\times X
   ,\quad
   (s,x)\mapsto (s,\varphi_s(x))
$$
has an inverse and $\Phi^{-1}$ is of class $C^k$ as well.
\end{lemma}

\begin{proof}
For fixed $s\in(-\eps,\eps)$ the map $\varphi_s:=\varphi(s,\cdot)$
is invertible by Lemma~\ref{le:A.3.2}.
The inverse of $\Phi$ is then given by
$$
   \Phi^{-1}\colon \Im(\Phi)\to (-\eps,\eps)\times B_R(x_0)
   ,\quad
   (s,y)\mapsto (s,{\varphi_s}^{-1}(y))
$$
and its derivative by
$$
   d\Phi^{-1}|_{(s,y)}
   \begin{bmatrix}\hat s\\\hat y\end{bmatrix}
   =
   \begin{bmatrix}1&0\\
     -\bigl(d\varphi_s|_{\varphi_s^{-1}(y)}\bigr)^{-1}\dot\varphi_s|_{\varphi_s^{-1}(y)}
      &d(\varphi_s)^{-1}|_y\end{bmatrix}
   \begin{bmatrix}\hat s\\\hat y\end{bmatrix}
$$
where $\dot\varphi:=\p_1\varphi(\cdot,\cdot)$ is the $s$-derivative.
If $\varphi$, hence $\Phi$, is $C^k$, then using inductively the chain
rule on the above formula shows that $\Phi^{-1}$ is $C^k$ as well.
\end{proof}

\begin{remark}[Corollary to Lemma~\ref{le:A.3.2}]
\label{rem:A.3.2.}
Let $\gamma<1$ and $R$ be positive real numbers.
Let $X$ be a Banach space, $x_0\in X$, and
$
   \varphi\colon B_R(x_0)\to X
$
be a $C^k$ map with
$$
   \norm{d\varphi|_x-\Id}\le \gamma
$$
for every $x\in B_R(x_0)$.
In particular, by Corollary~\ref{cor:C2},
% Lemma~\ref{le:A.3.2}, 
the map $\varphi$ is a $C^k$ diffeomorphism onto its image.
For $\beta\in[0,1]$ we define a map
\begin{equation}\label{eq:A.3.2-S}
   \Ss^\varphi_{\beta,x_0}\colon X\supset B_R \to X
   ,\quad
   \eta\mapsto (1-\beta)(\varphi(x_0)+\eta)+\beta\varphi(x_0+\eta) .
\end{equation}
The derivative at $\eta\in B_R$ is given by
$$
   d \Ss^\varphi_{\beta,x_0}|_\eta \colon X\to X,\quad
   \hat\eta\mapsto (1-\beta)\hat\eta+\beta d\varphi|_{x_0+\eta} \hat\eta.
$$
Hence
\begin{equation}\label{eq:hfgjk7799}
   d \Ss^\varphi_{\beta,x_0}|_\eta-\Id
   =\beta\left(d\varphi|_{x_0+\eta} -\Id\right)\colon X\to X.
\end{equation}
Since $\beta\le 1$ we obtain for the operator norm
$$
   \norm{d \Ss^\varphi_{\beta,x_0}|_\eta-\Id}
   \le \gamma
$$
whenever $\eta\in B_R$.
Therefore, by Lemma~\ref{le:A.3.2}, all maps
$$
   \Ss^\varphi_{\beta,x_0}\colon X\supset B_R \to X
$$
are $C^k$ diffeomorphisms onto the image.
\end{remark}

%%%%%%%%%%%%%%%%%%%%%%%%%%%%%%%%%%%
%%%%%%%%%%%%%%%%%%%%%%%%%%%%%%%%%%%
%%%%%%% Subsection  %%%%%%%%%%%%%%%%%%%
%%%%%%%%%%%%%%%%%%%%%%%%%%%%%%%%%%%
%%%%%%%%%%%%%%%%%%%%%%%%%%%%%%%%%%%
\subsection{Qualitative -- family inversion}
\label{sec:qualitative-IFT}  %\label{sec:inversion}

\begin{proposition}\label{prop:inversion}
Let $H$ be a Hilbert space.
Consider a family of maps $\Ff\colon\R\times H\to H$
of class $C^2$ such that for each $s\in\R$ the map
$$
   \Ff_s:=\Ff(s,\cdot)\colon H\to H
$$
is a $C^2$-diffeomorphism.
Then the map defined by
\begin{equation}\label{eq:Gg}
   \Gg\colon\R\times H\to H
   ,\quad
   (s,y)\mapsto {\Ff_s}^{-1}(y)
\end{equation}
is of class $C^2$ as well.
\end{proposition}

\begin{proof}
The idea is to apply the Implicit Function Theorem,
see e.g. the book by Lang~\cite[Thm.\,5.9]{lang:2001a},
to the following map
$$
   f\colon\R \times H\times H\to H
   ,\quad
   (s,x,y)\mapsto \Ff(s,y)-x .
$$
This map is of class $C^2$ by our assumptions.
Abbreviate $z=(s,x)$, then in the notation
of~\cite[Thm.\,5.9]{lang:2001a} it holds that
$$
   D_2f(z,y)=d\Ff_s(y)\colon H\to H
$$
and since according to our assumptions
$\Ff_s$ is a $C^2$-diffeomorphism
it follows, that $D_2f(z,y)$ is an isomorphism.

According to the Implicit Function
Theorem~\cite[Thm.\,5.9]{lang:2001a},
there then exists a $C^2$ map
$
   g: \R \times H \to H
$
such that
\begin{equation}\label{eq:Em2-(1)}
   \qquad \qquad \qquad \qquad
   f(z,g(z))=0  \qquad \qquad\qquad \qquad\qquad            (1)
\end{equation}
for every $z \in \R \times H$.
Substituting now $z=(s,x)$ in the definition of $f$ we obtain
\begin{equation}\label{eq:Em2-(2)}
   f(z,g(z))=f(s,x,g(s,x))=F(s,g(s,x))-x .
\end{equation}
By~(\ref{eq:Em2-(1)}) and~(\ref{eq:Em2-(2)}) it follows
$
   x=\Ff(s,g(s,x))=\Ff_s(g(s,x))
$ and therefore
$$
   g(s,x)={\Ff_s}^{-1}(x) .
$$
Since $g$ is of class $C^2$ this proves Proposition~\ref{prop:inversion}.
\end{proof}

\begin{lemma}\label{le:dGg}
The first two derivatives of $\Gg$ in~(\ref{eq:Gg})
are~(\ref{eq:dGg}) and~(\ref{eq:d2Gg}).
\end{lemma}

\begin{proof}
The proof has two steps.

\smallskip
\noindent
\textbf{Step~1.}
The first derivative of $\Gg$ for
$s,\hat s\in\R$ and $y,\hat y\in\R\times H$ is given by
\begin{equation}\label{eq:dGg}
\boxed{
\begin{aligned}
   d\Gg|_{(s,y)}(\hat s,\hat y)
   &=- (d\Ff_s|_{{\Ff_s}^{-1}(y)})^{-1} \dot{\Ff}_s|_{{\Ff_s}^{-1}(y)} \hat s
   +(d\Ff_s|_{{\Ff_s}^{-1}(y)})^{-1}\hat y
\\
   &=
   \begin{bmatrix}
      - (d\Ff_s|_{{\Ff_s}^{-1}(y)})^{-1} \dot{\Ff}_s|_{{\Ff_s}^{-1}(y)}&(d\Ff_s|_{{\Ff_s}^{-1}(y)})^{-1}
   \end{bmatrix}
   \begin{bmatrix}
      \hat s\\\hat y
   \end{bmatrix}
\end{aligned}
}
\end{equation}
in any direction $(\hat s,\hat y)\in\R\times H$

\smallskip
\noindent
Observe that
\begin{equation}\label{eq:Gg_s}
   \Gg_s={\Ff_s}^{-1}\colon H\to H
   ,\qquad
   \Gg_s\circ\Ff_s
   ={\Ff_s}^{-1}\circ\Ff_s=\id_H .
\end{equation}
Hence we get
\begin{equation}\label{eq:dGg-y}
\begin{split}
   d\Gg|_{(s,y)}(0,\hat y)
   =d(\Ff_s)^{-1}|_y \hat y
   =(d\Ff_s|_{{\Ff_s}^{-1}(y)})^{-1}\hat y
\end{split}
\end{equation}
for every $\hat y\in H$.
Abbreviating
$$
   \dot{\Ff}_s(\cdot):=\p_s \Ff(s,\cdot) ,
$$
then at
$(s,x)\in\R\times H$ we get
\begin{equation*}
\begin{split}
   d\Ff|_{(s,x)}(\hat s,0)
   &=\dot{\Ff}_s(x) \, \hat s
\end{split}
\end{equation*}
for every $\hat s\in \R$.
We take the $s$-derivative of $\Ff_s\circ{\Ff_s}^{-1}(y)=y$
to obtain
\begin{equation}\label{eq:Ff_s-inv-ident}
   \dot{\Ff}_s|_{{\Ff_s}^{-1}(y)}
   +d\Ff_s|_{{\Ff_s}^{-1}(y)} \dot{({\Ff_s}^{-1})}|_y
   =0
\end{equation}
for every $y\in H$ and thus
\begin{equation}\label{eq:Ff_s-inv}
   \dot{({\Ff_s}^{-1})}(y)
   =-(d\Ff_s|_{{\Ff_s}^{-1}(y)})^{-1} \dot{\Ff}_s|_{{\Ff_s}^{-1}(y)} .
\end{equation}
Therefore we obtain
\begin{equation*}
\begin{split}
   d\Gg|_{(s,y)}(\hat s,0)
   &=\dot{\Gg}_s|_y \hat s
\\
   &\stackrel{2}{=}\dot{({\Ff_s}^{-1})}|_y \hat s
\\
   &\stackrel{3}{=} - (d\Ff_s|_{{\Ff_s}^{-1}(y)})^{-1} \dot{\Ff}_s|_{{\Ff_s}^{-1}(y)} \hat s
\end{split}
\end{equation*}
where step 2 is by~(\ref{eq:Gg_s})
and step 3 by~(\ref{eq:Ff_s-inv}).
Together with~(\ref{eq:dGg-y}) 
this proves~(\ref{eq:dGg}).

\smallskip
\noindent
\textbf{Step~2.}
We calculate the second derivative of $\Gg$ at a
point $(s,y)\in\R\times H$.

\smallskip
\noindent
The second derivative of $\Gg$ at a
point $(s,y)\in\R\times H$ is of the form
\begin{equation}\label{eq:d2Gg}
\boxed{
   d^2\Gg|_{(s,y)}(\hat s_1,\hat y_1; \hat s_2,\hat y_2)
   =
   \left(
   \begin{bmatrix}
      A&B\\B&D
   \end{bmatrix}
   \begin{bmatrix}
      \hat s_1\\\hat y_1
   \end{bmatrix}
   \right)^t
   \begin{bmatrix}
      \hat s_2\\\hat y_2
   \end{bmatrix}
}
\end{equation}
whenever $y_1,\hat y_2\in H$ and $\hat s_1,\hat s_2 \in\R$.
The terms $A$, $B$, and $D$ are as follows.

\smallskip
\noindent
\textsc{Term~$A$.}
Take the $s$-derivative of~(\ref{eq:Ff_s-inv-ident}) to obtain
\begin{equation*}%\label{eq:}
\begin{split}
   0
   &=\underline{\p_s\left(\dot{\Ff}_s|_{{\Ff_s}^{-1}(y)}\right)}
   +\p_s\left( d\Ff_s|_{{\Ff_s}^{-1}(y)} \dot{({\Ff_s}^{-1})}|_y\right)
\\
   &=\underline{\ddot{\Ff}_s|_{{\Ff_s}^{-1}(y)}+d\dot{\Ff}_s|_{{\Ff_s}^{-1}(y)}
   \dot{({\Ff_s}^{-1})}|_y}
   +d\dot{\Ff}_s|_{{\Ff_s}^{-1}(y)} \dot{({\Ff_s}^{-1})}|_y
   \\
   &\quad
   +d^2\Ff_s|_{{\Ff_s}^{-1}(y)}\left(\dot{({\Ff_s}^{-1})}|_y,\dot{({\Ff_s}^{-1})}|_y\right)
   +d\Ff_s|_{{\Ff_s}^{-1}(y)} \ddot{({\Ff_s}^{-1})}|_y .
\end{split}
\end{equation*}
We use this identity in step 3 of the following calculation
\begin{equation*}
\begin{split}
   &d^2\Gg|_{(s,y)}(\hat s_1,0;\hat s_2,0)
\\
   &=
   \ddot{\Gg}_s|_y \hat s_2\hat s_2
\\
   &\stackrel{2}{=}
   \ddot{({\Ff_s}^{-1})}|_y \hat s_2\hat s_2
\\
   &\stackrel{3}{=}
   - \hat s_2\hat s_2 (d\Ff_s|_{{\Ff_s}^{-1}(y)})^{-1}
   \biggl(
   \ddot{\Ff}_s|_{{\Ff_s}^{-1}(y)} 
   +2 d\dot{\Ff}_s|_{{\Ff_s}^{-1}(y)} \underline{\dot{({\Ff_s}^{-1})}|_y}
   \\
   &\quad
   +d^2\Ff_s|_{{\Ff_s}^{-1}(y)}\left(\underline{\dot{({\Ff_s}^{-1})}|_y},
   \underline{\dot{({\Ff_s}^{-1})}|_y}\right)
   \biggr)
\\
   &\stackrel{4}{=}
   -\hat s_2\hat s_2 (d\Ff_s|_{{\Ff_s}^{-1}(y)})^{-1}
   \biggl(
   \ddot{\Ff}_s|_{{\Ff_s}^{-1}(y)}
%   \\
%   &\quad
   -2d\dot{\Ff}_s|_{{\Ff_s}^{-1}(y)} 
   \underline{(d\Ff_s|_{{\Ff_s}^{-1}(y)})^{-1} \dot{\Ff}_s|_{{\Ff_s}^{-1}(y)}}
   \\
   &\quad
   +d^2\Ff_s|_{{\Ff_s}^{-1}(y)}
   \left(
   \underline{(d\Ff_s|_{{\Ff_s}^{-1}(y)})^{-1} \dot{\Ff}_s|_{{\Ff_s}^{-1}(y)}}
   ,
   \underline{(d\Ff_s|_{{\Ff_s}^{-1}(y)})^{-1} \dot{\Ff}_s|_{{\Ff_s}^{-1}(y)}}
   \right)
   \biggr)
\end{split}
\end{equation*}
where step 2 is by~(\ref{eq:Gg_s})
and step 4 by~(\ref{eq:Ff_s-inv}).

\smallskip
\noindent
\textsc{Term~$D$.}
Let $s\in\R$ and $y,\hat y_1,\hat y_2\in H$. Set $x:={\Ff_s}^{-1}(y)$.
Then we have
\begin{equation*}
\begin{split}
   d^2\Gg|_{(s,y)}(0,\hat y_1;0,\hat y_2)
   &\stackrel{1}{=}
   d^2(\Ff_s)^{-1}|_y (\hat y_1,\hat y_2)
\\
   &\stackrel{2}{=}
   -(d\Ff_s|_x)^{-1}
   d^2\Ff_s|_x
   \left(
   (d\Ff_s|_x)^{-1}\hat y_1
   ,
   (d\Ff_s|_x)^{-1}\hat y_2
   \right)
\end{split}
\end{equation*}
Step~1 is by~(\ref{eq:dGg-y}).
Step~2 is analogous to~(\ref{eq:sec-deriv}).

\smallskip
\noindent
\textsc{Term~$B$.}
Let $s,\hat s\in\R$ and $y,\hat y\in H$. Set $x:={\Ff_s}^{-1}(y)$.
Then we have
\begin{equation*}
\begin{split}
   d^2\Gg|_{(s,y)}(0,\hat y;\hat s,0)
   &\stackrel{1}{=}
   \p_s \left(
   (d\Ff_s|_x)^{-1}\hat y
   \right) \hat s
\\
   &\stackrel{2}{=}
   -(d\Ff_s|_x)^{-1} d\dot{\Ff}_s|_x (d\Ff_s|_x)^{-1}\hat y \hat s
\end{split}
\end{equation*}
where step~1 is by~(\ref {eq:dGg-y})
and step~2 by the following consideration.
For $x,\xi\in H$ take the $s$-derivative of
$\xi=(d\Ff_s|_x)^{-1} d\Ff_s|_x \xi$ to obtain
\begin{equation*}
\begin{split}
   0
   &=\left(\p_s(d\Ff_s|_x)^{-1}\right) d\Ff_s|_x\xi
   +(d\Ff_s|_x)^{-1} \p_s (d\Ff_s|_x)
\\
   &=\left(\p_s(d\Ff_s|_x)^{-1}\right) d\Ff_s|_x+(d\Ff_s|_x)^{-1} d\dot{\Ff}_s|_x
\end{split}
\end{equation*}
where we used that derivatives commute by the Theorem of Schwarz.
Hence
$$
   \p_s(d\Ff_s|_x)^{-1}
   =-(d\Ff_s|_x)^{-1} d\dot{\Ff}_s|_x (d\Ff_s|_x)^{-1}
$$
for all $x,\xi \in H$.
This concludes the proof of Step~2
and Lemma~\ref{le:dGg}.
\end{proof}

\newpage %.\newpage
%%%%%%%%%%%%%%%%%%%%%%%%%%%%%%%%%%%
%%%%%%%%%%%%%%%%%%%%%%%%%%%%%%%%%%%
%%%%%%% Section  %%%%%%%%%%%%%%%%%%%%%%
%%%%%%%%%%%%%%%%%%%%%%%%%%%%%%%%%%%
%%%%%%%%%%%%%%%%%%%%%%%%%%%%%%%%%%%
\section{Hilbert manifold structure for the path space
of finite dimensional manifolds with the help of the exponential map}
\label{sec:Hmf-of-paths}

Finite dimensional manifolds are automatically tame
and therefore Theorem~\ref{thm:A} in particular implies that the space of
$W^{1,2}$ paths on a finite dimensional manifold is a $C^1$ Hilbert manifold.
In this appendix we show how this can as well be
deduced more traditionally with the help of the exponential map.
\\
If one uses the exponential map on a $C^2$ manifold,
the finite dimensional version of the parametrized version of
Theorem~\ref{thm:B} is not quite sufficient since in general one
will not have two continuous derivatives in both the parameter $s$ and
the space variable.
We therefore establish as a technical tool
Theorem~\ref{thm:mf-of-paths-finite-model}
which allows us to deal with this complication.

\smallskip
Let $M$ be a $C^2$ manifold of finite dimension $n$.
Pick two points $x_-,x_+\in M$.
Manifolds of maps $N\to M$ between manifolds
have been constructed with the use of an exponential map
on the target side.
While El{\u{\i}}asson~\cite{Eliasson:1967a} assumes compactness of the
domain $N$ but allows infinite dimension of the target,
Schwarz~\cite{schwarz:1993a} deals with maps $\R\to M$ and uses smooth
maps $x\colon\R\to M$ which reach $x_\mp$ only asymptotically, i.e. in
infinite time $\mp \infty$, as the fundamental paths
around which coordinate charts are being built.

As in the construction of Theorem~\ref{thm:A} in
\S\ref{sec:coordinate-charts}, but in contrast to
Schwarz~\cite{schwarz:1993a} 
we build our charts around \emph{basic paths}, i.e. paths which reach 
$x_\mp$ already in finite time $\mp T$.
In Subsection~\ref{sec:exp-map}
we recall the elegant construction of coordinate charts of path space using the
exponential map associated to the choice of a Riemannian metric
on $M$. Here we already use finite time basic paths.

\begin{definition}\label{def:basic-path}
Consider a $C^2$ path $x\colon\R\to M$ with the property that
$$
   x(s)=
   \begin{cases}
      x_-&\text{, $s\le -T$,}
      \\
      x_+&\text{, $s\ge T$,}
   \end{cases}
$$
for some $x_-,x_+\in M$ and some $T>0$.
Such paths are called \textbf{basic paths}.
\end{definition}

We call two Hilbert spaces \textbf{equivalent}
if they coincide as vector spaces and their inner products
are equivalent. When we refer to an \textbf{equivalence class} of Hilbert spaces
we mean equivalence with respect to this equivalence relation.
For a basic path $x$ we consider the
\textbf{equivalence class of Hilbert spaces}
\begin{equation}\label{eq:H^1_x}
   H^1_x
   :=W^{1,2}(\R,x^*TM) .
\end{equation}
In fact, to choose an inner product in $H^1_x$ we have to choose
a trivialization $\Tt^x$ of the bundle $x^*TM\to\R$.
As we will explain below, Proposition~\ref{prop:change-of-triv},
choosing different trivializations gives rise to equivalent inner
products on $H^1_x$.

%%%%%%%%%%%%%%%%%%%%%%%%%%%%%%%%%%%
%%%%%%% Subsection  %%%%%%%%%%%%%%%%%%%
%%%%%%%%%%%%%%%%%%%%%%%%%%%%%%%%%%%
\subsection{Technical Tool}\label{sec:tec-tool-n}

The following
theorem will be the base to prove Theorem~\ref{thm:mf-of-paths-finite}
(transition maps are $C^1$ diffeomorphisms).

\begin{theorem}\label{thm:mf-of-paths-finite-model}
Let $\varphi\colon\R\times \R^n\to\R^n$ be a $C^1$ map
with the following properties.
\begin{enumerate}\setlength\itemsep{0ex} 
  \item[\rm (i)]
  For each $s\in\R$ the map $\varphi_s:=\varphi(s,\cdot)\colon\R^n\to\R^n$
  is $C^2$.
\item[\rm (ii)]
  The map $\dot\varphi_s\colon\R^n\to \R^n$ is $C^1$
  where  $\dot \varphi:=\p_1\varphi(\cdot,\cdot)$ is the $s$-derivative.
\item[\rm (iii)]
  The map $\R\times\R^n\to\Ll(\R^n,\R^n;\R^n)$, $(s,x)\mapsto d^2\varphi_s|_x$,
  is continuous.
\item[\rm (iv)]
  The map $\R\times\R^n\mapsto \Ll(\R^n)$, $(s,x)\mapsto d\dot\varphi_s|_x$,
  is continuous.
\item[\rm (v)]
  There exist $T>0$ such that $\varphi_s=\varphi_\pm$
  whenever $\pm s>T$ and
  $\varphi_\pm\colon\R^n\to\R^n$ is of class $C^2$ and maps $0$ to $0$.
\end{enumerate}
Then the map
\begin{equation}\label{eq:Psi-map}
\begin{split}
   \Phi\colon W^{1,2}(\R,\R^n)&\to W^{1,2}(\R,\R^n)
   \\
   \xi&\mapsto [s\mapsto (\Phi(\xi))(s):=\varphi(s,\xi(s))]
\end{split}
\end{equation}
is well defined and $C^1$.
\end{theorem}

\begin{corollary}\label{cor:psi}
Let $\varphi\colon\R\times \R^n\to\R^n$ be a $C^1$ map
linear in the second variable, i.e.
$\varphi_s:=\varphi(s,\cdot)\in\Ll(\R^n)$ for any $s\in\R$,
and such that there exists $T>0$ with $\varphi_s=\id_{\R^n}$ whenever
$\abs{s}>T$. Then $\Phi$ in~(\ref{eq:Psi-map}) is a bounded linear map.
\end{corollary}

\begin{proof}
We check that $\varphi$ satisfies the conditions of
Theorem~\ref{thm:mf-of-paths-finite-model}.
Since $\varphi_s$ and $\dot\varphi_s$ are linear maps
they are in particular arbitrarily often differentiable.
In particular (i) and (ii) hold.
Too see that condition (iii) holds consider the map
$$
   \R\times\R^n\to\Ll(\R^n,\R^n;\R^n)
   ,\quad
   (s,x)\mapsto d^2\varphi_s|_x .
$$
Observe that $(v,w)\mapsto d^2\varphi_s|_x(v,w)=\varphi_s w$ is independent
of $v$ and linear in~$w$.
Hence (iii) holds.
Observe that $d\dot\varphi_s|_xw=\dot\varphi_s w$
and therefore condition (iv) holds as well.
Since $\varphi_s=\id_{\R^n}$
whenever $s>T$
condition (v) is also satisfied.
Thus, by Theorem~\ref{thm:mf-of-paths-finite-model},
the map $\Phi$ is a well defined bounded linear map
on $W^{1,2}$.
\end{proof}

\begin{proof}[Proof of Theorem~\ref{thm:mf-of-paths-finite-model}]
The proof is in three steps.
Hypothesis (i) and (ii) serve to formulate (iii) and (iv).
We abbreviate $\xi_s:=\xi(s)$ and $\dot\xi_s:=\frac{d}{ds}\xi_s$.

\medskip\noindent
\textbf{Step~1.} $\Phi$ takes values in $W^{1,2}$ (well defined).

\begin{proof}
Pick $\xi \in W^{1,2}(\R,\R^n)$.
\\
a) We show that $\Phi(\xi)\in L^2(\R,\R^n)$.
To see this note that since $\xi\in W^{1,2}$
it is continuous and has to decay asymptotically.
Therefore there exists a constant $c$ such that
$\abs{\xi_s}\le c$ for every $s\in\R$.
By continuity of $\varphi$, there is a constant
$\kappa=\kappa(c)>0$ such that
$
   \norm{\varphi|_{[-T,T]\times B_c}}_\infty\le \kappa
$
where $B_c\subset\R^n$ is the radius $c$ ball about the origin.
Hence $\abs{(\Phi(\xi))(s)}\le\kappa$ for every $s\in[-T,T]$.

Let $\varphi_\pm$ be the maps provided by (v).
We next show that there exists a constant $C(c)$ such that
\begin{equation}\label{eq:c_kappa}
   \norm{\varphi_\pm (v)}\le C \abs{v}
\end{equation}
whenever $\abs{v}\le c$.
Indeed, since $d\varphi_\pm$ is $C^1$, there is a constant $C$ such that
\begin{equation*}%\label{eq:d-psi}
   \norm{d\varphi_\pm(x)}\le C
\end{equation*}
whenever $\abs{x}\le c$.
Now use $\varphi_\pm(0)=0$ to estimate
\begin{equation*}
\begin{split}
   \abs{\varphi_\pm(v)}
   &=\Bigl|
   \int_0^1 \tfrac{d}{dt}\varphi_\pm(tv)\, dt
   \Bigr|
\\
   &\le\int_0^1\norm{d\varphi_\pm(tv)}\cdot\abs{v}\, dt
\\
   &\le
   C\abs{v} .
\end{split}
\end{equation*}
This proves~(\ref{eq:c_kappa}).
Using~(\ref{eq:c_kappa}) in equality 3, we can estimate
\begin{equation*}
\begin{split}
   &\norm{\varphi(\cdot,\xi(\cdot))}_{L^2(\R,\R^n)}^2
\\
   &=\norm{\Phi(\xi)}_{L^2(\R,\R^n)}^2
\\
   &=\norm{\Phi(\xi)}_{L^2((-\infty,-T),\R^n)}^2
   +\norm{\Phi(\xi)}_{L^2([-T,T],\R^n)}^2
   +\norm{\Phi(\xi)}_{L^2((T,\infty),\R^n)}^2
\\
   &\stackrel{3}{=}
   C^2\norm{\xi}_{L^2((-\infty,-T),\R^n)}^2
   +\norm{\Phi(\xi)}_{L^2([-T,T],\R^n)}^2
   +C^2\norm{\xi}_{L^2((T,\infty),\R^n)}^2
\\
   &\le C^2\norm{\xi}_{W^{1,2}(\R,\R^n)}^2
   +2T\kappa^2<\infty.
\end{split}
\end{equation*}
b) We show that $\Phi(\xi)\in W^{1,2}(\R,\R^n)$.
So see this we calculate
$$
   \p_s(\Phi(\xi))(s)
   =\underbrace{\dot\varphi(s,\xi_s)}_{\text{$=0$, $\abs{s}> T$}}
   +\underbrace{d\varphi_s|_{\xi_s}}_{\text{$\stackrel{\text{(v)}}{=}d\varphi_\pm|_{\xi_s}$, $\abs{s}> T$}} \dot \xi_s .
$$
We show that both summands are in $L^2(\R,\R^n)$.
\\
Summand one: Let the constant $c$ be as in a).
Since $\varphi$ is $C^1$ it follows that $\dot\varphi$ is continuous.
Thus there is a constant
$\kappa_2=\kappa_2(c)>0$ such that it holds
$
   \norm{\dot \varphi|_{[-T,T]\times B_c}}_\infty\le \kappa_2 .
$
Hence $\abs{\dot \varphi(s,\xi(s))}\le\kappa_2$ for any $s\in[-T,T]$. We estimate
\begin{equation*}
\begin{split}
   \norm{\dot\varphi(\cdot,\xi(\cdot))}_{L^2(\R,\R^n)}
   &=\norm{\dot\varphi(\cdot,\xi(\cdot))}_{L^2([-T,T],\R^n)}\\
   &\le \kappa_2\sqrt{2T} <\infty
\end{split}
\end{equation*}
where we used in the identity that $\dot\varphi(s,\cdot)=0$ whenever $\abs{s}>T$.
\\
Summand two: Since $\xi$ is in $W^{1,2}$ it is in particular
continuous and decays asymptotically.
So the operator norm of $d\varphi_s|_{\xi_s}$
is uniformly bounded when $s$ lies in compact set $[-T,T]$.
Since outside of $[-T,T]$ we have $d\varphi_s|_{\xi_s}=d\varphi_\pm|_{\xi_s}$ by (v),
the operator norm is actually uniformly bounded on the whole of
$\R$, say by a constant $\kappa_3$. Hence
\begin{equation*}
\begin{split}
   \norm{d\varphi_\cdot|_{\xi(\cdot)} \dot\xi_s}_{L^2(\R,\R^n)}
   &\le \kappa_3 \norm{\dot\xi_s}_{L^2(\R,\R^n)}
   <\infty .
\end{split}
\end{equation*}
This proves b) and Step~1.
\end{proof}

%\smallskip
\noindent
\textbf{Step~2.} $\Phi$ is differentiable with derivative at $\xi\in
W^{1,2}(\R,\R^n)$ given by
$$
   \left(d\Phi|_\xi \eta\right)(s)
   =d\varphi_s|_{\xi_s}\eta_s
$$
for every $\eta\in W^{1,2}(\R,\R^n)$ and every $s\in\R$.

\begin{proof}
We need to show that our candidate for the derivative of $\Phi$ at $\xi$ in
direction $\eta$, namely the map
$s\mapsto d\varphi_s|_{\xi_s}\eta_s$, lies in $W^{1,2}$.
That this map lies in $L^2$ follows by the argument
in the proof of Step~1 b) for the second summand
by using continuity of $d\varphi_s$.
It remains to show that the $s$-derivative lies in $L^2$.
This $s$-derivative is given by
\begin{equation*}
\begin{split}
   \p_s \left(d\Phi|_\xi \eta\right)(s)
   &=d\dot\varphi_s|_{\xi_s}\eta_s
   +d^2\varphi_s|_{\xi_s}\left(\dot\xi_s,\eta_s\right)
   +d\varphi_s|_{\xi_s}\dot\eta_s .
\end{split}
\end{equation*}
That the third summand is in $L^2$ is the argument
in Step~1~b) for summand two.
That the first summand is in $L^2$ follows by the same argument,
but now using continuity of $d\dot\varphi_s$, which holds by hypothesis
(iv), and since $d\dot\varphi_s=0$ for $\abs{s}>T$.
It remains to discuss the second term.
Observe that $d^2\varphi_s(\cdot)=d^2\varphi_\pm(\cdot)$ whenever $\abs{s}>T$.
By compactness of $[-T,T]$ and by continuity of the map
$d^2\varphi$, by (iii),
as well as continuity and asymptotic decay in $s$ of the maps $\xi_s$ and $\eta_s$, the operator norm
$$
   \norm{v\mapsto d^2\varphi_s|_{\xi_s}\left(v,\eta_s\right)}_{\Ll(\R^n)}
    \le c_2 \norm{\eta}_\infty
   \le c_3 \norm{\eta}_{W^{1,2}(\R,\R^n)}
$$
is bounded uniformly by a constant $c_2=c_2(T)>0$ for every $s\in\R$.
Here $c_3$ is the constant of the embedding $W^{1,2}\INTO C^0$.
We obtain
$$
   \norm{d^2\varphi_\cdot|_{\xi}(\dot\xi,\eta)}_{L^2(\R,\R^n)}
   \le c_3 \norm{\eta}_{W^{1,2}(\R,\R^n)} \norm{\dot\xi}_{L^2(\R,\R^n)}
   <\infty .
$$
This proves that our derivative candidate
$d\varphi_\cdot|_{\xi(\cdot)}\eta(\cdot)$ is element of $\Ll(W^{1,2})$.

\smallskip
Let us abbreviate
$\norm{\cdot}_2:=\norm{\cdot}_{L^2(\R,\R^n)}$ and
$\norm{\cdot}_{1,2}:=\norm{\cdot}_{W^{1,2}(\R,\R^n)}$.
To see that our candidate actually is the derivative we show that the
limit
\begin{equation*}
\begin{split}
   &\lim_{h\to 0}\sup_{\norm{\eta}_{1,2}\le 1}
   \frac{\norm{\Phi(\xi+h\eta)-\Phi(\xi)-h\, d\varphi_\cdot|_\xi\eta}_{W^{1,2}}^2}{h^2}
\\
   &=\lim_{h\to 0}\sup_{\norm{\eta}_{1,2}\le 1}
   \frac{\int_\R\Abs{
   \varphi_s (\xi_s+h\eta_s)-\varphi_s (\xi_s)
   -h\, d\varphi_s|_{\xi_s} \eta_s}^2ds}{h^2}
   \\
   &\quad+\lim_{h\to 0}\sup_{\norm{\eta}_{1,2}\le 1}
   \frac{\int_\R\Abs{\tfrac{d}{ds}\left(\varphi_s (\xi_s+h\eta_s)
   -\varphi_s (\xi_s)-h\, d\varphi_s|_{\xi_s} \eta_s\right)}^2ds}{h^2}
\end{split}
\end{equation*}
exists and vanishes. To see that summand one vanishes we use the fundamental
theorem of calculus to write it in the form
\begin{equation}\label{eq:hbjhbhj576567}
\begin{split}
   &\lim_{h\to 0}\sup_{\norm{\eta}_{1,2}\le 1}\tfrac{1}{h^2}\int_\R
   \biggl|\int_0^1\biggl(\underbrace{\tfrac{d}{dt}\varphi_s(\xi_s+th\eta_s)}
      _{=d\varphi_s|_{\xi_s+th\eta_s} h\eta_s}
   -h\, d\varphi_s|_{\xi_s}\eta_s\biggr) dt\biggr|_{\R^n}^2 ds
\\
   &=\lim_{h\to 0}
   \sup_{\norm{\eta}_{1,2}\le 1}
   \underbrace{
   \int_\R
   \biggl|\int_0^1\biggl(d\varphi_s|_{\xi_s+th\eta_s}
   -d\varphi_s|_{\xi_s}\biggr)\eta_s dt\biggr|_{\R^n}^2 ds
   }_{=:F(h,\eta)} .
\end{split}
\end{equation}
There exists a constant $c_4$ such that
$\abs{\eta_s}\le c_4$ for all $\eta$ in the unit $W^{1,2}$ ball and $s\in\R$.
Since $d\varphi$ is continuous and $\varphi_s=\varphi_\pm$ is independent of
$s$ for $\abs{s}>T$, it is in particular uniformly continuous
on compact sets, and therefore for any $\eps>0$ there exists an $h_0>0$
such that for every $h\in[0,h_0]$ there is the estimate
$$
   \norm{d\varphi_s|_{\xi_s+th\eta_s}-d\varphi_s|_{\xi_s}}_{\Ll(\R^n)}
   \le\eps
$$
whenever $s\in\R$ and $\norm{\eta}_{1,2}\le 1$ and $t\in[0,1]$.
Hence
$
   F(h,\eta)
   \le\eps^2\norm{\eta}_{2}^2
   \le\eps^2
$
and therefore summand one $\lim_{h\to0} \sup_{\norm{\eta}_{1,2}\le 1} F(h,\eta)=0$
vanishes.

To see that summand two vanishes as well, we abbreviate and compute
\begin{equation*}
\begin{split}
   G(s):
   &=\tfrac{d}{ds}\left(\varphi_s (\xi_s+h\eta_s)
   -\varphi_s (\xi_s)-h\, d\varphi_s|_{\xi_s} \eta_s\right)
\\
   &\stackrel{1}{=}\dot\varphi_s|_{\xi_s+h\eta_s}
   +d\varphi_s|_{\xi_s+h\eta_s}(\dot\xi_s+h\dot\eta_s)
   \\
   &\quad-\dot\varphi_s|_{\xi_s}-d\varphi_s|_{\xi_s} \dot\xi_s
   \\
   &\quad
   -h\, d\dot\varphi_s|_{\xi_s}\eta_s
   -h\, d^2\varphi_s|_{\xi_s}(\dot\xi_s,\eta_s)
   -h\, d\varphi_s|_{\xi_s} \dot\eta_s
\\
   &\stackrel{2}{=}\underbrace{\dot\varphi_s|_{\xi_s+h\eta_s}-\dot\varphi_s|_{\xi_s}}
      _{=\int_0^1\tfrac{d}{dt}\dot\varphi_s|_{\xi_s+th\eta_s} dt}
   -h\, d\dot\varphi_s|_{\xi_s}\eta_s
   \\
   &\quad +d\varphi_s|_{\xi_s+h\eta_s}\dot\xi_s
   -d\varphi_s|_{\xi_s} \dot\xi_s
   -h\, d^2\varphi_s|_{\xi_s}(\dot\xi_s,\eta_s)
   \\
   &\quad +d\varphi_s|_{\xi_s+h\eta_s} h\dot\eta_s -h\, d\varphi_s|_{\xi_s} \dot\eta_s
\\
   &\stackrel{3}{=}h\int_0^1\left(d\dot\varphi_s|_{\xi_s+th\eta_s}-d\dot\varphi_s|_{\xi_s}\right)\eta_s\,
   dt
   \\
   &\quad+h\int_0^1\left(d^2\varphi_s|_{\xi_s+th\eta_s}-d^2\varphi_s|_{\xi_s}\right)
   (\dot\xi_s,\eta_s)\, dt
   \\
   &\quad+h \left(d\varphi_s|_{\xi_s+h\eta_s} -d\varphi_s|_{\xi_s}
   \right) \dot \eta_s
\end{split}
\end{equation*}
for every $s\in\R$.
Square this identity and integrate to obtain
\begin{equation*}
\begin{split}
   &\sup_{\norm{\eta}_{1,2}\le 1}\tfrac{1}{h^2}\int_\R\abs{G(s)}_{\R^n}^2 ds
   \\
   &\le\sup_{\norm{\eta}_{1,2}\le 1}
   \int_{-T}^T 3\biggl|\int_0^1\left(d\dot\varphi_s|_{\xi_s+th\eta_s}
   -d\dot\varphi_s|_{\xi_s}\right)\eta_s\, dt \biggr|^2_{\R^n} ds
   \\
   &\quad+\sup_{\norm{\eta}_{1,2}\le 1}
   \underbrace{\int_\R
     3\biggl|\int_0^1\left(d^2\varphi_s|_{\xi_s+th\eta_s}-d^2\varphi_s|_{\xi_s}\right)
   (\dot\xi_s,\eta_s)\, dt \biggr|^2_{\R^n} ds}_{=:F_2(h,\eta,\dot\xi)}
   \\
   &\quad+\sup_{\norm{\eta}_{1,2}\le 1}
   \int_\R 3\biggl|\left(d\varphi_s|_{\xi_s+h\eta_s} -d\varphi_s|_{\xi_s}
   \right) \dot \eta_s \biggr|^2_{\R^n} ds .
\end{split}
\end{equation*}
Term 1.
The limit, as $h\to 0$, of the first of the three terms
vanishes by exactly the same argument as in~(\ref{eq:hbjhbhj576567})
by using continuity of $d\dot\varphi$ by (iv).
\\
Term 2.
For the second term note that as in~(\ref{eq:hbjhbhj576567}) by continuity
(iii) and~(v) for every $\eps>0$ there is $h_0>0$ such that for any
$h\in[0,h_0]$ there is the estimate
$$
   \norm{d^2\varphi_s|_{\xi_s+th\eta_s}-d^2\varphi_s|_{\xi_s}}_{\Ll(\R^n)}
   \le\eps
$$
whenever $s\in\R$ and $\norm{\eta}_{1,2}\le 1$ and $t\in[0,1]$.
There exists a constant $\kappa>0$ such that
$\Norm{\eta}_\infty\le\kappa_T\Norm{\eta}_{W^{1,2}}$.
So, for $h\in[0,h_0]$, we have
\begin{equation}\label{eq:hghjgyh5688}
\begin{split}
   F_2(h,\eta,\dot\xi)
   &\le 3\eps^2\norm{\abs{\dot\xi}\cdot\abs{\eta}}_2^2
\\
   &\le 3\eps^2\norm{\eta}_\infty^2
   \norm{\dot\xi}_2^2
\\
   &\le 3\eps^2\kappa^2 \norm{\eta}_{1,2}^2
   \norm{\xi}_{1,2}^2 .
\end{split}
\end{equation}
Therefore $\sup_{\norm{\eta}_{1,2}\le 1} F_2(h,\eta,\dot\xi)
\le 3\eps^2\kappa^2\norm{\xi}_{1,2}^2$.
Consequently the limit vanishes
$\lim_{h\to0} \sup_{\norm{\eta}_{1,2}\le 1} F_2(h,\eta,\dot\xi)=0$.
\\
Term 3.
This follows as in~(\ref{eq:hbjhbhj576567}) by noticing that
$\norm{\eta}_{L^2(\R,\R^n)}\le \norm{\eta}_{W^{1,2}(\R,\R^n)}$.

This concludes the proof of Step 2 ($\Phi$ differentiable).
\end{proof}

%\smallskip
\noindent
\textbf{Step~3.} The differential $d\Phi\colon W^{1,2}\to\Ll(W^{1,2})$
is a continuous map.

\begin{proof}
For $\xi,\tilde\xi,\eta\in W^{1,2}=W^{1,2}(\R,\R^n)$ we estimate the difference
\begin{equation*}
\begin{split}
   &\norm{d\Phi|_\xi \eta-d\Phi|_{\tilde \xi} \eta}_{1,2}^2\\
%1
   &=\int_\R \abs{\bigl(d\varphi_s|_{\xi_s}-d\varphi_s|_{\tilde\xi_s}\bigr)\eta_s}^2 ds
   +\int_\R
   \abs{\tfrac{d}{ds}\bigl[\bigl(d\varphi_s|_{\xi_s}-d\varphi_s|_{\tilde\xi_s}\bigr)\eta_s\bigr]}^2 ds
\\
%2
   &\le\int_\R \abs{\bigl(d\varphi_s|_{\xi_s}-d\varphi_s|_{\tilde\xi_s}\bigr)\eta_s}^2 ds
   \\
   &\quad+\int_{-T}^T\abs{\bigl(d\dot\varphi_s|_{\xi_s}-d\dot\varphi_s|_{\tilde\xi_s}\bigr)
   \eta_s}^2ds
   \\
   &\quad+\int_\R\abs{\bigl(
   d^2\varphi_s|_{\xi_s}\dot\xi_s
{\color{gray}\;
   -\;d^2\varphi_s|_{\tilde\xi_s}\dot\xi_s
   +d^2\varphi_s|_{\tilde\xi_s}\dot\xi_s
}
   -d^2\varphi_s|_{\tilde\xi_s}\dot{\tilde\xi}_s\bigr)
   \eta_s}^2ds
   \\
   &\quad+\int_\R\abs{\bigl(d\varphi_s|_{\xi_s}-d\varphi_s|_{\tilde\xi_s}\bigr)
   \dot\eta_s}^2ds
\\
%3
   &\le\int_\R \abs{\bigl(d\varphi_s|_{\xi_s}-d\varphi_s|_{\tilde\xi_s}\bigr)\eta_s}^2 ds
   \\
   &\quad+\int_{-T}^T\abs{\bigl(d\dot\varphi_s|_{\xi_s}-d\dot\varphi_s|_{\tilde\xi_s}\bigr)
   \eta_s}^2ds
   \\
   &\quad+\int_\R 2\abs{\bigl(
   d^2\varphi_s|_{\xi_s}\dot\xi_s
   -d^2\varphi_s|_{\tilde\xi_s}\dot\xi_s
   \bigr)
   \eta_s}^2ds
   \\
   &\quad+\int_\R 2\abs{
   d^2\varphi_s|_{\tilde\xi_s}
   \bigl(\dot\xi_s-\dot{\tilde\xi}_s,\eta_s\bigr)}^2ds
   \\
   &\quad+\int_\R \abs{\bigl(d\varphi_s|_{\xi_s}-d\varphi_s|_{\tilde\xi_s}\bigr)
   \dot\eta_s}^2ds .
\end{split}
\end{equation*}
Term 1.
Since the $W^{1,2}$ norm can be estimated by the $C^0$ norm,
the continuity of $d\varphi$, together with the fact that $\varphi_s$ is
independent of $s$ for $\abs{s}>T$,
implies the following:
For every $\eps>0$ there exists $\delta=\delta(\eps)>0$ such that
$$
   \norm{\tilde\xi-\xi}_{1,2}\le\delta
   \Rightarrow
   \norm{d\varphi_s|_{\xi_s}-d\varphi_s|_{\tilde\xi_s}}_{\Ll(\R^n)}
  \le\eps .
$$
In particular, for every $\tilde\xi$ in the $W^{1,2}$ $\delta$-ball
around $\xi$ it holds that
$$
   \sup_{\norm{\eta}_{1,2}\le 1}
   \int_\R \abs{\bigl(d\varphi_s|_{\xi_s}-d\varphi_s|_{\tilde\xi_s}\bigr)\eta_s}^2 ds
   \le\eps^2 .
$$
Term 2.
Maybe after shrinking $\delta>0$ the continuity (iv) of $d\dot\varphi$
implies (same argument as for Term1) the estimate
$$
   \sup_{\norm{\eta}_{1,2}\le 1}
   \int_{-T}^T \abs{\bigl(d\dot\varphi_s|_{\xi_s}-d\dot\varphi_s|_{\tilde\xi_s}\bigr)\eta_s}^2 ds
   \le\eps^2 .
$$
Term 3.
Maybe after shrinking $\delta>0$ again,
by continuity (iii) of the bi-linear valued map $d^2\varphi$
and by~(v), we have the implication
$$
   \norm{\tilde\xi-\xi}_{1,2}\le\delta
   \Rightarrow
   \norm{d^2\varphi_s|_{\xi_s}-d^2\varphi_s|_{\tilde\xi_s}}_{\Ll(R^n,\R^n;\R^n)}
  \le\eps .
$$
In particular, for every $\tilde\xi$ in the $W^{1,2}$ $\delta$-ball
around $\xi$ it holds that
\begin{equation*}
\begin{split}
   &\int_\R 2\abs{\bigl(
   d^2\varphi_s|_{\xi_s}\dot\xi_s
   -d^2\varphi_s|_{\tilde\xi_s}\dot\xi_s
   \bigr)
   \eta_s}^2ds
\\
   &\le2\eps^2
   \norm{\abs{\dot\xi}\cdot\abs{\eta}}_2^2
\\
   &\le
   2\eps^2\kappa^2\norm{\eta}_{1,2}^2\norm{\xi}_{1,2}^2
\end{split}
\end{equation*}
where in the last inequality we used estimate~(\ref{eq:hghjgyh5688}).
Therefore we get
$$
   \sup_{\norm{\eta}_{1,2}\le 1}   
   \int_\R 2\abs{\bigl(
   d^2\varphi_s|_{\xi_s}
   -d^2\varphi_s|_{\tilde\xi_s}
   \bigr)
   \bigl(\dot\xi_s,\eta_s\bigr)}^2ds
   \le 2\eps^2\kappa^2\norm{\xi}_{1,2}^2 .
$$
Term 4.
By continuity (iii) of $d^2\varphi$ and finite dimension of $\R^n$,
and by~(v) there exists a constant $c>0$, only
depending on $\xi$ but not $\tilde\xi$, such that
$$
   \norm{\tilde\xi-\xi}_{1,2}\le\delta
   \Rightarrow
   \norm{d^2\varphi_s|_{\tilde\xi_s}}_{\Ll(\R^n,\R^n;\R^n)}
  \le c .
$$
Hence we estimate term 4 by
\begin{equation*}
\begin{split}
   &\int_\R 2\abs{
   d^2\varphi_s|_{\tilde\xi_s}
   \bigl(\dot\xi_s-\dot{\tilde\xi}_s,\eta_s\bigr)}^2ds\\
   &\le 2c^2
   \norm{\abs{\dot\xi-\dot{\tilde\xi}}\cdot\abs{\eta}}_2^2\\
   &\le 2c^2\kappa^2 
   \norm{\dot\xi-\dot{\tilde\xi}}_{1,2}^2
   \norm{\eta}_{1,2}^2
\end{split}
\end{equation*}
where in the last inequality we used estimate~(\ref{eq:hghjgyh5688}).
Maybe after shrinking $\delta>0$ we can assume that $\delta<\eps$.
Taking the supremum we obtain the estimate
$$
   \sup_{\norm{\eta}_{1,2}\le 1}
   \int_\R 2\abs{
   d^2\varphi_s|_{\tilde\xi_s}
   \bigl(\dot\xi_s-\dot{\tilde\xi}_s,\eta_s\bigr)}^2ds
   \le 2c^2\kappa^2\eps^2 .
$$
Term 5.
By the argument from term 1 using that
$\norm{\dot\eta}_2^2\le \norm{\eta}_{1,2}^2\le 1$ we get
$$
   \sup_{\norm{\eta}_{1,2}\le 1}
   \int_\R \abs{\bigl(d\varphi_s|_{\xi_s}-d\varphi_s|_{\tilde\xi_s}\bigr)\dot\eta_s}^2 ds
   \le\eps^2 .
$$

\smallskip
The term by term analysis above shows that
for every $\eps>0$ there exists a $\delta>0$ such that
\begin{equation*}
\begin{split}
   \norm{d\Phi|_\xi-d\Phi|_{\tilde \xi} }_{\Ll(W^{1,2})}^2
   &=\sup_{\norm{\eta}_{1,2}\le 1}
   \norm{d\Phi|_\xi \eta-d\Phi|_{\tilde \xi} \eta}_{1,2}^2\\
   &\le \eps^2\left(
   3
   +2\kappa^2 \norm{\xi}_{W^{1,2}(\R,\R^n)}^2
   +2c^2\kappa^2
   \right) .
\end{split}
\end{equation*}
This concludes the proof of Step~3.
\end{proof}

This proves Theorem~\ref{thm:mf-of-paths-finite-model}.
\end{proof}

%\newpage
%%%%%%%%%%%%%%%%%%%%%%%%%%%%%%%%%%%
%%%%%%% Subsection  %%%%%%%%%%%%%%%%%%%
%%%%%%%%%%%%%%%%%%%%%%%%%%%%%%%%%%%
\subsection{Hilbert manifold structure via exponential map}
\label{sec:exp-map}

Let $M$ be a $C^2$ manifold of finite dimension $n$
and let $x_-,x_+\in M$.
We define the space $\Pp_{x_-x_+}$ of $W^{1,2}$-paths
$x\colon\R\to M$ from $x_-$ to $x_+$ and endow it with the structure of
a $C^1$ Hilbert manifold; see~\cite{Eliasson:1967a,schwarz:1993a}.
To define charts we pick a $C^2$ Riemannian metric $g$ on $M$.
Let $\iota_{x_0}>0$ be the injectivity radius at $x_0\in M$ of the
exponential map $\exp\colon T_{x_0} M\to M$.

For convenience of the reader we repeat the construction
of coordinate charts of path space using the exponential map,
but using already finite time basic paths as coordinate centers.
Given a manifold $M$ of finite dimension $n$,
pick a Riemannian metric $g$ on $M$.
The associated exponential map
$$
   \exp=\exp^g\colon TM\to M\times M
   ,\quad
   (q,v)\mapsto (q,\exp_q v)
$$
when defined on a sufficiently small neighborhood of the zero section
is a diffeomorphism onto its image. One defines $\exp_q v:=\gamma_v(1)$
where $\gamma_v\colon[0,1]\to M$ is the unique geodesic
($\Nabla{t}\dot\gamma_v\equiv 0$) with
$\gamma_v(0)=q$ and $\dot\gamma_v(0)=v$.

\boldmath
%%%%%%%%%%%%%%%%%%%%%%%%%%%%%%%%%%%
%%%%%%% Subsubsection  %%%%%%%%%%%%%%%%
%%%%%%%%%%%%%%%%%%%%%%%%%%%%%%%%%%%
\subsubsection{Local parametrizations}\label{sec:fin-loc-par}
\unboldmath

\begin{definition}[Local parametrizations]\label{def:Uu_x}
Let $x$ be a basic path connecting two points $x_\mp\in M$.
Consider the open neighborhood 
$
   \Uu_x
$
of the zero section in the
space of $W^{1,2}$ vector fields $\xi$ along $x$
that consists of all $\xi$ such that $\abs{\xi(s)}_{g}<\iota_{x(s)}$ for
every $s\in\R$.
We define a \textbf{local chart} of $\Pp_{x_-x_+}$ about $x$
as the inverse of the \textbf{local parametrization \boldmath$\Psi_x$}
given by utilizing the exponential  map of $(M,g)$ pointwise for every
$s\in\R$, in symbols
\begin{equation*}
\boxed{
\begin{aligned}
   \Psi_x=\exp_x\colon H^1_x\supset \Uu_x
   &\to \exp_x\Uu_x=:\Vv_x\subset\Pp_{x_-x_+}
\\
   \xi&\mapsto \exp_x\xi:=[s\mapsto \exp_{x(s)}\xi(s)] .
\end{aligned}
}
\end{equation*}
These local parametrizations endow $\Pp_{x_-x_+}$ with the structure of
a \emph{topological} Hilbert manifold modelled on $W^{1,2}(\R,\R^n)$.
\end{definition}
\begin{figure}[h]
  \centering
  \includegraphics[width=10cm]    %[width=0.99\textwidth]
                             %[height=5.5cm]
                             {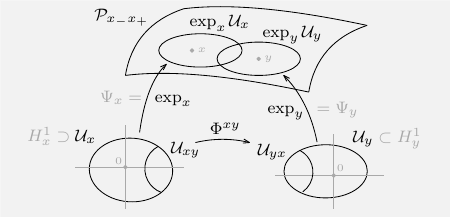}
  \caption{Exponential transition map
                 $\Phi^{xy}\colon\Uu_{xy}\to\Uu_{yx}$,
                 $H^1_x:=W^{1,2}(\SS^1,x^*TM)$}
   \label{fig:fig-chart-findim}
\end{figure}

%%%%%%%%%%%%%%%%%%%%%%%%%%%%%%%%%%%
%%%%%%% Subsubsection  %%%%%%%%%%%%%%%%
%%%%%%%%%%%%%%%%%%%%%%%%%%%%%%%%%%%
\subsubsection{Transition maps and basic trivializations}
\label{sec:fin-transition-maps}

\begin{theorem}[Exponential parametrization transition maps]
\label{thm:mf-of-paths-finite}
Assume that $x$ and $y$ are two basic paths in $M$
connecting $x_-$ to $x_+$.
Consider the open Hilbert space subsets defined by
\begin{equation*}
\begin{split}
   \Uu_{xy}&:={\exp_x}^{-1}\left(\exp_x\Uu_x\CAP \exp_y\Uu_y\right)
   \subset H^1_x ,
\\
   \Uu_{yx}&:={\exp_y}^{-1}\left(\exp_x\Uu_x\CAP \exp_y\Uu_y\right)
   \subset H^1_y .
\end{split}
\end{equation*}
Then the chart transition map
$$
\boxed{
   \Phi^{xy}:={\Psi_y}^{-1}\circ\Psi_x|_{\Uu_{xy}}
   ={\exp_y}^{-1}\circ \exp_x|_{\Uu_{xy}}
   \colon\Uu_{xy}\to \Uu_{yx}
}
$$
is a $C^1$ diffeomorphism.
\end{theorem}

\begin{definition}[Basic trivialization]\label{def:basic-triv}
Suppose that $x$ is a basic path such that
$x(s)=x_-$ for $s\le -T$ and
$x(s)=x_+$ for $s\ge T$.
To define an inner product on $H^1_x$
we choose a trivialization
$$
   \Tt\colon x^*TM\to\R\times\R^n
   ,\quad
   (s,v)\mapsto (s,\Tt_s v)
$$
depending continuously differentiable on $s$, i.e. being of class
$C^1$ in $s$, and with the property that there are linear isomorphisms
$$
   \Tt_\pm\colon T_{x_\pm}M\to\R^n
$$
such that
$$
   \Tt_s= \Tt_\pm
   ,\quad \text{whenever $\pm s\ge T$.}
$$
We refer to $\Tt$ as a \textbf{basic trivialization}.
\end{definition}

\smallskip
Given a basic trivialization $\Tt$, we define the inner product
of $\xi,\eta \in H^1_x$ by
$$
   \INNER{\xi}{\eta}
   :=\INNER{\Tt(\xi)}{\Tt(\eta)}_{W^{1,2}(\R,\R^n)} .
$$
If we choose a different basic trivialization, then we get an
equivalent inner product on $H^1_x$ as the following proposition shows.

\begin{proposition}\label{prop:change-of-triv}
Assume that $\Tt$ and $\tilde\Tt$ are two basic trivializations for a
basic path $x$, then
$$
   \tilde\Tt\circ\Tt^{-1}
   \colon W^{1,2}(\R,\R^n)\to W^{1,2}(\R,\R^n)
$$
is a linear isomorphism.
\end{proposition}

\begin{proof}
By Corollary~\ref{cor:psi} the map $\tilde\Tt\circ\Tt^{-1}$
is a well defined linear map from $W^{1,2}(\R,\R^n)$ to $W^{1,2}(\R,\R^n)$.
By the same corollary the inverse
$$
   (\tilde\Tt\circ\Tt^{-1})^{-1}
   =\Tt\circ\tilde\Tt^{-1}
$$
is as well a well defined linear map from $W^{1,2}(\R,\R^n)$ to
$W^{1,2}(\R,\R^n)$.
Therefore $\tilde\Tt\circ\Tt^{-1}$ is a linear isomorphism.
\end{proof}

\begin{proof}[Proof of Theorem~\ref{thm:mf-of-paths-finite}]
Let $x$ and $y$ be basic paths in the $n$-dimensional manifold $M$.
In order to apply Theorem~\ref{thm:mf-of-paths-finite-model}
to prove Theorem~\ref{thm:mf-of-paths-finite} we use
basic trivializations $\Tt^x$ and $\Tt^y$ of the vector bundles
$x^*TM\to\R$ and $y^*TM\to\R$, respectively.
The \textbf{trivialized transition map}
\begin{equation}\label{eq:Psi}
\boxed{
   \Phi:=\Tt^y\circ\Phi^{xy}\circ(\Tt^x)^{-1}\colon
   W^{1,2}(\R,\R^n)\to W^{1,2}(\R,\R^n)
}
\end{equation}
gives rise to a map $\varphi$ as in
Theorem~\ref{thm:mf-of-paths-finite-model}.
Indeed for $s\in\R$ we define
$$
   \varphi_s:=\varphi(s,\cdot)\colon
   \R^n
\overbrace{
   \stackrel{{\Tt^x_s}^{-1}}{\longrightarrow}
}^{\rm linear}
   T_{x(s)}M
\overbrace{
   \stackrel{\exp_{x(s)}}{\longrightarrow}  U_{x(s)}\cap U_{y(s)}
   \stackrel{\exp_{y(s)^{-1}}}{\longrightarrow}}^{=: \Phi^{xy}_s\in C^2}
%^{\text{of class $C^2$}}
   T_{y(s)}M
\overbrace{
   \stackrel{\Tt^y_s}{\longrightarrow}
}^{\rm linear}
   \R^n .
$$
The map $\varphi\colon\R\times\R^n\to\R^n$
is a composition of the three $C^1$ maps
$\Tt^x\colon x^*TM\to\R\times\R^n$, $\Tt^y$, and $\exp\colon TM
\to M$ on a neighborhood of the zero section. Therefore $\varphi$ is $C^1$.
We verify hypotheses (i--v) in Theorem~\ref{thm:mf-of-paths-finite-model}.
\\
(i)~As illustrated in the displayed diagram the map
$\varphi_s\colon\R^n\to\R^n$ is $C^2$.
\\
(ii)~The $s$-derivative of $\varphi_s:=\varphi (s,\cdot)$ is of the form
$$
   \dot\varphi_s
   =\p_s{\Tt^x_s}^{-1}\circ \underbrace{\Phi^{xy}_s}_{\in C^2}\circ \Tt^y_s
   +{\Tt^x_s}^{-1}\circ \underbrace{\p_s \Phi^{xy}_s}_{\in C^1}\circ \Tt^y_s
   +{\Tt^x_s}^{-1}\circ \underbrace{\Phi^{xy}_s}_{\in C^2}\circ \p_s \Tt^y_s
$$
where each summand starts and ends with a linear map.
Thus $\dot \varphi_s=\p_s\varphi_s\in C^1$.
\\
(iii)~and~(iv) follow from the last two displayed diagrams,
respectively, since the linear maps ${\Tt^x_s}^{-1}$ and $\Tt^y_s$
depend continuously on $s$.
\\
(v)~This follows since $x$ and $y$ are basic paths and $\Tt^x$ and
$\Tt^y$ are basic trivializations and therefore constant whenever
$\abs{s}$ is large enough.

\smallskip
It follows from Theorem~\ref{thm:mf-of-paths-finite-model}
that $\Phi$ defined by~(\ref{eq:Psi}) is $C^1$.
Interchanging the roles of $x$ and $y$ we get that the inverse map
$\Phi^{-1}=\Tt^x\circ\Phi^{yx}\circ(\Tt^y)^{-1}$
is $C^1$ as well.
Hence $\Phi$ is a diffeomorphism from $W^{1,2}(\R,\R^n)$ to
$W^{1,2}(\R,\R^n)$.

\smallskip
This proves Theorem~\ref{thm:mf-of-paths-finite}
in view of Proposition~\ref{prop:change-of-triv}.
\end{proof}

%%%%%%%%%%%%%%%%%%%%%%%%%%%%%%%%%%%
%%%%%%% Subsubsection  %%%%%%%%%%%%%%%%
%%%%%%%%%%%%%%%%%%%%%%%%%%%%%%%%%%%
\subsubsection{Independence of choice of Riemannian metric}
\label{sec:fin-indep-metric}

As a consequence of Theorem~\ref{thm:mf-of-paths-finite}
the space $\Pp_{x_-x_+}$ carries the structure of a $C^1$ Hilbert manifold.
The atlas $\Aa=\Aa(g)$ we constructed depends on the choice
of a $C^2$ Riemannian metric $g$ on $M$.
However, the $C^1$ Hilbert manifold structure of $\Pp_{x_-x_+}$
does not depend on the choice of $C^2$ metric
as the next theorem shows.

\begin{theorem}\label{thm:indep-g}
If $g$ and $\tilde g$ are two $C^2$ Riemannian metrics on a $C^2$
manifold~$M$, then the atlases $\Aa(g)$ and $\Aa(\tilde g)$ are compatible.
\end{theorem}

\begin{figure}[h]
  \centering
  \includegraphics[width=10cm]    %[width=0.99\textwidth]
                             %[height=5.5cm]
                             {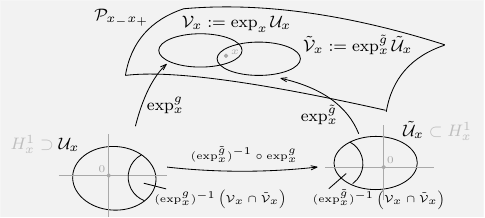}
  \caption{Charts about a basic path $x$ in the atlases $\Aa(g)$ and $\Aa(\tilde g)$}
   \label{fig:chart-indep-metric}
\end{figure}

\begin{proof}
First note that the notions of basic path and basic trivialization do
not depend on the Riemannian metric.
Suppose that $x$ is a basic path.
Let $\Uu_x$ and $\tilde\Uu_x$ be the open neighborhoods
from Definition~\ref{def:Uu_x} of the zero section with respect to $g$
and $\tilde g$, respectively.
Define the open image sets in $\Pp_{x_-x_+}$ by
$$
   \Vv_x:=\exp_x^g\Uu_x
   ,\qquad
   \tilde{\Vv}_x:=\exp_x^{\tilde g}\tilde{\Uu}_x .
$$
By the argument in the proof of Theorem~\ref{thm:mf-of-paths-finite}
it is a $C^1$ diffeomorphism the~map
$$
   (\exp_x^{\tilde g})^{-1}\circ\exp_x^g
   \colon (\exp_x^g)^{-1} \left(\Vv_x\CAP\tilde \Vv_x\right)
   \to (\exp_x^{\tilde g})^{-1} \left(\Vv_x\CAP\tilde \Vv_x\right) .
$$
As the set $C^2_0(\R,\R^n)$ of compactly supported $C^2$ maps is
dense in $W^{1,2}(\R,\R^n)$, basic paths are dense in $\Pp_{x_-x_+}$
and Theorem~\ref{thm:indep-g} follows.
\end{proof}

\newpage  %. \newpage
%%%%%%%%%%%%%%%%%%%%%%%%%
%%%%%%%%% REFERENCES %%%%%%
%%%%%%%%%%%%%%%%%%%%%%%%
%\renewcommand{\bibname}{References}
%\bibliographystyle{plain}
         %   erzeugt:     [1] Joa Weber
%\bibliographystyle{abbrv}
         %  erzeugt:      [1] J. Weber and 
\bibliographystyle{alpha}
         %  article:    [Web05]  J. Weber
         %  book:      [Web05]  Joa Weber
         % more authors: [HZ87]
%%%%%%%%%%%%%%%%%%%%%%%%%
%% include Bibliography in TOC %%
% en.wikibooks.org/wiki/LaTeX/Bibliography_Management#Using_tocbibind
%%%%%%%%%%%%%%%%%%%%%%%%%
% Using hyperref, one should say:
%\cleardoublepage
%\phantomsection
\addcontentsline{toc}{section}{References}
\bibliography{$HOME/Dropbox/0-Libraries+app-data/Bibdesk-BibFiles/library_math,$HOME/Dropbox/0-Libraries+app-data/Bibdesk-BibFiles/library_math_2020,$HOME/Dropbox/0-Libraries+app-data/Bibdesk-BibFiles/library_physics}{}
%$
%%%%%%%%%%%%%%%%%%%%%%%%%
%%%%%%%%% standard %%%%%%%%%
%%%%%%%%%%%%%%%%%%%%%%%%%
%\begin{thebibliography}{00000}
%\small
%\end{thebibliography}

\end{document}